\documentclass[a4paper,12pt]{article}

\usepackage{amsthm,amsmath,stmaryrd,bbm,hyperref,geometry,color,authblk}
\usepackage[utf8]{inputenc}
\usepackage{amssymb}
\usepackage[english]{babel}
\usepackage{graphicx}
\usepackage{amsfonts,amssymb}
\usepackage{verbatim}
\usepackage{enumerate}
\usepackage{subcaption}

\newcommand{\po}{\left(}
\newcommand{\pf}{\right)}
\newcommand{\co}{\left[}
\newcommand{\cf}{\right]}
\newcommand{\cco}{\llbracket}
\newcommand{\ccf}{\rrbracket}
\newcommand{\R}{\mathbb R}

\newcommand{\N}{\mathbb N} 
 
\newcommand{\tE}{\widetilde{\mathcal E}} 
\newcommand{\tF}{\widetilde{\mathcal F}} 
\newcommand{\dd}{\mathrm{d}}
\newcommand{\bX}{\mathbf{X}}
\newcommand{\bY}{\mathbf{Y}}
\newcommand{\bx}{\mathbf{x}}
\newcommand{\na}{\nabla}
\newcommand{\1}{\mathbbm{1}} 

\newcommand{\rev}[1]{#1}

\newtheorem{theorem}{Theorem}
\newtheorem{assumption}{Assumption}
\newtheorem{lemma}[theorem]{Lemma}

\newtheorem{corollary}[theorem]{Corollary}
\newtheorem{proposition}[theorem]{Proposition}
\newtheorem{remark}{Remark}
\newtheorem{example}{Example}

\title{Free energy Wasserstein gradient flow and their particle counterparts: toy model, (degenerate) PL inequalities and exit times}
\author[1,2]{Pierre Monmarché\footnote{pierre.monmarchet@univ-eiffel.fr}}
\affil[1]{LAMA, Université Gustave Eiffel, 77420 Champs-sur-Marne, France.}
\affil[2]{Institut Universitaire de France}

\makeindex

\begin{document}
\maketitle

\begin{abstract}
In finite dimension, the long-time and metastable behavior of a gradient flow perturbated by a small Brownian noise is well understood. A similar situation arises  when a Wasserstein gradient flow over a space of probability measure is approximated by a system of mean-field interacting particles, but classical results do not apply in these infinite-dimensional settings. We are concerned with the situation where the objective function of the optimization problem contains an entropic penalization, so that the particle system is a Langevin diffusion process.

\rev{The work has two parts: the first and main one is devoted to the general study of a very simple class of models, for which the infinite-dimensional behavior is fully characterized by an autonomous diffusion process at fixed finite dimension. In this situation, among other things,  on the one hand, we prove the equivalence between fast convergence for the mean-field gradient flow and for the particle system in terms of functional inequalities (log-Sobolev, Polyak-Łojasiewicz, Talagrand and variants thereof) including the degenerate polynomial case occurring at phase transitions, and on the other hand we characterise the metastable behavior in terms of transition time between local minimiser -- establishing an Eyring-Kramers formula -- and of exit time from unstable critical points. These results on a simple benchmark fix some objectives, conjectures and (counter-)examples for the general situation.  }

\rev{In a second part, motivated by and building on the systematic study of the toy models, we consider the continuous  Curie-Weiss model in a symmetric double-well potential, exactly at the critical temperature. We show that,  although  the $N$-particle Gibbs measure does not satisfy a uniform-in-$N$ standard log-Sobolev inequality (the optimal constant growing like $\sqrt{N}$), it does satisfy a more general Łojasiewicz inequality uniformly in $N$, inducing uniform polynomial long-time convergence rates, propagation of chaos at stationarity and uniformly in time, and creation of chaos. The obtained polynomial order is shown to be sharp.}
 
\end{abstract}

\tableofcontents

\section{Introduction}\label{sec:intro}

\subsection{Overview}

Denoting by $\mathcal P_2(\R^d)$ the set of probability measures over $\R^d$ with finite second moment, consider some energy functional $\mathcal E : \mathcal P_2(\R^d)\rightarrow\R $. The associated free energy at temperature $\sigma^2>0$ is defined by
\begin{equation}
\label{eq:freeenergyF}
\mathcal F(\mu) = \mathcal E(\mu) + \sigma^2 \int_{\R^d} \mu \ln\mu\,,
\end{equation}
for $\mu\in\mathcal P_2(\R^d)$ with a  density (still denoted $\mu$) and finite entropy, and $\mathcal F(\mu)=+\infty$ otherwise. The question of minimizing free energies arises in several situations : first examples came from statistical physics, where some models of interacting particles turned out to solve such an optimization problem \cite{McKean,dawson1986large}. More recently, the development of high-dimensional applications led to the mean-field analysis of finite-dimensional algorithms with exchangeable parameters \cite{Szpruch,chizat2018global,geshkovski2025mathematical}, and also to optimization problems directly set in spaces of probability measures e.g. for variational inference or sampling \cite{lambert2022variational,LelievreLinMonmarche,lambert2022variational}. In these latter situations, the entropic term in~\eqref{eq:freeenergyF} may be either a penalization term to ensure smooth solutions, or a toy model to describes the randomness due to the use of stochastic gradient methods in practice (this amounts to modelling the error of the stochastic gradient estimator as an additive Gaussian variable with constant variance, uniform in space). See \cite{MonmarcheReygner} for further references of applications.

Endowing $\mathcal P_2(\R^d)$ with the associated $L^2$-Wasserstein distance
\[\mathcal W_2(\nu,\mu) = \inf_{\pi\in\Pi} \po \int_{\R^d} |x-y|^2 \pi(\dd x,\dd y)\pf^{1/2} \]
with $\Pi$ the set of couplings of $\nu$ and $\mu$ (i.e. probability measures over $\R^d\times\R ^d$ with marginals $\nu$ and $\mu$), the theory of gradient flows in metric spaces allows to define the gradient flow for the free energy, which turns out to be the parabolic equation
\begin{equation}
\label{eq:EDP}
\partial_t \rho_t = \na\cdot \po \rho_t D \mathcal E(\cdot,\rho_t) \pf + \sigma^2 \Delta \rho_t\,,
\end{equation}
with $D\mathcal E (x,\mu)=\na_x \frac{\delta  \mathcal E}{\delta \mu}(x,\mu)$ the so-called intrinsic derivative of $\mathcal E$, where $\frac{\delta   \mathcal E}{\delta \mu}$ is its Lions derivative (or first variation), defined such that 
\[\forall \nu,\mu\in\mathcal P_2(\R^d),\qquad \mathcal E(\mu) - \mathcal E(\nu) = \int_0^1 \int_{\R^d}\frac{\delta \mathcal E}{\delta \mu} (x,t\mu+(1-t)\nu) (\mu-\nu)(\dd x) \dd t \]
provided these exist (see \cite{ambrosio2005gradient,carmona2018probabilistic} for details).

When the dimension $d$ is large, \eqref{eq:EDP} cannot be solved by deterministic PDE methods such as finite elements or finite difference schemes. In practice, the solution is rather approximated by the empirical distribution of a system of $N$ interacting particles. Alternatively, in statistical physics for instance, the particle system might be the microscopic model of initial interest, and~\eqref{eq:EDP} is its macroscopic limit in the large-population regime. The positions at time $t$ of the particles $\bX_t = (X_t^1,\dots,X_t^N)$ solves 
\begin{equation}
\label{eq:particulesEDS1}
\forall i\in\cco 1,N\ccf,\qquad \dd X_t^i = - D\mathcal E(X_t^i ,\pi_{\bX_t}) \dd t + \sqrt{2}\sigma \dd B_t^i\,,
\end{equation}
where $\mathbf{B}=(B^1,\dots,B^N)$ is a $dN$-dimensional Brownian motion and 
\[\pi_{\bx} = \frac1N\sum_{i=1}^N \delta_{x_i}\]
stands for the empirical distribution of a vector $\bx\in(\R^d)^N$. Alternatively, by introducing the mean-field potential
\[U_N(\bx) = N \mathcal E \po \pi_{\bx} \pf  \]
for $\bx\in(\R^d)^N$, \eqref{eq:particulesEDS1} can equivalently be written as
\begin{equation}
\label{eq:particulesEDS2}
 \dd \bX_t = -\na U_N(\bX_t) \dd t + \sqrt{2}\sigma \dd \mathbf{B}_t\,.
\end{equation}
Under suitable conditions on $\mathcal E$, it is known that the mean field limit of the particle system to the non-linear limit holds, meaning that $\pi_{\bX_t}$ almost surely weakly converges to $\rho_t$ when $\bX$ and $\rho$ respectively solve~\eqref{eq:particulesEDS2} and~\eqref{eq:EDP}. This is equivalent to propagation of chaos, namely to the fact the joint law of $(X_t^1,\dots,X_t^k)$ converges to $\rho_t^{\otimes k}$ for any fixed $k\geqslant 1$ as $N\rightarrow \infty$ for all times $t\geqslant 0$ provided this holds at time $t=0$.

For all $t\geqslant 0$, $\pi_{\bX_t}$ is a random probability measure. Its evolution, ruled by the so-called Dean-Kawasaki equation (see \cite{illien2025dean} for a recent review, and references within), can be seen as a stochastic perturbation of the deterministic gradient flow~\eqref{eq:EDP}, with the random  fluctuations vanishing as $N\rightarrow \infty$ (at rate $N^{-1/2}$, see \cite{FERNANDEZ199733,tanaka1981central}). This is thus similar to the situation of a finite-dimensional gradient flow perturbated by a Gaussian noise
\begin{equation}
\label{gradientepsilon}
\dd Y_t = -\na f (Y_t) \dd t + \sqrt{2\varepsilon} \dd B_t\,,
\end{equation}
in the low temperature regime where $\varepsilon \rightarrow 0$. This classical model has been extensively studied for decades, and is  well understood, in particular thanks to Large Deviation and the Freidlin-Wentzell theory \cite{freidlin1998random}, potential theory \cite{bovier2016metastability} or spectral analysis \cite{helffer2004quantitative}. \rev{In particular, from an optimisation point of view, an important question is the convergence of the gradient flow $\dot x_t = -\na f(x_t)$ to the minimisers of $f$. A convergence of the form $f(x_t) \leqslant e^{-\lambda t}f(x_0)$ for all $t\geqslant 0$ and $x_0\in\R^d$ is equivalent to the following Polyak-Łojasiewicz (PL) inequality:
\begin{equation}
\label{loc:PLfRd}
\forall x\in \R^d,\qquad f(x) - \inf f \leqslant \frac{1}{\lambda}|\na f(x)|^2\,,
\end{equation}
and similarly for $\varepsilon>0$ the exponential convergence at some rate $\lambda_\varepsilon>0$ of the law $\nu_t$ of $Y_t$ towards $\mu_\varepsilon \propto e^{-f/\varepsilon}$ in terms of the relative entropy $\mathcal H(\nu|\mu_\varepsilon) = \int \nu \ln(\nu/\mu_\varepsilon)$ is equivalent to the so-called log-Sobolev inequality (LSI)
\begin{equation}
\label{loc:LSIepsilon}
\forall \nu \in \mathcal P(\R^d),\qquad \mathcal H(\nu|\mu_\varepsilon) \leqslant \frac{1}{\lambda_\varepsilon}\mathcal I(\nu|\mu_\varepsilon) 
\end{equation}
with the Fisher Information $\mathcal I(\nu|\mu_\varepsilon) = \int|\na \ln(\nu/\mu_\varepsilon)|^2\nu$. For the low-temperature behavior of $\lambda_\varepsilon$, see~\cite{MenzSchlichting,ChewiStromme}. When $f$ has several local minima, the long-time behavior of $Y_t$ is metastable, as described by the study of the transition times between local minimizer and of the exit time from unstable saddle points. }

\rev{We are interested in these questions, but in the Wasserstein space. In this situation,} at least for regular interactions,  the convergence from the  particle system to the mean-field limit is well understood (law of large number, fluctuations, large deviations\dots) \emph{over finite time intervals}, see e.g. the reviews \cite{Chaintron1,Chaintron2}. However, for optimization purpose, when running a gradient descent, one is interested in sufficiently long trajectories to achieve convergence to (local) minimizers. Yet, when it comes to the interplay between large $t$ and large $N$,    classical results on noisy perturbation of gradient flows do not apply to the Wasserstein case, and many things remain to be understood. On this topic, the situation which have received much of the attention for now is when the invariant measure of~\eqref{eq:particulesEDS2} satisfies a LSI with a constant uniform in $N$ \rev{(i.e. the analogue of~\eqref{loc:LSIepsilon} with $\lambda_\varepsilon$ independent from $\varepsilon$)}. This is still an active question, see \cite{Dagallier,SongboLSI,MonmarcheLSI} for recent progresses and references within for earlier works (see also references in Section~\ref{sec:motivInegalites}). Nevertheless, this  uniform LSI implies that all critical points of the free energy are global minimizers and that $\mathcal F$ satisfies a global PL inequality \rev{(analogue to~\eqref{loc:PLfRd})} \cite{Pavliotis,GuillinWuZhang}. The perspective of the recent works~\cite{MonmarcheReygner,MonmarcheMetastable} has been to go beyond this situation and to tackle free energies with several local minimizers, corresponding to metastable particle systems. However, many results which are expected to hold in the Wasserstein space by analogy with the finite-dimensional situation~\eqref{gradientepsilon} are yet to be formally established.

\rev{The present work is a contribution in this ongoing program aiming at a better understanding of the interplay between large $t$ and large $N$ in mean-field particle systems and of the long-time behavior of free energy Wasserstein gradient flows. It is constituted of two parts: }
\begin{itemize}
\item \rev{The first part, which is the main one in terms of length, is a systematic study of a class of toy models (the free energies~\eqref{eq:F}). The specificity of these models is that the dynamics~\eqref{eq:particulesEDS2} can be decomposed in two independent parts, the barycentre of the particles which follows an autonomous Langevin equation (at fixed dimension $d$ independent from $N$), and the system seen from its barycentre which is simply an Ornstein-Uhlenbeck process. As a consequence, the analysis essentially boils down to the classical finite-dimensional case. This is similar to the Curie-Weiss model on $\{-1,1\}^N$, where the process is described by the mean magnetization $\frac1N\sum_{i=1}^N x_i$, which for each $N$ is a Markov chain in $[-1,1]$ \cite{JournelLeBris}. Our purpose it thus to translate the finite-dimensional results into the infinite-dimensional settings, to obtain statements in a form which makes sense in general in the Wasserstein space. This gives some insights and examples of the kind of results one may expect to hold in more general situations. We will consider two topics:
\begin{itemize}
\item Functional inequalities (such as PL and LSI) related to the long-time convergence of the mean-field gradient flow~\eqref{eq:EDP} and the particle system~\eqref{eq:particulesEDS2}. We have particularly in mind the conjecture from~\cite{Pavliotis} relating the convergence rates of the two systems, and the degenerate situation observed in~\cite{MonmarcheReygner} where, at a phase transition, the convergence becomes polynomial instead of exponential.
\item Metastable stopping times; both transition times between local minimisers and exit times from unstable saddle  points.
\end{itemize}}
\item \rev{The second part, which is the main one in terms of practical interest of the new results themselves, is concerned with a Curie-Weiss spin model (free energy~\eqref{eq:F_CurieWeiss}). This classical model of statistical physics has been extensively studied, in particular as a representative model of phase transition. Among the various questions discussed in the first part of the work, here we specifically focus on the one about degenerate polynomial rates at the phase transition (on the other topics, non-degenerate functional inequalities are already considered for this model in~\cite{Dagallier}, and some results on metastable exit times are given in~\cite{MonmarcheMetastable}). } 
\end{itemize}
\rev{The relation between the two parts is the following: the insights (and some general results) obtained in the first part inspired the analysis in the Curie-Weiss model. In other words, the second part shows that the ideas developed along the general study of the toy models are useful to address real models of interest, even when the very specific simplification and reduction to finite dimension settings that were true in the toy models do not occur anymore. }

\subsection{Highlights and main contributions}

Let us outline some important high-level points from this work. \rev{First we start with the part on the toy models:}

\begin{enumerate}
\item Proposition~\ref{prop:equivalenceInegal} shows that the conjecture from \cite{Pavliotis} is true for our class of toy models, namely that a PL inequality for the Wasserstein gradient flow is equivalent to a uniform-in-$N$ LSI for the $N$-particle Gibbs measure. Moreover, the constants are sharp. Besides,  in previous works such as \cite{MonmarcheLSI,MonmarcheReygner}, the main step to obtain mean-field PL inequalities has been to establish a so-called free-energy condition (cf. \eqref{eq:ineqEntropy}). Proposition~\ref{prop:equivalenceInegal} also shows that this condition is in fact equivalent to a PL inequality (again, with sharp constants). This indicates that a direction to address  the conjecture from \cite{Pavliotis} in more general situations would be to deduce the uniform LSI from the free energy condition (see \cite{SongboToAppear} on this topic\footnote{\rev{After the release of the present work, an important progress has been made on this topic in~\cite{MonmarchePLLSI}}}).  
\item For the toy models, we show in Proposition~\ref{prop:equivLocale} that, at a  minimizer of the free energy, a local PL inequality  is equivalent to the Hessian of the free energy in the sense of Otto being positive with a spectral gap. However, even in the vicinity of the minimizer, this doesn't require the free energy to be convex along flat interpolations (as discussed in Section~\ref{subsec:misc}, with its notations, this would correspond to $\na^2 V$ being positive, while Proposition~\ref{prop:equivLocale} shows that positiveness of the Hessian is equivalent to $\na^2V + \kappa I$ being positive). This also shows the sharpness of the ``semi-flat-convexity" lower bound required in \cite{MonmarcheLSI}.

\item A function $f\in\mathcal C^1(\R^d)$ is said to satisfy a (local) Łojasiewicz inequality in the vicinity of a minimizer $x_*$ if 
\[ f(x)-f(x_*) \leqslant C \po |\na f(x)|^2\pf^\theta\]
for all $x$ in a neighborhood of $x_*$ for some constants $C>0$ and $\theta>0$ \cite{law1965ensembles} (the inequality is global if it holds for all $x\in\R^d$). The PL inequality corresponds to $\theta=1$ \cite{polyak1964gradient}, and we shall sometimes call \emph{degenerate} the cases where $\theta<1$, or more generally when $|\na f|^2$ is replaced in the PL inequality by $\Theta(|\na f|^2)$ for some non-linear function $\Theta$. Although PL inequalities for Wasserstein gradient flow have been intensively studied, this is less the case of the degenerate case: \cite{BLANCHET20181650} provides a general theoretical study, and  \cite[Proposition 15]{MonmarcheReygner} establishes such an inequality with $\Theta(r) = C(r+r^{1/3})$ for the continuous Curie-Weiss at critical temperature (while PL fails). This degenerate situation can  typically happen at critical parameters for systems undergoing phase transitions  (see Example~\ref{ex:degen}). We  discuss in much more details this situation in Section~\ref{sec:degenPL}, in particular we introduce the variant of the log-Sobolev inequality for the $N$-particle Gibbs measure corresponding to a degenerate Łojasiewicz  inequality and discuss the consequences (in terms of transport inequalities generalizing the T2 Talagrand inequality). As pointed to us by Max Fathi, Łojasiewicz inequalities for the relative entropy with $\theta<1$ follow from modified LSI introduced in  \cite{bobkov2000brunn}. However,  the difference of scaling in $N$ in a mean-field situation between the standard LSI and a suitable Łojasiewicz inequality seems to be a new observation. 

\item Since the study of the $N$-particle Gibbs measure in $\R^{dN}$ involves the study of a  Gibbs measure in dimension $d$ at temperature $\varepsilon=1/N$ (this is true for the toy model~\eqref{eq:F}, but also in more general case using a two-scale approach, as in Section~\ref{sec:CurieWeiss}), we clarify this topic by providing the low-temperature behavior of the LSI constant in~\eqref{loc:LSIepsilon} when the potential is convex with a degenerate minimizer (in the spirit of \cite{ChewiStromme} which covers the non-degenerate situation),  providing an upper and lower bound on the optimal LSI constant respectively in Propositions~\ref{prop:LSIdegen} and \ref{prop:lowerboundCLSIdegenerate}. When the potential scales like $|x-x_*|^\beta$ with $\beta>2$ around the minimizer $x_*$, the order in $N$ of the two bounds matches, showing that the constant scales like $N^{-\frac{2}{\beta}}$. Moreover, we show in Lemma~\ref{lem:LSIdegenerate} a degenerate Łojasiewicz inequality with a better scaling in terms of the temperature (in the sense that, in the mean-field case, it induces an inequality for the Gibbs measure which gives a non-trivial inequality when $N\rightarrow\infty$).
\item Concerning the transition time between local minimizers, for the toy models, we get in Theorem~\ref{thm:Arrhenius} the Arrhenius law (yet to be proven for the continuous Curie-Weiss model since its statement in \cite{dawson1986large}), and even get the more precise Eyring-Kramers formula~\eqref{thm:EK}. It turns out that the latter is consistent with formula involved in other infinite-dimensional situations \cite{barret2015sharp,BerglundGentz,Berglund2} involving functional determinants \cite{gel1960integration,forman1987functional,mckane1995regularization}. Concerning exit times from saddle points, Theorem~\ref{thm:exit} gives the first result we are aware of in mean-field settings. The objective would be next  to establish a similar result for the continuous Curie-Weiss model at low temperature.
\end{enumerate}

\rev{Now, let us describe our new results concerning the  Curie-Weiss model at critical temperature,  which  constitute our main contribution on an application of interest per se. It corresponds to point (3) in the previous list, namely the question of polynomial convergence and degenerate  Łojasiewicz inequalities. We show that the analysis developed on the toy model can be extended to this situation, and get the following in Theorem~\ref{thm:CurieWeiss} and subsequent Corollaries~\ref{Cor:CW} and~\ref{cor:sharp}: 
\begin{enumerate}
\item We show that the free energy satisfies a degenerate Łojasiewicz inequality  with non-linear map $\Theta(r) = C\max(r,r^{2/3})$, improving the rate $1/3$ from~\cite{MonmarcheReygner}, and show that $2/3$ is optimal.
\item We show that the optimal LSI constant of the $N$-particle Gibbs measure $\mu_\infty^N$ is of order $\sqrt{N}$ but that a uniform-in-$N$ degenerate Łojasiewicz inequality 
\[\forall \nu \in \mathcal P(\R^{dN}),\qquad \frac1N \mathcal H(\nu|\mu_\infty^N) \leqslant \Theta\po \frac{1}N \mathcal I(\nu|\mu_\infty^N)\pf\]
holds, with the same sharp $\Theta(r) = C\max(r,r^{2/3})$.
\item We show that these inequalities imply some degenerate Talagrand inequalities,  a polynomial long-time convergence at speed $t^{-2}$ (uniformly in $N$ for the particle system), uniform propagation of chaos and creation of chaos. 
\end{enumerate}
}

\subsection{Organization}

\rev{The rest of this work is organized as follows. Sections~\ref{sec:settings} to~\ref{sec:practical_conclusion} are concerned with the toy models, while the Curie-Weiss model at critical temperature is the topic of Section~\ref{sec:CurieWeiss}. More precisely:}
\begin{itemize}
\item Section~\ref{sec:settings} introduces in details the framework of the study. More precisely, the class of toy models is introduced in Section~\ref{subsec:toy}, followed by a running example in Section~\ref{subsec:runningex}.  More general settings,  notions and notations are given in Section~\ref{subsec:general}. Sections~\ref{subsec:preliminary} and \ref{subsec:hessian} gather specific explicit computations in the case of the toy models.
\item Section~\ref{sec:ineq} is devoted to functional inequalities such as PL and LSI, which quantify the long-time convergence rates of the processes (particle system or mean-field PDE). After some overview in Section~\ref{sec:motivInegalites}, we will study in Section~\ref{subsec:non-degen-contract} these contraction inequalities in the situation where they have been most studied for now, namely when the free energy has a unique critical point, which is a non-degenerate global minimizer. We will also discuss in a similar non-degenerate situation some coercivity inequalities (with the terminology of \cite{SongboToAppear}, corresponding to transport inequalities in the field of geometric measure theory) in Section~\ref{subsec:coerci-non-degen}. We will then consider these two sorts of inequalities in the situation where the minimizer is degenerate, respectively in Section~\ref{sec:degenPL} (for contractivity) and Section~\ref{sec:degenCoer} (for coercivity). Finally, Section~\ref{subsec:local_ineq} discusses local versions of the previous inequalities when the free energy has several critical points and the global inequalities fail.
\item Section~\ref{sec:exittimes} focuses on the study of exit times of the empirical distribution of particles in metastable situations, namely when the free energy has several critical points. We consider transition times between minimizers in Section~\ref{subsec:transitiontimes} and escape time from unstable saddle points in Section~\ref{subsec:exittimes}.
\item In Section~\ref{sec:practical_conclusion}, we show how combining local inequalities as in Section~\ref{subsec:local_ineq} and control of metastable exit times as in Section~\ref{sec:exittimes} yields non-asymptotic fast local convergence estimates for the particle system towards local minimizers of the free energy, upon extremely long time intervals (typically covering all the simulation time in most situations).
\item In contrast to the rest of the work, Section~\ref{sec:CurieWeiss} is not concerned with the toy models introduced in Section~\ref{subsec:toy}, but with the classical continuous Curie-Weiss model in a symmetric double well, at critical temperature. The study builds upon the ideas developed in Section~\ref{sec:degenPL} for degenerate minimizers, thus  illustrating the interest of the insights obtained on the toy model in more general situations, relevant for applications.
\end{itemize}

\subsection*{General notations}

We write respectively  $|u|$ and $u\cdot v$ for the Euclidean norm and scalar product of  $u,v\in \R^d$. The identity matrix of dimension $d$ is denoted by $I_d$. For $\delta \geqslant 0 $, we write $\mathcal B(x,\delta)$ and $\mathcal B(\mu,\delta)$ the closed Euclidean and $\mathcal W_2$ balls of radius $\delta$  with respective center $x\in\R^d$ and $\mu \in\mathcal P_2(\R^d)$. For a measurable $f:\R^d\rightarrow\R_+$, the notation $\mu\propto f$ means that $\mu$ is the probability distribution with Lebesgue density proportional to $f$. The Gaussian distribution on $\R^d$ with mean $m\in\R^d$ and covariance matrix $\Sigma^2$ is written $\mathcal N(m,\Sigma^2)$.

\section{Settings}\label{sec:settings}

\subsection{The general toy model}\label{subsec:toy}

Consider a problem of minimizing  a function $V\in\mathcal C^2(\R^d,\R)$ going to infinity at infinity. We are interested in a mean-field variation where the goal is to minimize  $V(\int_{\R^d} x \mu(x)\dd x)$ over the space of probability densities $\mu$ over $\R^d$, with an additional entropy penalization and, to ensure stability, a quadratic confinement penalization. Writing $m_\mu=\int_{\R^d} x \mu(x)\dd x$, we end up with an objective function
\begin{equation}
\label{eq:F}
\mathcal F(\mu) = V(m_\mu) +  \frac{\kappa}2 \int_{\R^d} |x|^2 \mu(x)\dd x + \sigma^2 \int_{\R^d} \mu(x)\ln(\mu(x))\dd x\,,
\end{equation}
for some parameters $\kappa,\sigma^2>0$. Notice that, without  the quadratic confinement (i.e. taking $\kappa=0$), then this function is not bounded below as Gaussian distributions with a fixed mean and an arbitrarily large variance keeps $V(m_\mu)$ constant while the entropy goes to $-\infty$. This quadratic confinement also appears in more realistic applications, see e.g.   \cite[Equation (21)]{chizat}.

By comparison, a classical model in statistical physics is \rev{the continuous Curie-Weiss spin system with harmonic interaction, corresponding  to}  the free energy
\begin{equation}
\label{eq:F_CurieWeiss}
\mathcal F_0(\mu) = -\frac{\kappa}2 |m_\mu|^2 +  \int_{\R^d} V(x) \mu(x)\dd x + \sigma^2 \int_{\R^d} \mu(x)\ln(\mu(x))\dd x\,,
\end{equation}
with a general $V\in\mathcal C^2(\R^d,\R)$\rev{, see e.g. \cite{dawson1986large,Tugautdoublewell,Dagallier}}. When $\sigma$ is small enough and $V$ has several local minima, $\mathcal F_0$ admits several critical points. In this situation, the long-time  behavior of the corresponding gradient flow and particle system has been investigated  in \cite{MonmarcheReygner,MonmarcheMetastable}. We see that, in \eqref{eq:F}, the non-linear part of the energy may be intricate, but the linear part is given by a simple quadratic confinement, and this is exactly the converse in~\eqref{eq:F_CurieWeiss}.

Introducing $\mathcal E (\mu) =  V(m_\mu) +  \frac{\kappa}2 \int_{\R^d} |x|^2 \mu(x)\dd x$ the energy associated to~\eqref{eq:F}, we get
\begin{equation}
\label{eq:deltaEbmu}
\frac{\delta \mathcal E}{\delta\mu}(x,\mu) = \na V(m_\mu) \cdot x + \frac{\kappa}2 |x|^2 \,,\qquad b_{\mu}(x):= D\mathcal E(x,\mu) = \na V(m_\mu) + \kappa x\,,
\end{equation}
and
\[U_N(\bx) := N \mathcal E\po \pi_{\bx}\pf = N V\po \bar x\pf + \frac{\kappa}{2}|\bx|^2\,. \]
Here and throughout the rest of this article for $\bx\in(\R^d)^N$ we write $\bar x= \frac1N\sum_{i=1}^N x_i$.

Apart from Section~\ref{sec:CurieWeiss}, the guidline of this work is the study of the long-time behavior of the mean-field non-linear PDE~\eqref{eq:EDP} and the particle system~\eqref{eq:particulesEDS2} associated to the free energy~\eqref{eq:F}. As we will soon see, by contrast to the case of~\eqref{eq:F_CurieWeiss}, the temperature $\sigma$ plays no role (there is no phase transition as $\sigma$ varies, see Remark~\ref{rem:phasetransition}) and thus from now on we take $\sigma=1$.

\subsection{A running specific example}\label{subsec:runningex}

Let us describe a mean-field model of variational auto-encoder, of practical interest. We will then simplify it by considering the linear case, in which case the algorithm boils down to Principal Component Analysis (PCA). In this situation the mean-field case is not relevant in practice but it enters our framework. The results obtained in the linear case then provide insights and benchmarks for the more interesting non-linear situation.

Given some data in $\R^n$ distributed according to $\nu_{data} \in \mathcal P(\R^n)$, a variational auto-encoder is given by two functions, $f:\R^n \rightarrow \R^p$ (the encoder) and $g:\R^p\rightarrow \R^n$ (the decoder) with $p<n$, with the goal to minimise
\[\frac12 \int_{\R^n} | \theta - g\circ f(\theta)|^2 \nu_{data}(\dd \theta)\,.\]
That way, $f(\theta)$ is a low-dimensional representation of the data point $\theta$, from which $\theta$ can be approximately reconstructed thanks to $g$. For simplicity in the following we take $p=1$. We consider functions $f$ and $g$ parametrized by one layer of indistinguishable perceptrons, namely
\[f(\theta) = \frac1N \sum_{i=1}^N \psi\po w_i \cdot \theta + b_i \pf\,,\]
where $w_i \in \R^n$ and $b_i \in\R$ are the weights and bias parameters of the $i^{th}$ neuron and $\psi:\R\rightarrow\R$ is a given activation function, and similarly for each component of $g=(g_1,\dots,g_n)$. We take the same number $N$ of neurons for each of these functions.

As mentioned above, we now consider the linear case, namely $\psi(z)=z$. For simplicity we also disregard bias parameters and assume that the data is centered, i.e. $\int_{\R^n} \theta\nu_{data}(\dd \theta)=0$. Then
\[f(\theta) = m_0 \cdot \theta \,,\qquad g(\xi) = \xi m_1\,,\]
where $m_0 \in\R^n$ (resp. the $j^{th}$ component of $m_1\in\R^n$)  is the average of the weight parameter over the $N$ neuron parametrizing $f$ (resp. $g_j$, for $j\in\cco 1,n\ccf$).

As a conclusion, denoting by $\mu$ the probability distribution of $(w_0,w_1) \in\R^n\times\R^n$ (with $w_0$ the weight of a neuron from $f$, and $w_1=(w_{1,1},\dots,w_{1,n})$ the weights for $g_1,\dots,g_n$), the question becomes to minimize $V(m_\mu)$ where
\begin{equation}
\label{eq:VmuEncoder}
V(m) = \frac12 \int_{\R^n} | \theta - (\theta\cdot m_0) m_1 |^2 \nu_{data}(\dd \theta) = c -  m_1^T M m_0 + \frac12  |m_1|^2 m_0^T  M m_0\,,
\end{equation}
where we decomposed $m = (m_0,m_1) \in\R^n\times \R^n$ and  introduced the covariance matrix of the data $M = \int_{\R^n} \theta \theta^T \nu_{data}(\dd \theta)\in\R^{n^2}$ and the constant $c=\frac12  \int_{\R^n} | \theta |^2 \nu_{data}(\dd \theta)$.

\subsection{General considerations and definitions}\label{subsec:general}

Before addressing in details the specific model~\eqref{eq:F}, let us first recall some known facts and useful notations in the general case of a free energy~\eqref{eq:freeenergyF} (with $\sigma=1$). Here we keep the discussion at an informal level, without any precise assumptions on the energy $\mathcal E$ and the probability measures considered, as we won't state any precise statement in this general case, and the well-posedness of the elements introduced below is classical in the specific case of~\eqref{eq:F}.

A critical point of $\mathcal F$ is a fixed-point for the associated gradient flow, i.e. a stationary solution of the PDE~\eqref{eq:EDP}. For a fixed $\mu\in\mathcal P_2(\R^d)$, denote by $\Gamma(\mu)$ the invariant probability measure of the linear Fokker-Planck equation
\begin{equation}
\label{eq:EDP-linear}
\partial_t \rho_t = \na\cdot \po \rho_t D \mathcal E(\cdot,\mu) \pf +  \Delta \rho_t\,.
\end{equation}
We refer to $\Gamma(\mu)$ as the local equilibrium associated to $\mu$. Then,
\[\Gamma(\mu) \propto \exp \po - \frac{\delta \mathcal E}{\delta \mu}(x,\mu)\pf \dd x\,.\]
Moreover, a probability measure $\mu$ is a stationary solution of~\eqref{eq:EDP} if and only if $\mu = \Gamma(\mu)$. In the following we write
\[\mathcal K = \left\{ \mu\in\mathcal P_2(\R^d),\ \mu = \Gamma(\mu)\right\}\]
the set of critical points of $\mathcal F$.

For suitable probability densities $\nu,\mu$ over $\R^d$, write
\[\mathcal H(\nu|\mu) = \int_{\R^d} \ln \po \frac{\nu}{\mu}\pf \dd \nu\,,\qquad \mathcal I(\nu|\mu) = \int_{\R^d} \left|\na \ln \po \frac{\nu}{\mu}\pf\right|^2 \dd \nu\]
the relative entropy and Fisher information of $\nu$ with respect to $\mu$. A measure $\mu$ is said to satisfy a log-Sobolev inequality (LSI) with constant $C>0$ if for all suitable $\nu$,
\begin{equation}
\label{eq:LSI}
\mathcal H(\nu|\mu) \leqslant \frac{C}2\mathcal I(\nu|\mu)\,.
\end{equation}
In that case we write $C_{LS}(\mu) = \sup\{2\mathcal H(\nu|\mu)/\mathcal I(\nu|\mu),\ \nu\ll\mu,\ \nu\neq \mu\}$ the optimal LSI constant.

Along the flow~\eqref{eq:EDP},
\begin{equation}
\label{eq:dissipation}
\partial_t \mathcal F(\rho_t) = - \mathcal I\po \rho_t|\Gamma(\rho_t)\pf\,.
\end{equation}
The free energy $\mathcal F$ is said to satisfy a Polyak-Łojasiewicz (PL) inequality with constant $C>0$ if for all suitable $\mu$,
\begin{equation}
\label{eq:PLF}
\overline{\mathcal F}(\mu) := \mathcal F(\mu) - \inf_{\mathcal P_2(\R^d)} \mathcal F \leqslant \frac{C}2\mathcal I\po\mu|\Gamma(\mu)\pf\,,
\end{equation}
and we write $C_{PL}(\mathcal F) = \sup\{2\overline{\mathcal F}(\mu)/\mathcal I\po \mu| \Gamma(\mu)\pf,\ \mu \neq \Gamma(\mu)\}$ the optimal constant.  
When this is the case, thanks to~\eqref{eq:dissipation}, along~\eqref{eq:EDP},  for all $\rho_0$ and $t\geqslant 0$, 
\[\overline{\mathcal F}(\rho_t) \leqslant e^{- 2t/C} \overline{\mathcal F}(\rho_0)\,.\]
We say that $\mathcal F$ satisfies a uniform local LSI if there exists $C>0$ such that $\Gamma(\mu)$ satisfies a LSI with constant $C$ for all $\mu$; in other words, if for all suitable $\nu,\mu$,
\begin{equation}
\label{eq:locUnifLSI}
\mathcal H\po \nu|\Gamma(\mu)\pf \leqslant \frac{C}{2} \mathcal I\po \nu|\Gamma(\mu)\pf \,.
\end{equation}

For $N \geqslant 1$, the invariant measure of the particle system~\eqref{eq:particulesEDS2} (with $\sigma=1$) is the Gibbs measure
\begin{equation}
\label{eq:GIbbsdef}
\mu_\infty^N(\bx)  \propto \exp\po - U_N(\bx)\pf  \,.
\end{equation}
Denoting by $\rho_t^N$ the law of $\bX_t$, along~\eqref{eq:particulesEDS2},
\[\partial_t \mathcal H\po \rho_t^N |\mu_\infty^N \pf = - \mathcal I\po \rho_t^N |\mu_\infty^N \pf\,.\]
As a consequence, if $\mu_\infty^N$ satisfies a LSI with constant $C>0$ independent from $N$ (in which case we say that $\mu_\infty^N$ satisfies a uniform LSI), then for all $N\geqslant 1$, $t\geqslant 0$ and $\rho_0^N\in\mathcal P_2(\R^{dN})$,
\[\mathcal H\po \rho_t^N |\mu_\infty^N \pf \leqslant e^{-2 t/C} \mathcal H\po \rho_0^N |\mu_\infty^N \pf\,.\]

\subsection{Simple preliminary computations}\label{subsec:preliminary}

We work under the following settings:

\begin{assumption}\label{assu:general}
The free energy  $\mathcal F$ is given by~\eqref{eq:F} with $\sigma=1$, for some $\kappa>0$ and $V\in\mathcal C^2(\R^d,\R)$. Denoting $V_{\kappa}(x) = V(x) + \frac{\kappa}2|x|^2$, we assume that $V_\kappa$ is strongly convex outside a compact set.
\end{assumption}

This ensures in particular that, for all $N\geqslant 1$, the Gibbs measure~\eqref{eq:GIbbsdef} is well-defined and admits finite moments of all order (and even satisfies a log-Sobolev inequality, cf. Section~\ref{subsubsec:inegalfonct} below).

\subsubsection{Critical points}\label{subsec:criticalpoints}

First,  denoting by $\gamma_m$ the  Gaussian distribution on $\R^d$ with mean $m\in\R^d$ and covariance matrix $\kappa^{-1}I_d $, we see that
\begin{align}
 \mathcal F(\mu) &= V_\kappa(m_\mu) +  \frac{\kappa}2 \mathrm{Var}(\mu) +  \int_{\R^d} \mu(x)\ln(\mu(x))\dd x\nonumber\\
 &= V_\kappa(m_{\mu}) +   \mathcal H\po \mu | \gamma_{m_\mu}\pf - \frac{d}{2} \ln\po \frac{2\pi }{\kappa}\pf \,.\label{eq:FVkappa}
\end{align}
In particular, for any $\mu \in\mathcal P_2(\R^d)$,
\[\mathcal F(\mu) \geqslant \mathcal F(\gamma_{m_\mu})\,,\]
and the global minimizers of $\mathcal F$ are exactly the measures $\gamma_{m_*}$ with $m_* \in\R^d$ the global minimizers of $V_\kappa$.

For $\mu \in\mathcal P_2(\R^d)$, using~\eqref{eq:deltaEbmu}, 
\[\Gamma(\mu) = \gamma_{f(m_\mu)}\]
with 
\[f(m) = -\kappa^{-1} \na V(m) \,. \] 
As a consequence, the critical points of $\mathcal F$ are exactly
\[\mathcal K = \{\gamma_m,\ m\in\R^d,\ f(m) = m\} = \{\gamma_m,\ m\in\R^d,\ \na V_{\kappa}(m)=0\}\,.\]
In other words, there is a one-to-one correspondence between critical points of $\mathcal F$ and of $V_\kappa$. Moreover, the critical points are of the same nature: thanks to~\eqref{eq:FVkappa}, $m_*\in\R^d$ is a local minimizer of $V_\kappa$ if and only if $\gamma_{m_*}$ is a  minimizer of $\mathcal F$ over some Wasserstein ball.

\begin{remark}\label{rem:phasetransition}
If  the parameter $\sigma$ had not been fixed to $1$, we would have ended up with the same conclusion, except that the variance of $\gamma_m$ would be $\sigma^2/\kappa$ for $m\in\R^d$. In particular, there is thus no phase transition with the temperature $\sigma^2$ in terms of the number of stationary solutions, which is the number of critical points of $V_\kappa$  independently from $\sigma^2$. This is by contrast with the situation of~\eqref{eq:F_CurieWeiss} and many other classical models. At high temperature $\sigma$, the entropy part of $\mathcal F$ is preponderant and thus minimizers tend to look like Gaussian distributions, but in the toy model~\eqref{eq:F} anyway the critical points are all  Gaussian distributions at all temperatures.
\end{remark}

\begin{example}
\label{ex:criticalpoints}
In the running example~\eqref{eq:VmuEncoder}, the equations for critical points are 
\begin{align}
0 = \na_{m_0} V_{\kappa}(m_0,m_1) &= \kappa m_0  -   M m_1 +  |m_1|^2  M m_0\label{loc:m0=0}\\
0 = \na_{m_1} V_{\kappa}(m_0,m_1) &= \kappa m_1  -   M m_0 +  (m_0^T M m_0) m_1\,.\label{loc:m1=0}
\end{align}
We see that $(0,0)$ is a critical point. Let $(m_0,m_1) $ be a non-zero critical point. If $Mm_0=0$, then~\eqref{loc:m1=0} would give $m_1=0$ and  then~\eqref{loc:m0=0} would give $m_0=0$, so we deduce that $Mm_0 \neq 0$. Moreover,~\eqref{loc:m1=0} shows that $M m_0 \in \mathrm{span}(m_0,m_1)$, and then~\eqref{loc:m0=0} shows that $ \mathrm{span}(m_0,m_1)$ is stable by $M$. If $m_1$ were not colinear to $m_0$, using~\eqref{loc:m0=0} and \eqref{loc:m1=0}, the matrix of $M$ in the basis $(m_0,m_1)$ would be
\[\begin{pmatrix}
0 & \kappa   \\
\kappa + m_0^T Mm_0 & |m_1|^2 \po \kappa + m_0^T Mm_0\pf 
\end{pmatrix}\,,\]
whose determinant is negative (since $M$ is a non-negative symmetric matrix), leading to a contradiction with the fact that $M$ is non-negative. Hence, we deduce that $m_1 = \alpha m_0$ for some $\alpha\in\R$, and then from~\eqref{loc:m1=0} that $m_0$ is an eigenvector of $M$. Denote by $\lambda>0$ the associated eigenvalue. Now, \eqref{loc:m0=0} and \eqref{loc:m1=0} with $m_0\neq 0$ reads
\[\kappa - \alpha \lambda + \alpha^2 \lambda |m_0|^2 = 0 \,,\qquad \kappa \alpha - \lambda + \alpha \lambda |m_0|^2 = 0\,.  \]
Multiplying the second equation by $\alpha$ and substracting the equations show that $\alpha^2=1$. Moreover, the second equation shows that $\alpha>0$, hence $m_0=m_1$, and then $|m_0|^2 = (\lambda-\kappa)/\lambda$, which is possible only if $\lambda>\kappa$. 

Conversely, let $(v,\lambda)$ be an eigenpair of $M$ such that $\lambda>\kappa$ and with $|v|=1$. Then it is straightforward to check that $m^{(\lambda)} := \sqrt{1-\kappa/\lambda}(v,v)$ solves the critical point equations~\eqref{loc:m0=0} and \eqref{loc:m1=0}. We have  thus  identified all critical points of $V_\kappa$. Moreover, for the critical point $m^{(\lambda)}$ associated to an eigenvalue $\lambda>0$, recalling~\eqref{eq:VmuEncoder},
\[V_\kappa(m^{(\lambda)}) = c- \frac{\lambda-\kappa}{2}\po 1+ \frac{\lambda}{\kappa}\pf + \frac{\kappa}\lambda (\lambda-\kappa) = c -  \frac{\lambda-\kappa}{2}\po 1+ \frac{\lambda}{\kappa}- 2\frac{\kappa}{\lambda}\pf \,,\]
while $V_\kappa(0)=c$.  Over $(\kappa,+\infty)$, $\lambda \mapsto V_\kappa(m^{(\lambda)})$ is decreasing and smaller than $V_\kappa(0)$. As a consequence, the global minimum of $V_\kappa$ is attained at  $m^{(\lambda_n)}$ where $\lambda_n$ is the largest eigenvalue of $M$ if $\lambda_n>\kappa$, otherwise it is attained at $0$. See Figure~\ref{fig:Vkappa} for an illustration.
\end{example}

Let us conclude with some computations that will prove useful later on. Thanks to~\eqref{eq:FVkappa}, for any $\mu\in\mathcal P_2(\R^d)$ and $m\in\R^d$,
 \begin{equation}
\label{locgh}
\mathcal F(\mu) - \mathcal F(\gamma_m) =  V_\kappa(m_\mu) - V_\kappa(m) +  \mathcal H\po \mu|\gamma_{m_\mu} \pf  \,. 
 \end{equation}
Besides, for $m'\in\R^d$,
\begin{align}
 \mathcal H\po \mu|\gamma_{m_\mu} \pf &= \mathcal H\po \mu|\gamma_{m'} \pf + \frac{\kappa}2 \int_{\R^d} \po |x-m_\mu|^2 -|x-m'|^2 \pf \mu(x)\dd x  \nonumber \\
&= \mathcal H\po \mu|\gamma_{m'} \pf - \frac\kappa2 |m_\mu - m'|^2  \,.\label{loc1bis}
\end{align}

\begin{figure}[t]
\begin{subfigure}[t]{0.48\textwidth}
    \includegraphics[scale=0.35]{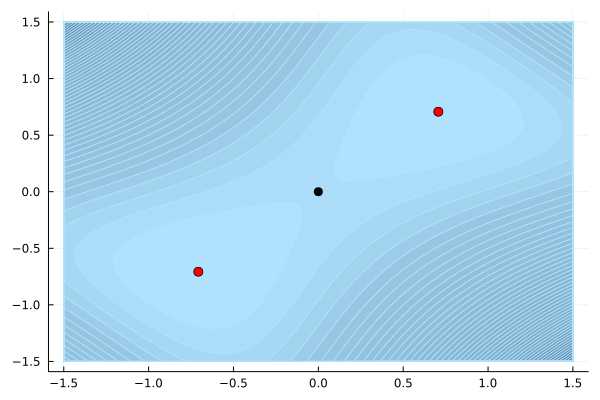}
    \caption{$\kappa=0.5$}
\end{subfigure}
\hfill
\begin{subfigure}[t]{0.48\textwidth}
    \includegraphics[scale=0.35]{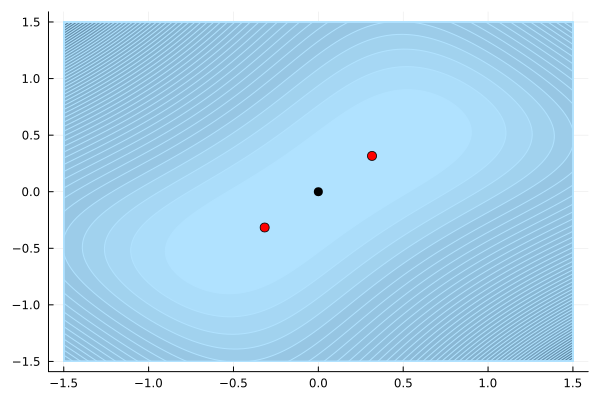}
    \caption{$\kappa=0.9$}
\end{subfigure}
\begin{subfigure}[t]{0.48\textwidth}
    \includegraphics[scale=0.35]{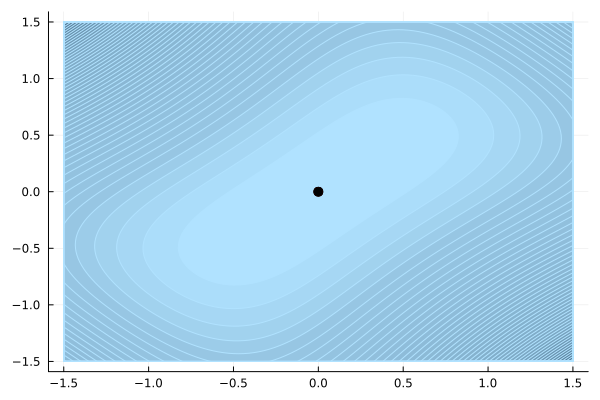}
   \caption{$\kappa=1$}
\end{subfigure}
\hfill
\begin{subfigure}[t]{0.48\textwidth}
    \includegraphics[scale=0.35]{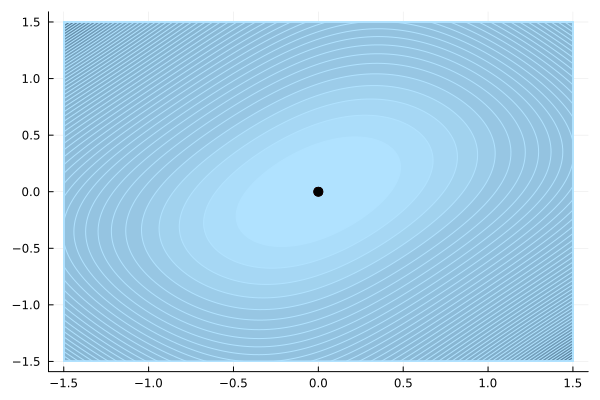}
\caption{$\kappa=2$}
\end{subfigure}
\caption{Level sets of $V_\kappa$ with $V$ from the running example~\eqref{eq:VmuEncoder} plotted for $(m_0,m_1) \in \{(x v,y v),\ x,y\in\R\}$ where $v$ is a unit eigenvector of $M$ associated to the eigenvalue $\lambda=1$, respectively for $\kappa \in \{0.5,0.9,1,2\}$. The red dots are the minimizers $\pm \sqrt{1-\kappa/\lambda}(v,v)$ (for $\kappa<\lambda$), the black dot is $0$. }
\label{fig:Vkappa}
\end{figure}

\subsubsection{barycentre dynamics}\label{sec:barycentre}

The Wasserstein gradient descent for the free energy~\eqref{eq:F} (with $\sigma=1$) reads
\begin{equation}
\label{eq:gradientEDP}
\partial_t \rho_t = \na\cdot \po b_{\rho_t} \rho_t \pf + \Delta \rho_t 
\end{equation}
with $b_\mu$ given by~\eqref{eq:deltaEbmu}.  The associated particle system is 
\begin{equation}
\label{eq:particles}
\forall i\in\cco 1,N\ccf,\qquad \dd X_t^i = - \na V(\bar X_t ) \dd t - \kappa X_t^i \dd t  + \sqrt{2} \dd B_t^i
\end{equation}
with the barycentre $\bar X_t = \frac1N\sum_{i=1}^N X_t^i$. Multiplying~\eqref{eq:gradientEDP} by $x$ and integrating over $\R^d$ shows that
\begin{equation}
\label{eq:flot_mt}
\partial_t m_{\rho_t } = - \na V_\kappa(m_{\rho_t})\,.
\end{equation} 
Similarly, summing~\eqref{eq:particles} over $i\in\cco 1,N\ccf$ and dividing by $N$ gives  
\begin{equation}
\label{eq:EDSbarycentre}
 \dd \bar X_t = - \na V_\kappa (\bar X_t ) \dd t  +\frac{\sqrt{2}}{\sqrt N}\dd \tilde B_t\,,
\end{equation}
with $\tilde B= \frac{1}{\sqrt{N}} \sum_{i=1}^N B_t^i$ a brownian motion over $\R^d$. The fact that the barycentre solves a closed SDE, and is thus an autonomous Markov process (of the form~\eqref{gradientepsilon}), is reminiscent to the situation in the discrete Curie-Weiss model \cite{JournelLeBris}. It is the   reason why the analysis is much simpler in this case than in other situations, for instance~\eqref{eq:F_CurieWeiss} considered in \cite{MonmarcheMetastable}.

\rev{
\begin{remark}\label{rem:ajout}
To see that the fact that the barycentre satisfies an autonomous equation~\eqref{eq:EDSbarycentre} is very specific to the toy models~\eqref{eq:F}, compare with the Curie-Weiss model~\eqref{eq:F_CurieWeiss}: in the latter,
\[ \dd \bar X_t = \kappa \bar X_t \dd t  - \frac1N\sum_{i=1}^N \na V(X_t^i) \dd t  +\frac{\sqrt{2}}{\sqrt N}\dd \tilde B_t\,.\]
If $V$ is not quadratic, the second term cannot be expressed in terms of $\bar X_t$.
\end{remark}
}

\subsubsection{Orthogonal dynamics}\label{subsec:orthogonDyn}

Substracting~\eqref{eq:EDSbarycentre} to \eqref{eq:particles} gives  
\begin{equation}
\label{eq:particles-centered}
\dd (X_t^i - \bar X_t) =  - \kappa \po X_t^i - \bar X_t\pf \dd t   + \sqrt{2}  \po \dd B_t^i - \frac{1}{\sqrt N}\dd \tilde B_t\pf \,.
\end{equation}
In other words, considering $P$ a $(d-1)N\times dN$ matrix whose kernel is exactly the orthogonal of  $E_0:=\{x\in(\R^d)^N,\ \frac1N\sum_{i=1}^N x_i = 0\}$ and such that $P P^T$ is the identity of $\R^{(d-1)N}$ and $P^T P$  is the orthogonal projection on $E_0$ (this amounts to say that $P\bx$ is the coordinates of the orthogonal projection of $\bx$ on $E_0$ in an orthonormal basis of $E_0$) and setting $Y_t = P \bX_t$, then
\begin{equation}
\label{eq:OU}
\dd Y_t = - \kappa Y_t \dd t + \sqrt{2}  \dd W_t\,,
\end{equation}
with $W = P \mathbf{B}$ a $(d-1)N$-dimensional brownian motion independent from $\tilde B$. 

Similarly to~\eqref{eq:particles-centered}, considering the centered density $\rho_t^c(x) = \rho_t(x+ m_{\rho_t})$ with $\rho_t$ solving~\eqref{eq:gradientEDP}, thanks to~\eqref{eq:flot_mt}, we see that
\begin{align*}
\partial_t \rho_t^c(x) 
&= \kappa \na\cdot \po x \rho_t^c(x) \pf +    \Delta \rho_t^c(x)\,,
\end{align*}
which is the Fokker-Planck equation associated to the Ornstein-Uhlenbeck process~\eqref{eq:OU} (here in dimension $d$). In other words, $\rho_t^c$ is the law of $e^{-\kappa t} X_0 + \sqrt{1-e^{-2\kappa t}} G_0$ where $X_0 \sim \rho_0^c$ and $G_0 \sim \gamma_0$.

\subsubsection{Equilibria and log-Sobolev inequalities}\label{subsubsec:inegalfonct}

For any $\mu\in\mathcal P_2(\R^d)$,  as a Gaussian distribution with variance $\kappa^{-1} I_d$, $\Gamma(\mu)$ satisfies a LSI with constant $\kappa$ (see \cite[Proposition 5.5.1]{BakryGentilLedoux}), so that the uniform local LSI~\eqref{eq:locUnifLSI} holds.

 We write $Q$ the $dN\times dN$ orthogonal matrix such that
\[(\sqrt{N} \bar X_t,Y_t) = Q \bX_t\,, \]
with $Y_t$ as in Section~\ref{subsec:orthogonDyn}. The invariant measure of $(\bar X_t,Y_t)$ is 
\begin{equation}
\label{eq:decomposeorthonu}
   \nu_\infty^N \otimes \mathcal N \po 0, \kappa^{-1} I_{(d-1)N}\pf 
\end{equation}
with $\nu_\infty^N$ the invariant measure of~\eqref{eq:EDSbarycentre}, i.e. the Gibbs measure
\begin{equation}
\label{eq:nuinfinityN}
\nu_\infty^N  \propto \exp\po - N V_\kappa(z)\pf \dd z\,.
\end{equation}
By classical properties on LSI, the following holds: 



\begin{lemma}
\label{lem:LSI}
Under Assumption~\ref{assu:general}, for all $N\geqslant 1$, the Gibbs measure $\mu_\infty^N$ defined in~\eqref{eq:GIbbsdef} satisfies a LSI if and only if $\nu_\infty^N$ satisfies a LSI, and in that case
\[C_{LS}(\mu_\infty^N) \  = \  \max\po N C_{LS}(\nu_\infty^N),\kappa^{-1}\pf\,.\]
\end{lemma}
In particular, a uniform LSI for $\mu_\infty^N$ is equivalent to the fact $\nu_\infty^N$ satisfies a LSI with constant of order $1/N$.

\begin{proof}
By \cite[Proposition 5.2.7]{BakryGentilLedoux}, if $\mu$ and $\nu$ are two probability measures, then $\mu\otimes \nu$ satisfies a LSI if and only if $\mu$ and $\nu$ both satisfies a LSI, and moreover $C_{LS}(\mu\otimes \nu)=\max(C_{LS}(\mu),C_{LS}(\nu))$ (strictly speaking, \cite[Proposition 5.2.7]{BakryGentilLedoux} only states the ``if" part, but the ``only if" is trivial by applying the LSI of $\nu\otimes \mu$ to measures of the form $\nu'\otimes \mu$ and $\nu\otimes \mu'$).

By \cite[Proposition 5.4.3]{BakryGentilLedoux}, if $\mu$ satisfies a LSI with constant $C$ then its image by a $L$-Lipschitz map satisfies a LSI with constant $CL^2$.

By \cite[Proposition 5.5.1]{BakryGentilLedoux} (and the discussion after for the optimality of the constant), the standard Gaussian distribution $\mathcal N(0,1)$ satisfies a LSI with optimal constant $C=1$.

Conclusion follows by combining these properties with the fact that $\mu_\infty^N$ is the image by an orthogonal transformation of the invariant measure of $(\sqrt{N} \bar X_t,Y_t)$.
\end{proof}

\begin{remark}
\rev{In the spirit of Remark~\ref{rem:ajout}, let us emphasize that,} in the study of the LSI for $\mu_\infty^N$, the fact that, up to an orthogonal transformation, the Gibbs measure can be exactly decomposed as the tensor product of a strongly log-concave measure and a possibly more complicated but $d$-dimensional measure is quite specific to the particular model~\eqref{eq:F}. However, in more realistic situations such as~\eqref{eq:F_CurieWeiss}, this situation is somewhat generalized thanks to the two-scale method used e.g. in \cite{Dagallier,MonmarcheMetastable}, references within, and Section~\ref{sec:CurieWeiss}.
\end{remark}

\subsubsection{Miscellaneous remarks}\label{subsec:misc}

\noindent\emph{(Stable families)} There are two interesting families of probability measures which are preserved by the non-linear flow~\eqref{eq:gradientEDP}:
\begin{itemize}
\item \emph{Gaussian distributions.} If $\rho_0 = \mathcal N(m_0,\Sigma_0^2)$ for some postivie symmetric $d\times d$ matrix $\Sigma_0^2$  and $m_0\in\R^d$, then $\rho_t = \mathcal N(m_t,\Sigma_t^2)$ where $m_t = m_{\rho_t}$ solves~\eqref{eq:flot_mt} and $\Sigma_t^2$ solves
\[\partial_t \Sigma_t^2 = -2\kappa \Sigma_t^2 + 2 I_d\,, \]
so that $\Sigma_t^2 = e^{-2\kappa t} \Sigma_0^2 + \kappa^{-1}(1-e^{-2\kappa t}) I_d$. 
\item \emph{Tensor products.} If $\rho_0 = \bigotimes_{i=1}^d \rho_{0,i}$ where $\rho_{0,i}$ is a probability distribution on $\R$ (more generally $\rho_0$ could be a tensor product of multi-dimensional probability distributions) then $\rho_t = \bigotimes_{i=1}^d \rho_{t,i}$ for all $t\geqslant 0$ where, for $i\in\cco 1,d\ccf$, $\rho_{t,i}$ solves
\[\partial_t \rho_{t,i}(x_i)  = \na_{x_i}\cdot \co \po \na_{x_i} V(m_{\rho_t}) + \kappa x_i\pf \rho_{t,i}(x_i)\cf  + \Delta_{x_i} \rho_{i,t}(x_i)\,.\] 
For instance, in the running example from Section~\ref{subsec:runningex}, if the parameters of the encoder and the decoder are initially independent, they remain independent along the flow~\eqref{eq:gradientEDP}.
\end{itemize}

However, in general, this is no longer true for the particle system~\eqref{eq:particles}. When $V$ is not a quadratic function, even if $\bX_0$ is a Gaussian vector, $(\bX_t)_{t\geqslant 0}$ is not a Gaussian process.

\medskip

\noindent\emph{(Flat convexity)} The energy is said to be convex along linear interpolations, or flat-convex, if for all $\nu,\mu\in\mathcal P_2(\R^d)$ and $t\in[0,1]$,
\[\mathcal E\po t\nu + (1-t)\mu) \pf \leqslant t \mathcal E(\nu) + (1-t) \mathcal E(\mu)\,.\]
It is straightforward to check that $\mathcal E$ is flat-convex if and only if $V$ is convex. More generally,  provided $\|(\na^2 V)_{-}\|_\infty := \sup\{ - u\cdot \na^2 V(x) u, u,x\in\R^d,|u|=1\}$ is finite, using that $V(x) + \|(\na^2 V)_{-}\|_\infty |x|^2/2$ is convex, we get that  for all  $\nu,\mu\in\mathcal P_2(\R^d)$ and $t\in[0,1]$,
\begin{equation}\label{loc:convexitt}
\mathcal E\po t\nu + (1-t)\mu) \pf \leqslant  t \mathcal E(\nu) + (1-t) \mathcal E(\mu) + \frac12t (1-t) \|(\na^2 V)_{-}\|_\infty |m_\nu-m_\mu|^2\,.
\end{equation}
This kind of conditions have been used in \cite{MonmarcheReygner,MonmarcheLSI} to prove LSI beyond the flat-convex case.

\begin{example}
\label{ex:hessienne}
In the running example~\eqref{eq:VmuEncoder}, for $m=(m_0,m_1)$, we compute
\[\na^2 V(m) = \begin{pmatrix}
|m_1|^2 M & -M+2 (M m_0)m_1^T \\
 -M+2 m_1(M m_0)^T & (m_0^T M m_0) I_d
\end{pmatrix}\,.\]
In particular, except for the trivial case where $M=0$, $V$ is never convex since, considering an eigenpair $(v,\lambda)$ of $M$ with  $\lambda>0$, then
\begin{equation}
\label{loc:vMV}
\begin{pmatrix}
v \\ v 
\end{pmatrix}^T \na^2 V(0) \begin{pmatrix}
v \\ v 
\end{pmatrix} = - 2 \lambda|v|^2<0\,.
\end{equation}
\end{example}

\subsection{Hessian of the free energy}\label{subsec:hessian}

As it will give some insights and interpretations of the results in Sections~\ref{sec:ineq} and \ref{sec:exittimes}, let us determined  the Hessian of the free energy~\eqref{eq:F}, in the sense of Otto calculus \cite{otto2001geometry,OttoVillani} (we will  not use it in rigorous proofs, and thus we keep computations formal and functional settings vague). In this formalism, the tangent space of $\mathcal P_2(\R^d)$ at some point $\mu$ is understood ad the set of functions $\Phi:\R^d\rightarrow \R$ defined up to an additive constant, endowed with the scalar product $\langle \Phi_1,\Phi_2\rangle_{\mu} = \int_{\R^d} \na \Phi_1 \cdot \na \Phi_2 \dd \mu$. The Hessian of $\mathcal F$ at $\mu$ is given as the quadratic form on the tangent space
\[\langle \mathrm{Hess}\mathcal F_{|\mu}  \Phi,\Phi\rangle_{\mu} = \po \frac{\dd}{\dd t}\pf^2 \mathcal F(\mu_t) \,,\]
when
\begin{equation}
\label{loc:mutOtoo}
\partial_t \mu_t + \na\cdot \po \mu_t \na \Phi_t \pf =0\,,\qquad \partial_t \Phi_t + \frac12|\na \Phi_t|^2 =0\,,
\end{equation}
with $\mu_0=\mu$ and $\Phi_0=\Phi$. 
\begin{lemma}\label{lem:hessian}
For $\mathcal F$ given by~\eqref{eq:F} with $\sigma^2=1$, for suitable $\mu \in\mathcal P_2(\R^d)$ and $\Phi:\R^d\rightarrow\R$,
 \begin{equation}
 \label{eq:HessianF}
\langle \mathrm{Hess}\mathcal F_{|\mu}  \Phi,\Phi\rangle_{\mu} = \int_{\R^d} \co \| \na^2 \Phi\|_{F}^2 + \kappa |\na \Phi|^2\cf \dd  \mu + \mu(\na\Phi)^T \na^2 V(m_\mu) \mu(\na \Phi)\,,
\end{equation}  
with $\|\cdot\|_F$ the Frobenius norm and $\mu(\na \Phi) = \int_{\R^d} \na \Phi \dd \mu$.
\end{lemma}
\begin{proof}
As computed in \cite{OttoVillani}, for any   $\mu\in\mathcal P_2(\R^d)$ and $\rho_* \propto e^{-U}$,
\[\langle \mathrm{Hess}\mathcal H(\cdot|\rho_*)_{|\mu}  \Phi,\Phi\rangle_{\mu} = \int_{\R^d} \co \| \na^2 \Phi\|_{F}^2 + \na \Phi\cdot \na^2 U \na \Phi \cf \dd  \mu\,. \]
When $\rho_* = \gamma_{m_*}$ for some $m_*\in\R^d$, this gives
\[\langle \mathrm{Hess}\mathcal H(\cdot|\rho_*)_{|\mu}  \Phi,\Phi\rangle_{\mu} = \int_{\R^d} \co \| \na^2 \Phi\|_{F}^2 + \kappa |\na \Phi|^2\cf \dd \mu\,. \]
 From \eqref{locgh} and \eqref{loc1bis},
\[\mathcal F(\mu) - \mathcal F(\rho_*) = W(m_\mu)  + \mathcal H(\mu|\rho_*)\]
with $W(m)=V(m)-V(m_*)+ \kappa (m-m_*)\cdot m_*$, and thus  we have to compute $\partial_t^2 W(m_{\mu_t})$ at $t=0$ along~\eqref{loc:mutOtoo}. We find
\[\partial_t m_{\mu_t} =   \mu_t(\na \Phi_t)  \]
and, for any $j\in\cco 1,d\ccf$,
\[\partial_t \int_{\R^d} \partial_j \Phi_t \dd \mu_t =  \int_{\R^d} \na\Phi_\cdot   \na \partial_j \Phi_t \dd  \mu_t - \frac12 
 \int_{\R^d}  \partial_{j} \po |\na \Phi_t|^2\pf  \mu_t = 0\,,\]
 hence $\partial_t^2 m_{\mu_t}=0$. As a consequence, we end up with~\eqref{eq:HessianF}.

\end{proof}

The orthogonal projection of $\Phi$ on the set $E_\ell$ of linear forms in the tangent space of $\mathcal P_2(\R^d)$ at $\mu$ is given by
\[\Pi_{\ell}\Phi (x) = \mu(\na \Phi)\cdot x\,,\]
since this ensure that for any linear $\Psi_u(x) = u\cdot x$ with $u\in\R^d$,
\[\langle \Psi_u,\Phi-\Pi_\ell \Phi \rangle_{\mu} = u \cdot \co \mu(\na \Phi) - \mu(\na \Phi)\cf =0\,. \]
We see that $E_\ell^{\perp}$ is the set of functions $\Phi$ with $\mu(\na \Phi)=0$. In particular, for $\Phi\in E_\ell^\perp$,
\begin{equation}
\label{eq:HessCOnvexEortho}
\langle \mathrm{Hess}\mathcal F_{|\mu}  \Phi,\Phi\rangle_{\mu} = \int_{\R^d} \co \| \na^2 \Phi\|_{F}^2 + \kappa |\na \Phi|^2\cf \mu  \geqslant \kappa\|\Phi\|_\mu^2 \,,
\end{equation}
so that $\mathrm{Hess}\mathcal F_{|\mu} $ is positive with a spectral gap larger than $\kappa$ on $E_\ell^\perp$. On the other hand, for $u\in\R^d$,
\begin{equation}
\label{eq:HessEl}
\langle \mathrm{Hess}\mathcal F_{|\mu}  \Psi_u,\Psi_u\rangle_{\mu} = u\cdot \na^2 V_\kappa(m_\mu) u \,.
\end{equation}
Moreover, $E_{\ell}$ and its orthogonal are stable by $\mathrm{Hess}\mathcal F$, since
\begin{equation}
\label{eq:Hessorthodecompos}
\langle \mathrm{Hess}\mathcal F_{|\mu}  \Phi,\Phi\rangle_{\mu} = \langle \mathrm{Hess}\mathcal F_{|\mu}  (\Phi-\Pi_\ell\Phi),\Phi-\Pi_\ell\Phi\rangle_{\mu} + \langle \mathrm{Hess}\mathcal F_{|\mu} \Pi_\ell \Phi,\Pi_\ell \Phi\rangle_{\mu} \,.
\end{equation}

\begin{remark}
\label{rem:convexité}
As a consequence, the Hessian of $\mathcal F$ is positive everywhere if and only if $\na^2 V_\kappa = \na^2 V + \kappa I$ is positive everywhere. This should be compared to flat-convexity of the energy $\mathcal E$ which, as discussed in Section~\ref{subsec:misc}, is equivalent to $\na^2 V$ being positive everywhere.

More locally, at a critical point $\rho_*=\gamma_{m_*}$ of the free energy,  $\mathrm{Hess}\mathcal F$ is definite positive iff $\na^2 V_\kappa(m_*)$ is  definite positive, i.e. if $m_*$ is a non-degenerate local minimizer of $V_\kappa$.
\end{remark}

For $m \in \R^d$, let $L_m = \Delta + \na \ln \gamma_m \cdot \na = \Delta - \kappa (x-m) \cdot\na   $, which is such that, for all nice functions $f:\R^d\rightarrow\R$,  $\int_{\R^d} L_m f \gamma_m = 0$. Then, for nice $\Phi$, 
\begin{eqnarray*}
\langle L_m \Phi,\Phi \rangle_{\gamma_m} &= & \sum_{i=1}^d \int_{\R^d} \partial_{x_i} L_m \Phi \partial_{x_i} \Phi \gamma_m\\
&=& \sum_{i=1}^d \int_{\R^d} \po  \po L_m \partial_{x_i} + [\partial_{x_i},L_m]\pf   \Phi \partial_{x_i} \Phi - \frac12 L_m\po |\partial_{x_i} \Phi|^2\pf^2 \pf \gamma_m\,,
\end{eqnarray*}
with the commutator
\[[\partial_{x_i},L_m]   = \partial_{x_i} L_m -  L_m\partial_{x_i}  = \na \partial_{x_i} \ln \gamma_m \cdot \na  = - \kappa \partial_{x_i}\,. \]
Hence, using that $\frac12 L_m(f^2) - f L_m f = |\na f|^2$ for nice $f$,
\[
\langle L_m \Phi,\Phi \rangle_{\gamma_m} = - \sum_{i=1}^d \int_{\R^d}  \po  \kappa ( \partial_{x_i} \Phi)^2 +  |\na \partial_{x_i} \Phi|^2 \pf \gamma_m= - \int_{\R^d}  \po  \kappa |\na \Phi|^2 +  \|\Phi\|_F^2 \pf \gamma_m\,.
\]
This is a very classical computation, see e.g. \cite[Section 1.16.1]{BakryGentilLedoux}. As a consequence, we get the following form for $\mathrm{Hess}\mathcal F$ at points $\mu=\gamma_m$ (and thus in particular in all critical points of $\mathcal F$):

\begin{lemma}\label{lem:hessian_gammam}
For any $m\in\R^d$,
\begin{equation}
\label{eq:HessFm}
\mathrm{Hess}\mathcal F_{|\gamma_m}\Phi(x) = - L_m  \Phi(x) + x\cdot \na^2 V(m)  \mu(\na \Phi)\,.
\end{equation}
\end{lemma}

\begin{example}
\label{ex:hessienne2}
Let us proceed with the computations started in Example~\ref{ex:hessienne}. More generally than~\eqref{loc:vMV}, for any $u,v\in\R^n$,
\begin{equation}
\label{locuvvu}
\begin{pmatrix}
u \\ v 
\end{pmatrix}^T \na^2 V(0) \begin{pmatrix}
u \\ v 
\end{pmatrix} = - 2 u^T M v \geqslant -\lambda_n \po |u|^2 + |v|^2\pf \,,
\end{equation}
with $\lambda_n$ the largest eigenvalue of $M$.  This shows that $\na^2 V_\kappa(0)$ is positive definite when $\kappa >\lambda_n$ (which is the case where $0$ is the unique critical point of $V_\kappa$, as determined in Example~\ref{ex:criticalpoints}). Conversely, if $\kappa<\lambda_n$, applying~\eqref{loc:vMV} with $\lambda=\lambda_n$ shows that $\na^2 V_\kappa(0)$ has a negative eigenvalue, so that $0$ is a saddle point of $V_\kappa$.

Alternatively, if $(v,\lambda)$ is an eigenpair of $M$ with $\lambda>0$, $|v|=1$ and $m^{(\lambda)}=(1-\kappa/\lambda)(v,v)$,
\[\na^2 V(m^{(\lambda)}) = \begin{pmatrix}
(1-\kappa/\lambda) M & -M+2(\lambda-\kappa) vv^T \\
-M+2(\lambda-\kappa) vv^T & (\lambda-\kappa) I_d
\end{pmatrix}\,.\]
For any eigenpair $(u,\eta )$ of $M$ with $|u|=1$, $\mathrm{span}\{(u,0),(0,u)\}$ is stable by $\na^2 V(m^{(\lambda)})$, and in this basis $\na^2 V(m^{(\lambda)})$ reads
\begin{equation}
\label{locvlambdakappa0}
\begin{pmatrix}
(1-\kappa/\lambda)\eta  & -\eta \\
-\eta & \lambda-\kappa
\end{pmatrix}\,.
\end{equation}
if $u\notin \{-v,v\}$ (in which case $u\cdot v=0$) and
\begin{equation}
\label{locvlambdakappa1}
\begin{pmatrix}
\lambda-\kappa  & -\lambda + 2(\lambda-\kappa) \\
-\lambda + 2(\lambda-\kappa) & \lambda-\kappa
\end{pmatrix}\,.
\end{equation}
if $u \in \{-v,v\}$. In particular, $\na^2 V_\kappa(m^{(\lambda)}) = \na^2 V(m^{(\lambda)}) + \kappa I_{2n} $ is definite positive iff its restriction on all of these $2d$ planes is definite positive. In the case where $u\in\{-v,v\}$, adding $\kappa$ on the diagonal of~\eqref{locvlambdakappa1}, we get a symmetric matrix with positive trace and with determinant $\lambda^2 - (\lambda-2\kappa)^2$, which is positive when $\kappa < \lambda$. In the other case, adding $\kappa$ on the diagonal of~\eqref{locvlambdakappa0} gives again a positive trace (when $\kappa<\lambda$) and a determinant $(\lambda-\kappa) \eta + \kappa \lambda - \eta^2= (\lambda-\eta)(\eta + \kappa)$, which has the same sign as $\lambda-\eta$. As a conclusion, among the critical points, the only local minimizers are those associated to the largest eigenvalue $\lambda_n$ (which are the global minimizers of $V_\kappa$, as seen in Example~\ref{ex:criticalpoints}), and moreover they are non-degenerate iff $\lambda_n$ has multiplicity $1$.

\end{example}

\section{Functional inequalities}\label{sec:ineq}

\subsection{Motivations and overview}\label{sec:motivInegalites}

Since they describe the convergence rates of  Wasserstein gradient flows and their particle counterpart, the PL inequality~\eqref{eq:PLF} for $\mathcal F$ and  uniform LSI~\eqref{eq:LSI} for $\mu_\infty^N$ have been studied in many works in various situations, see \cite{CMCV,chizat,Pavliotis,GuillinWuZhang,Dagallier,SongboLSI,MonmarcheLSI,MonmarcheReygner,SongboToAppear} and references within. In general, a uniform LSI implies a PL inequality for $\mathcal F$, cf. \cite{Pavliotis,MonmarcheMetastable}. An open question, raised in particular in \cite{Pavliotis}, is whether the converse implication is true  (as we will see in Proposition~\ref{prop:equivalenceInegal}, we will answer positively to this in the case~\eqref{eq:F}). Establishing uniform LSI is in general a difficult question. Recently, motivated in particular by applications in optimization and machine learning, several works have addressed the case where the energy $\mathcal E$ is flat-convex (as discussed in Section~\ref{subsec:misc}, this corresponds in our situation to the case where $V$ is convex) \cite{chizat,SongboLSI,CheNitZha24LSI,Songbo}. In this situation, $\mathcal F$ admits a unique global minimizer $\rho_*$, and the following ``entropy sandwich inequality" have been established in \cite{chizat}: for all $\mu \in \mathcal P_2(\R^d)$,
\begin{equation}
\label{eq:sandwich}
\mathcal H(\mu|\rho_*) \leqslant \mathcal F(\mu) - \mathcal F(\rho_*) = \overline{\mathcal F}(\mu) \leqslant \mathcal H\po \mu|\Gamma(\mu)\pf\,.
\end{equation}
The last inequality is particularly interesting in order to establish the PL inequality~\eqref{eq:PLF}, since it is then sufficient to prove that $\Gamma(\mu)$ satisfies a LSI (in dimension $d$ instead of $dN$ for $\mu_\infty^N$) with constant independent from $\mu$, and there are many known criteria to establish an LSI. This is still true if, more generally (i.e. even in  non-flat convex cases), an  inequality of the form
\begin{equation}
\label{eq:ineqEntropy}
\forall \mu \in\mathcal P_2(\R^d),\qquad \overline{\mathcal F}(\mu) \leqslant C \mathcal H \po \mu|\Gamma(\mu)\pf\,, 
\end{equation}
holds for some constant $C>0$ (as we will see in Proposition~\ref{prop:equivalenceInegal}, in this case, necessarily, $C\geqslant 1$). Following~\cite{SongboToAppear}, we will refer to such an inequality as a \emph{free energy condition}. This has been used, with $C>1$, in \cite{MonmarcheLSI,MonmarcheReygner} to establish PL inequalities for free energies. A natural question is then the relation between a PL inequality and a free energy condition~\eqref{eq:ineqEntropy}.  We shall see that they are equivalent in the case of our model~\eqref{eq:F}.

Alternatively, we may also consider  the first inequality of~\eqref{eq:sandwich}, or rather the weaker form
\begin{equation}
\label{eq:coercivityF}
 \forall \mu \in\mathcal P_2(\R^d),\qquad \overline{\mathcal F}(\mu) \geqslant c \mathcal H(\mu|\rho_*)\,,
\end{equation}
for some $c>0$. Following again \cite{SongboToAppear}, we refer to this as an entropic coercivity inequality. It is clear that such an inequality implies that $\rho_*$ is the unique global minimizer of $\mathcal F$. However, contrary to the free energy condition~\eqref{eq:ineqEntropy} and the PL inequality~\eqref{eq:PLF}, it may hold in situations where $\mathcal F$ admits other critical points. Moreover, it can be used to deduce interesting properties, such as concentration for the Gibbs measure and uniform propagation of chaos when the initial distribution is in the bassin of attraction of $\rho_*$, see \cite{SongboToAppear}. More generally, from a gradient flow perspective, such an inequality allows to deduce the convergence of the parameter $\rho_t$ (along the flow) to the optimal parameter $\rho_*$ from the convergence of the objective function $\mathcal F(\rho_t)$ to the minimum $\mathcal F(\rho_*)$.

Since $\rho_*=\gamma_{m_*}$ satisfies a T2 Talagrand transport inequality \cite{OttoVillani}, \eqref{eq:coercivityF} implies the Wasserstein coercivity inequality
\begin{equation}
\label{eq:coercivityFW2}
 \forall \mu \in\mathcal P_2(\R^d),\qquad \overline{\mathcal F}(\mu) \geqslant \frac{c\kappa}2 \mathcal W_2^2(\mu,\rho_*)\,,
\end{equation}
with the same $c>0$.  Such an inequality, called a non-linear transport inequality in \cite{MonmarcheReygner} or a functional Łojasiewicz inequality in \cite{BLANCHET20181650} (where~\eqref{eq:PLF} is called by contrast a Łojasiewicz gradient property; both inequalities appear in \cite[Section 18]{law1965ensembles}), is implied by the PL inequality~\eqref{eq:PLF} when $\rho_*$ is the unique minimizer of $\mathcal F$, see \cite[Proposition 6]{ChewiStromme} or more generally \cite[Theorem 1]{BLANCHET20181650}. However, being a consequence of~\eqref{eq:coercivityF}, it can hold in cases where the PL inequality fails. In \cite{MonmarcheReygner}, a local version of \eqref{eq:coercivityFW2} is used in some intermediate steps to show that a local PL inequality implies the stability of some Wasserstein ball (relying on the idea that, as before, the inequality~\eqref{eq:coercivityFW2} provides the convergence of the parameter $\rho_t$ from the convergence of the objective function $\mathcal F(\rho_t)$).

\medskip

In finite dimension, it is classical that, at degenerate local minimizers (i.e. at which the Hessian is singular), it is still possible for some degenerate Łojasiewicz inequalities to hold, see \cite[Proposition 1]{law1965ensembles}, which can be used to deduce polynomial convergence rates for the corresponding gradient flows. In infinite dimension, for free energies in the Wasserstein space, a general study is given in \cite{BLANCHET20181650} and an application is given \cite{MonmarcheReygner} for the double-well continuous Curie-Weiss at critical temperature. We can expect this situation to be representative, in the sense that in general it would not be surprising that models exhibiting transition phases with respect to some parameters (e.g. temperature) satisfy at best a degenerate Łojasiewicz inequalities at the critical parameters (see e.g. Example~\ref{ex:degen}). Motivated by this, and by the relative scarcity of other examples in the literature apart from~\cite{MonmarcheReygner} (indeed, in \cite{BLANCHET20181650}, the only degenerate example, in Section 3.3, is for a potential energy $\mu \mapsto \int_{\R^d} V \mu$, which is linear with respect to $\mu$), we shall discuss under which condition this situation arise in the case of the toy models~\eqref{eq:F}.

Finally, \cite{MonmarcheReygner,MonmarcheMetastable} have developed a research direction concerning local convergence rates for Wasserstein gradient flow, and thus we will also discuss local inequalities in our situation.

\subsection{Non-degenerate minimizer: contraction}\label{subsec:non-degen-contract}

Our main result here is that, in the ``nicest" situation (i.e. unique globally attractive minimizer), in the case of the model~\eqref{eq:F}, all the inequalities discussed in Section~\ref{sec:motivInegalites} related to long-time convergence of the processes are equivalent. A remarkable fact concerning these   equivalences between functional inequalities is that the constant are sharp, as stated in~\eqref{eq:egaliteConstantes}.

\begin{proposition}\label{prop:equivalenceInegal}
Under Assumption~\ref{assu:general}, assume furthemore that $V_\kappa$ admits a unique global minimizer $m_*$ and that there exists $L>0$ such that $\Delta V \leqslant L(1+|\na V_\kappa|^2)$. Then, the following are equivalent :
\begin{enumerate}[(i)]
\item The potential $V_\kappa$ satisfies a PL inequality, namely there exists $C>0$ such that
\begin{equation}
\label{eq:PLfinitedim}
\forall x\in\R^d,\qquad V_\kappa(x) - \inf V_\kappa  \leqslant \frac{C}{2}|\na V_\kappa(x)|^2\,, 
\end{equation}
\item  The free energy $\mathcal F$ satisfies a PL inequality~\eqref{eq:PLF}.
\item For all $N\geqslant 1$ the Gibbs measure $\mu_\infty^N$ satisfies a  LSI~\eqref{eq:LSI}, and $\liminf_N C_{LS}(\mu_\infty^N)<\infty$.
\item The free energy $\mathcal F$ satisfies a free energy condition~\eqref{eq:ineqEntropy}.
\end{enumerate}
Moreover, in this situation,
\begin{equation}
\label{eq:egaliteConstantes}
C_{LS}(\mu_\infty^N) \underset{N\rightarrow \infty} \longrightarrow \max(\kappa^{-1}, C_{PL}(V_\kappa)) = C_{PL}(\mathcal F) = \kappa^{-1} C_{\mathcal F\mathcal H} \,,
\end{equation}
where $C_{PL}(V_\kappa)$ and $C_{\mathcal F\mathcal H}$ stands for the smallest constant such that, respectively, the PL inequality~\eqref{eq:PLfinitedim} and the free energy condition~\eqref{eq:ineqEntropy} hold.

\end{proposition}

The inequalities discussed in Proposition~\ref{prop:equivalenceInegal} classically have many  useful consequences for the gradient flow~\eqref{eq:EDP} and the particle system~\eqref{eq:particulesEDS2}, see e.g. \cite[Section 4]{MonmarcheMetastable} (or \cite{Katharina} for practical numerical schemes).

\begin{remark}
The assumption that $\Delta V \leqslant L(1+|\na V_\kappa|^2)$ is only used to prove that $(i)\Rightarrow (iii)$. It is not required for the implications $(iii)\Rightarrow (i),(ii)$, nor $(ii)\Rightarrow(i)$ and $(ii)\Leftrightarrow (iv)$. 
\end{remark}

\begin{example}
\label{ex:PLnondegeneree}
A consequence of the discussion in Examples~\ref{ex:criticalpoints} and \ref{ex:hessienne2} is that, for the running example~\eqref{eq:VmuEncoder}, the situation described by Proposition~\ref{prop:equivalenceInegal} holds if and only if $\kappa$ is larger than $\lambda_n$ the largest eigenvalue of $M$. Indeed, if $\kappa < \lambda_n$, $0$ is a critical point of $V_\kappa$ which is not a global minimizer, so that a PL inequality cannot hold, and if $\kappa=\lambda_n$, $\na^2 V_\kappa(0)$ is singular, which again prevents a PL inequality (see Figure~\ref{fig:Vkappa}\emph{(a)(b)(c)}). For $\kappa>\lambda_n$, $0$ is the unique critical point, $\na^2 V_\kappa(0)$ is definite positive and more generally, for any $m\in\R^d$, decomposing $\na^2 V(m)$ as in Example~\ref{ex:hessienne2} onto orthogonal planes by diagonalizing $M$,  with similar computations we see that for all $u,v\in\R^d$,
\[\begin{pmatrix}
u \\ v 
\end{pmatrix}^T \na^2 V(m) \begin{pmatrix}
u \\ v 
\end{pmatrix} \geqslant - 2 u^T  M v  \geqslant - \lambda_n \po |u|^2 + |v|^2\pf\,,  \]
so that $V_\kappa$ is strongly convex, and thus satisfies a PL (see Figure~\ref{fig:Vkappa}\emph{(d)}).
\end{example}

\begin{proof}
\textbf{[(i)$\Rightarrow$(iii)]} Chewi and Stomme showed in \cite[Theorem 1]{ChewiStromme} that $(i)$ implies that the measure $\nu_\infty^\infty$ defined in~\eqref{eq:nuinfinityN} satisfies a LSI for all $N\geqslant 1$ with $N C_{LS}(\nu_\infty^N) \rightarrow C_{PL}(V_\kappa) $ as $N\rightarrow \infty$. From Lemma~\ref{lem:LSI}, we deduce that $(i)$ implies $(iii)$, with moreover $\lim_N C_{LS}(\mu_\infty^N) =\max(C_{PL}(V_\kappa),\kappa^{-1}) $.

\textbf{[(iii)$\Rightarrow$(i)]} Conversely, assuming $(iii)$, Lemma~\ref{lem:LSI} shows that $\liminf N C_{LS}(\nu_\infty^N) <\infty$, from which $(i)$ follows from  \cite[Theorem 11]{ChewiStromme}.

\textbf{[(i) and (iii)$\Rightarrow$(ii)]} Now, assume  $(iii)$ and  $(i)$ in order to prove $(ii)$. From \cite[Proposition 7]{ChewiStromme}, thanks to the PL inequality for $V_\kappa$, $\na^2 V(m_*)$ is positive definite. For  $\rho\in\mathcal P_2(\R^d)$ with finite entropy,
\[\frac1N \mathcal H\po \rho^{\otimes N}|\mu_\infty^N \pf = \int_{\R^{dN}} V(\bar x) \rho^{\otimes N}(\bx) + \int_{\R^d} |x_1|^2 \rho(\dd x_1) + \int_{\R^d} \rho \ln\rho  + \frac1N \ln Z_N \]
with the normalizing constant (known as the partition function in this context)
\[Z_N = \int_{\R^{dN}} e^{-N V(\bar x) + \frac{\kappa}{2}  |\bx |^2 } \dd \bx = (2\pi/\kappa)^{d(N-1)/2} N^{d/2}\int_{\R^d} e^{-N V_\kappa(y) } \dd y \,. \]
By Laplace method,
\begin{multline*}
\int_{\R^d} e^{-N V_\kappa(y) } \dd y \underset{N\rightarrow\infty} \sim \int_{\R^d} e^{-N \co V_\kappa(m_*) + \frac12 \na^2 V_\kappa(m_*)(x-m_*)^2 \cf  } \dd y \\
= (2\pi / [N \mathrm{det}(\na^2 V_\kappa(m_*))])^{d/2} e^{-N V_\kappa(m_*)} 
\end{multline*}
and thus
\[\frac1N \ln Z_N \underset{N\rightarrow \infty}\longrightarrow  \frac{d}{2}\ln(2\pi/\kappa) - V_\kappa(m_*)  = - \mathcal F(\gamma_{m_*})\,. \]
Now, take $\rho\in\mathcal P_2(\R^d)$ with a smooth compactly supported density. By the law of large numbers, $\bar x$ converges a.s. to $m_\rho$ under $\rho^{\otimes N}$, and thus by dominated convergence (since $V$ is bounded on the support of $\rho$)
\begin{equation}
\label{loc:dominated}
 \int_{\R^{dN}} V(\bar x) \rho^{\otimes N}(\bx) \underset{N\rightarrow \infty}\longrightarrow  V(m_\rho)\,. 
\end{equation}
Second,
\begin{align*}
\frac1N \mathcal I\po \rho^{\otimes N}|\mu_\infty^N \pf &= \frac1N \sum_{i=1}^N \int_{\R^{dN}} \left| \na_{x_i} \ln\rho(x_i) - \na_{x_i} U_N(\bx) \right|^2  \rho^{\otimes N}(\bx)\\
&=  \int_{\R^{dN}} \left| \na_{x_1} \ln\rho(x_1) - \kappa x_1 - \na V(\bar x) \right|^2  \rho^{\otimes N}(\bx)\,.
\end{align*}
Again, by dominated convergence,
\[\frac1N \mathcal I\po \rho^{\otimes N}|\mu_\infty^N \pf \underset{N\rightarrow \infty}\longrightarrow   \int_{\R^{d}} \left| \na_{x_1} \ln\rho(x_1) - \kappa x_1 - \na V(m_\rho) \right|^2  \rho(\dd x_1) = \mathcal I(\rho|\Gamma(\rho))\,. \]
Hence, dividing by $N$ in the uniform LSI satisfies by $\mu_\infty^N$ applied to $\rho^{\otimes N}$ and letting $N\rightarrow 0$ gives the PL inequality for $\mathcal F$ (for all $\rho$ with smooth compactly supported densities, and thus all $\rho\in\mathcal P_2(\R^d)$ by density), with
\begin{equation}
\label{loclm}
C_{PL}(\mathcal F) \leqslant  \underset{N\rightarrow \infty}\lim   C_{LS}(\mu_\infty^N) = \max( C_{PL}(V_\kappa),\kappa^{-1})\,.
\end{equation}
This concludes the proof that $(iii)$ and $(i)$ imply $(ii)$.

\textbf{[(ii)$\Rightarrow$(i)]} Now, assume $(ii)$ and write $\bar C = C_{PL}(\mathcal F)$.  Thanks to \eqref{eq:FVkappa}, applying the PL inequality for $\mathcal F$ with $\mu=\gamma_m$ for any $m\in\R^d$ reads
\begin{align}
\label{eq:PLVkappa2}
V_\kappa(m) - V_\kappa(m_*) = \mathcal F(\gamma_m) - \mathcal F(\gamma_{m_*}) & \leqslant \frac{\bar C}{2} \int_{\R^d} \left|\na \ln \frac{\gamma_m}{\gamma_{f(m)}} \right|^2 \gamma_m \nonumber\\
& = \frac{\bar C \kappa^2 }{2} |m-f(m)|^2 = \frac{\bar C}{2} |\na V_\kappa(m)|^2\,.
\end{align}
This proves the PL inequality for $V_\kappa$ with $\bar C \geqslant C_{PL}(V_\kappa)$. Alternatively, applying the PL inequality for $\mathcal F$ to $\mu$ with $m_\mu=m_*$ gives
\begin{equation}
\label{loc45}
\mathcal F(\mu) - \mathcal F(\rho_*) =   \mathcal H\po \mu|\rho_* \pf \leqslant \frac{\bar C}{2} \mathcal I \po \mu| \Gamma(\mu)\pf = \frac{\bar C}{2} \mathcal I(\mu|\rho_*)\,. 
\end{equation}
Chosing  $\mu = \mathcal N(m_*,\sigma^2 I_d)$ for some $\sigma^2>0$, we see that
\begin{equation}
\label{eq:locgaussHI}
\mathcal H(\mu|\rho_*) = \frac{d}{2}  \co \sigma^2 \kappa - 1   - \ln (\kappa \sigma^2)\cf    \,,\qquad \mathcal I(\mu|\rho_*) = d \sigma^2 \po \kappa -\sigma^{-2}\pf^2\,.
\end{equation}
Plugging these values in~\eqref{loc45} and taking $\sigma^2$ arbitrarily large shows that $\bar C \geqslant \kappa^{-1}$. As a conclusion, we have proven that $(ii)$ implies $(i)$ with moreover
\begin{equation}
\label{locjk}
\max\po\kappa^{-1},C_{PL}(V_\kappa)\pf \leqslant C_{PL}(\mathcal F)\,.
\end{equation}
This concludes the proof of the equivalence between $(i)$, $(ii)$ and $(iii)$, and combining \eqref{loclm} with \eqref{locjk} gives the corresponding equalities in~\eqref{eq:egaliteConstantes}.

\textbf{[(ii)$\Rightarrow$(iv)]} Assume $(ii)$.  
 Combining \eqref{locgh}  and \eqref{loc1bis} applied  with $m'=f(m_\mu)$ reads
\begin{equation}
\label{loc2bishop}
\mathcal F(\mu) - \mathcal F(\rho_*) =  V_\kappa(m) - V_\kappa(m_*) -\frac1{2\kappa} |\na V_\kappa(m_\mu) |^2 + \mathcal H \po \mu|\Gamma(\mu)\pf\,,
\end{equation}
where we used that $\kappa|m_\mu - f(m_\mu)|^2 = |\na V_\kappa(m_\mu) |^2/\kappa$. Using the PL inequality gives
\[\mathcal F(\mu) - \mathcal F(\rho_*) \leqslant \po \frac{C_{PL}(V_\kappa)}{2} - \frac{1}{2\kappa}\pf_+  |\na V_\kappa(m_\mu) |^2 + \mathcal H \po \mu|\Gamma(\mu)\pf\,.\]
Moreover, since $f(m_\mu)=m_{\Gamma(\mu)}$,
\begin{equation}
\label{loc3}
|\na V_\kappa(m_\mu) |^2 = \kappa^2 |m_\mu - f(m_\mu)|^2 \leqslant  \kappa^2 \mathcal W_2^2\po \mu,\Gamma(\mu)\pf \leqslant  2 \kappa \mathcal H\po \mu|\Gamma(\mu)\pf\,,
\end{equation}
where we used the T2 Talagrand inequality (here with constant $\kappa^{-1}$) satisfied by Gaussian measures \cite{OttoVillani}. We have thus obtained $(iv)$ with
\[C_{\mathcal F\mathcal H} \leqslant 1+ \po \kappa C_{PL}(V_\kappa) - 1\pf_+ = \max \po 1,\kappa C_{PL}(V_\kappa)\pf\,.  \]

\textbf{[(iv)$\Rightarrow$(ii)]} Conversely, assume $(iv)$. First, applying the entropy inequality~\eqref{eq:ineqEntropy} with $\mu \neq \rho_*$ such that $m_\mu = \rho_*$, we get that
\[\mathcal H(\mu|\rho_*)  = \mathcal F(\mu) - \mathcal F(\rho_*) \leqslant C_{\mathcal F\mathcal H} \mathcal H(\mu|\rho_*) \,,\]
which shows that $C_{\mathcal F\mathcal H} \geqslant 1$. Second,  applying the entropy inequality~\eqref{eq:ineqEntropy} to $\mu=\gamma_m$ and using that $\mathcal H\po \gamma_m|\Gamma(\gamma_m)\pf=\frac{1}{2\kappa}|\na V_\kappa(m)|^2$ thanks to~\eqref{loc1bis} applied with $m'=f(m)$ 
 reads
\[V_\kappa(m) - V_\kappa(m_*) = \mathcal F(\gamma_m) - \mathcal F(\rho_*) \leqslant \frac{C_{\mathcal F\mathcal H}}{2\kappa}|\na V_\kappa(m)|^2\,. \]
which shows the PL inequality for $V_\kappa$ with $C_{PL}(V_\kappa) \leqslant C_{\mathcal F\mathcal H}/\kappa$, concluding the proof of the proposition.

\end{proof}

Notice that, under the settings of Proposition~\ref{prop:equivalenceInegal}, the PL inequality~\eqref{eq:PLfinitedim} for $V_\kappa$ (which cannot holds if $V_\kappa$ admits two isolated critical points: if there are two  isolated global minimizers, then the path with minimal elevation from one to the other goes through a saddle point, so that there is a critical point which is not a global minimizer) is easily characterized:

\begin{proposition}
Under Assumption~\ref{assu:general}, assume that $V_\kappa$ admits a unique critical point $m_* \in\R^d$. Then the PL inequality~\eqref{eq:PLfinitedim} holds for some $C>0$ if and only if $\na^2 V_\kappa(m_*)$ is non-singular.
\end{proposition}

\begin{proof}
If the PL inequality holds, then $\na^2 V(m_*)$ is non-singular thanks to \cite[Proposition 7]{ChewiStromme}.

Conversely, assume $\na^2 V(m_*)$ is non-singular. Since $V_\kappa$ goes to infinity at infinity, as the unique critical points, $m_*$ is necessarily the global minimizer of $V_\kappa$, and thus $\na^2 V(m_*)$ is positive definite. Under Assumption~\ref{assu:general}, we can thus find $R>r>0$ and $\lambda>0$ such that $\na^2 V_\kappa(m) \geqslant \lambda I_d$ for all $m\in\R^d$ with either $|m-m_*|\geqslant R$ or $|m-m_*|\leqslant r$. Denote $\mathcal D= \{m\in\R^d,\ r<|m-m_*|<R\}$.  Let $\xi = \inf_{\mathcal D}|\na V_\kappa|$, which is positive by compactness since $m_*$ is the unique critical point of $V_\kappa$. Then, as long as the gradient flow  $\dot x_t = -\na V_\kappa(x_t)$ stays in $\mathcal D$,
\[\partial_t f(x_t) - f(m_*) = - |\na f(x_t)|^2 \leqslant - \xi^2\,.\]
This shows that the time that a trajectory can spend in $\mathcal D$ is at most $T = [\sup_{\mathcal D}V_\kappa - \inf_{\mathcal D}V_\kappa]/\xi^2$. Let $\Lambda = \sup_{\mathcal D}\|\na^2 V_\kappa\|_\infty$, so that
\[\partial_t |f(x_t)|^2 = -2 \na f(x_t) \cdot \na^2 V_\kappa(x_t) \na f(x_t)  \leqslant - 2 \lambda |\na f(x_t)|^2 + 2 (\lambda + \Lambda) |\na f(x_t)|^2 \1_{x_t \in \mathcal D}\,.\]
We end up with
\[|\na f(x_t) |^2 \leqslant e^{2(\lambda+\Lambda) T} e^{-2\lambda t} |\na f(x_0)|^2\]
for all  $t\geqslant 0$ and initial condition $x_0\in\R^d$. We conclude by
\[f(x_0) - f(m_*) = \lim_{t\rightarrow \infty} f(x_0) - f(x_t) =  \int_0^\infty |\na f(x_t)|^2 \dd t \leqslant \frac{e^{2(\lambda+\Lambda) T}}{2\lambda } |\na f(x_0)|^2\,, \]
for any $x_0\in\R^d$, which is exactly~\eqref{eq:PLfinitedim}.
\end{proof}

\subsection{Non-degenerate minimizer: coercivity}\label{subsec:coerci-non-degen}

We turn to the study of the coercivity properties of the free energy, as in \cite{SongboToAppear}. Before addressing the $N$ particle system, let us first characterize the (infinite dimensional) coercivity inequalities~\eqref{eq:coercivityF} and \eqref{eq:coercivityFW2} in terms of the corresponding (finite dimensional)  condition on $V_\kappa$.

\begin{proposition}\label{prop:coercivity}
Under Assumption~\ref{assu:general}, let $m_*\in\R^d$ be a global minimizer of $V_\kappa$ and $\rho_*=\gamma_{m_*}$. Then the following are equivalent:
\begin{enumerate}[(i)]
\item The entropic coercivity  inequality~\eqref{eq:coercivityF} holds for some $c>0$.
\item The $\mathcal W_2$ coercivity  inequality~\eqref{eq:coercivityFW2} holds for some $c>0$.
\item $V_\kappa$ satisfies a coercivity inequality
\begin{equation}
\label{eq:coercivityV}
\forall m\in \R^d,\qquad V_\kappa(m) - V_\kappa(m_*) \geqslant \frac{c\kappa}2 |m-m_*|^2\,,
\end{equation}
for some $c>0$.
\end{enumerate}
  Moreover, in that case, writing $c_{coer}^{\mathcal H}(\mathcal F)$, $c_{coer}^{\mathcal W_2}(\mathcal F)$ and $c_{coer}(V_\kappa)$ the largest constants such that, respectively, \eqref{eq:coercivityF}, \eqref{eq:coercivityFW2} and \eqref{eq:coercivityV} hold, then 
  \begin{equation}
  \label{loc:eq:egalccoer}
 c_{coer}^{\mathcal H}(\mathcal F) = c_{coer}^{\mathcal W_2}(\mathcal F)  = \min\po 1,c_{coer}(V_\kappa)\pf\,.
  \end{equation}
\end{proposition}

\begin{remark}
Under Assumption~\ref{assu:general}, $V_\kappa$ is strongly convex outside a compact. In this situation, the coercivity inequality~\eqref{eq:coercivityV} holds for some $c>0$ if and only if $m_*$ is the unique global minimizer of $V_\kappa$ and is non-degenerate (but, contrary to the case of Proposition~\ref{prop:equivalenceInegal}, $V_\kappa$ may have other critical points).
\end{remark}

\begin{proof}
\textbf{[(i)$\Rightarrow$(ii)]}. This implication is a direct consequence of the Talagrand inequality (with constant $\kappa$) satisfied by $\rho_*=\gamma_{m_*}$, which shows $c_{coer}^{\mathcal W_2}(\mathcal F) \geqslant c_{coer}^{\mathcal H}(\mathcal F)$.

\textbf{[(ii)$\Rightarrow$(iii)]} applying~\eqref{eq:coercivityFW2} with $\mu=\gamma_m$ for some $m\in\R^d$ gives 
\[V_\kappa(m) - V_\kappa(m_*) = \overline{\mathcal F}(\gamma_m) \geqslant \frac{c\kappa}2 \mathcal W_2^2(\gamma_m,\rho_*) = \frac{c \kappa}{2}|m-m_*|^2 \,, \]
which shows that~\eqref{eq:coercivityV} holds with $c_{coer}(V_\kappa) \geqslant c_{coer}^{\mathcal W_2}(\mathcal F)$.

\textbf{[(iii)$\Rightarrow$(i)]}  Assume~\eqref{eq:coercivityV}. Using~\eqref{locgh}, the positivity of the relative entropy and then~\eqref{loc1bis}, for any $\delta \in(0,1]$,
\begin{align*}
\mathcal F(\mu) - \mathcal F(\rho_*) 
& =  V_\kappa(m_\mu) - V_\kappa(m_*) +  \mathcal H\po \mu|\gamma_{m_\mu} \pf  \\
& \geqslant  V_\kappa(m_\mu) - V_\kappa(m_*) +  \delta \mathcal H\po \mu|\gamma_{m_\mu} \pf  \\
& \geqslant \frac{c\kappa}{2} |m_\mu-m_*|^2 + \delta \po   \mathcal H\po \mu|\rho_* \pf  - \frac\kappa2 |m_\mu - m_*|^2 \pf  \,.
\end{align*}
Taking $\delta = \min(1,c)$ shows \eqref{eq:coercivityV} with $ c_{coer}^{\mathcal H}(\mathcal F)  \geqslant \min\po 1,c_{coer}(V_\kappa)\pf$.

This concludes the proof of the  equivalences. To get~\eqref{loc:eq:egalccoer},  
it remains to apply~\eqref{eq:coercivityFW2} with $\mu = \mathcal N(m_*,\sigma^2)$ for an arbitrary $\sigma^2>0$, which, recalling~\eqref{eq:locgaussHI}, reads
\begin{equation}
\label{loc:d2}
\frac{d}{2}  \co \sigma^2 \kappa - 1   - \ln (\kappa \sigma^2)\cf  = \mathcal H(\mu|\rho_*) = \overline{\mathcal F}(\mu) \geqslant \frac{c\kappa}{2}\mathcal W_2^2(\mu,\rho_*) = \frac{c\kappa}{2} d \po \sigma - \kappa^{-1/2} \pf^2   \,.
\end{equation}
Taking $\sigma$ arbitrarily large shows that $c_{coer}^{\mathcal W_2}(\mathcal F) \leqslant 1$. Combining this with the previous inequalities obtained yields~\eqref{loc:eq:egalccoer}.
\end{proof}

To discuss the consequences of the coercivity inequalities from Proposition~\ref{prop:coercivity} to the $N$ particles system as in \cite{SongboToAppear}, we introduce, for a density $\nu^N \in \mathcal P_2(\R^{dN})$, the modulated $N$-particle free energy
\begin{equation}
\label{eq:FNnuN}
\mathcal F^N(\nu^N|\rho_*) = \int_{\R^{dN}} \nu^N \ln \nu^N  + N \int_{\R^{dN}}  \mathcal E(\pi_{\bx}) \nu^{N}(  \dd \bx) - N \mathcal F(\rho_*)\,.
\end{equation}

\begin{proposition}
\label{prop:coercivityN}
Under Assumption~\ref{assu:general}, let $m_*\in\R^d$ be a global minimizer of $V_\kappa$ and $\rho_*=\gamma_{m_*}$. Assuming furthermore that $\na^2 V$ is bounded, then the three statements of Proposition~\ref{prop:coercivity} are also equivalent to the following:
\begin{enumerate}[(i)]
\setcounter{enumi}{3}
\item There exists $c>0$ such that for all $\varepsilon>0$, $N\geqslant 1$ and all density $\nu^N \in \mathcal P_2(\R^{dN})$, 
\begin{equation}
\label{eq:coercivityN}
\mathcal F^N(\nu^N|\rho_*) \geqslant \po c -  2 \frac{\|\na^2 V\|_\infty}{\kappa}  \po  \varepsilon + \frac{1+\varepsilon^{-1}}{N}\pf \pf \mathcal H \po \nu^N |\rho_*^{\otimes N}\pf -   \frac{\|\na^2 V\|_\infty}{\kappa}d (1+\varepsilon   + \varepsilon^{-1})   \,.
\end{equation}
\end{enumerate}
Moreover, in this case, the biggest $c>0$ such that $(iv)$ is satisfied is equal to $c_{coer}^{\mathcal H}(\mathcal F)$.
\end{proposition}

For useful consequences of~\eqref{eq:coercivityN}, see Corollary 3  (for concentration inequalities) or Section~3.2 (for uniform-in-time propagation of chaos) of \cite{SongboToAppear}.

\begin{proof}
\textbf{[(i)$\Rightarrow$(iv)]} The argument is inspired by the proof of \cite[Theorem 1]{SongboToAppear},  our situation being by some aspects simpler since the interaction involves only the barycentre of the system (however, this previous result doesn't apply directly to our case since it only covers pairwise interactions, which in our case corresponds to a quadratic $V$).  First, decomposing
\[\mathcal F(\mu) = \mathcal H(\mu|\rho_*) + \int \ln \rho_* \mu + \mathcal E(\mu)\,, \] 
we re-interpret the  coercivity  inequality~\eqref{eq:coercivityF} as the fact that for all $\mu\in\mathcal P_2(\R^d)$,
\begin{equation}
\label{loc:reinterpret}
(1-c) \mathcal H(\mu|\rho_*) + \int \ln \rho_* \mu + \mathcal E(\mu) - \mathcal F(\rho_*) \geqslant 0\,.
\end{equation}
Similarly, we decompose
\begin{equation}
\label{loc:FN}
\mathcal F^N(\nu^N|\rho_*) = \mathcal H \po \nu^N |\rho_*^{\otimes N}\pf + \int_{\R^{dN}} \nu^N(\bx) \co \ln \rho_*^{\otimes N}(\bx) + N \mathcal E(\pi_{\bx}) \cf \dd \bx - N \mathcal F(\rho_*)\,. 
\end{equation}
Denote by $\nu_k$ the conditional distribution of $X_k$ given $(X_j)_{j\in\cco 1,k-1\ccf}$ when $\bX \sim \nu^N$, and $\bar \nu = \frac1N\sum_{k=1}^N \nu_k$. In particular,  by interchangeability, $\mathbb E(\bar\nu)$ is the first $d$-dimensional marginal of $\nu^N$, and
\[ \int_{\R^{dN}} \nu^N(\bx) \co \ln \rho_*^{\otimes N}(\bx) + \frac{\kappa}{2}|\bx|^2 \cf \dd \bx = N \mathbb E \po  \int_{\R^d}  \co \ln \rho_*(x) + \frac{\kappa}{2}|x|^2\cf \bar\nu(\dd x)\pf \,.\]
By the chain rule of the entropy and then its convexity,
\begin{equation}
\label{eq:chainruleH}
\frac1N \mathcal H \po \nu^N |\rho_*^{\otimes N}\pf = \frac1N \sum_{k=1}^N \mathbb E \po \mathcal H \po \nu_k |\rho_*\pf\pf  \geqslant \mathbb E \po \mathcal H \po \bar\nu |\rho_*\pf\pf\,.
\end{equation}
Using this in \eqref{loc:FN} gives
\begin{align}
\mathcal F^N(\nu^N|\rho_*) &\geqslant  c \mathcal H \po \nu^N |\rho_*^{\otimes N}\pf + N \int_{\R^{dN}} V(\bar x) \nu^N (\dd \bx) \nonumber \\
&\qquad  +  N \mathbb E \po  (1-c) \mathcal H \po \bar\nu |\rho_*\pf +   \int_{\R^d}  \co \ln \rho_*(x) + \frac{\kappa}{2}|x|^2\cf \bar\nu(\dd x) - \mathcal F(\rho_*)\pf \nonumber\\
&\geqslant  c \mathcal H \po \nu^N |\rho_*^{\otimes N}\pf + N  \int_{\R^{dN}} V(\bar x) \nu^N (\dd \bx) - N \mathbb E \po V(m_{\bar\nu})\pf \,,\label{loc:654}
\end{align}
where we used~\eqref{loc:reinterpret} with $\mu=\bar \nu$. It remains to control the difference of the two last terms.

 Considering $\bX\sim \nu^N$, we bound 
\begin{align*}
 V(m_{\bar\nu}) - V(\bar X) & \leqslant \na V(\bar X) \cdot (\bar X-m_{\bar\nu}) + \frac{\|\na^2 V\|_\infty}{2}|\bar X - m_{\bar\nu}|^2\\
 & \leqslant \na V(m_*) \cdot (\bar X-m_{\bar\nu}) +   \frac{\|\na^2 V\|_\infty}{2}\co \varepsilon|\bar X-m_*|^2 + (1 + \varepsilon^{-1}) |\bar X-m_{\bar\nu}|^2\cf\,,
\end{align*}
for any $\varepsilon>0$.  Taking the expectation gives
\begin{equation}
\label{loc:123}
\mathbb E \po V(m_{\bar\nu}) - V(\bar X)\pf \leqslant  \frac{\|\na^2 V\|_\infty}{2}\co \varepsilon \mathbb E\po |\bar X-m_*|^2\pf  + (1 + \varepsilon^{-1}) \mathbb E \po |\bar X-m_{\bar\nu}|^2\pf \cf\,.
\end{equation}
To bound the first term, considering $(\bX,\mathbf{Y})$ a coupling of $\nu^N$ and $\rho_*^{\otimes N}$, we bound
\[
\mathbb E\po |\bar X-m_*|^2\pf   \leqslant 2 \mathbb E\po |\bar X- \bar Y |^2\pf + 2 \mathbb E\po |\bar Y-m_*|^2\pf  \leqslant \frac{2}{N} \mathbb E \po |\bX-\mathbf{Y}|^2 \pf + \frac{2}{N}\mathrm{Var}(\rho_*)\,.
\]
Taking the infimum over all couplings and using Talagrand inequality for $\rho_*^{\otimes N}$ gives
\begin{equation}
\label{loc:567}
\mathbb E\po |\bar X-m_*|^2\pf    \leqslant \frac{4}{N\kappa } \mathcal H \po \nu^N|\rho_*^{\otimes N}\pf  + \frac{2d}{N\kappa}\,.
\end{equation}
It remains to control the last term of~\eqref{loc:123}.  For $k\in\cco 0,N\ccf$, write $\mathcal F_k = \sigma(\{X_j\}_{j\in\cco 1,k\ccf})$ and consider the  martingale
\[M_k = \sum_{j=1}^k \po X_j - \mathbb E \po X_j|\mathcal F_{j-1}\pf\pf\,.\]
Then, by the martingale property,
\[ \mathbb E \po |\bar X-m_{\bar\nu}|^2\pf  = \frac{1}{N^2 }\mathbb E \po M_N^2\pf = \frac{1}{N^2} \sum_{k=1}^N \mathbb E \po |X_k -\mathbb E(X_k|\mathcal F_{k-1})|^2\pf  \,.\]
By the variational characterisation of the expectation,
\[ \mathbb E \po |X_k -\mathbb E(X_k|\mathcal F_{k-1})|^2 |\mathcal F_{k-1}\pf \leqslant  \mathbb E \po |X_k - m_*|^2\pf\,,\]
and thus, reasoning as previously,
\[ \mathbb E \po |\bar X-m_{\bar\nu}|^2\pf \leqslant  \frac{1}{N^2} \sum_{k=1}^N \mathbb E \po |X_k -m_*|^2\pf \leqslant \frac{2}{N^2} \mathcal W_2^2\po \nu^N,\rho_*^{\otimes N}\pf  + \frac{2d}{N \kappa } \,.\]
Using Talagrand inequality, plugging the resulting inequality together with~\eqref{loc:567} in \eqref{loc:123} and the latter in~\eqref{loc:654} finally leads to 
\[
\mathcal F^N(\nu^N|\rho_*) \geqslant c \mathcal H \po \nu^N |\rho_*^{\otimes N}\pf -   \frac{\|\na^2 V\|_\infty}{\kappa}\co 2 \po  \varepsilon + \frac{1+\varepsilon^{-1}}{N}\pf  \mathcal H \po \nu^N|\rho_*^{\otimes N}\pf  + d (1+\varepsilon   + \varepsilon^{-1}) \cf \,,
\]
concluding the proof of the implication (with $c \geqslant c_{coer}^{\mathcal H}(\mathcal F)$).

\medskip

\textbf{[(iv)$\Rightarrow$(i)]} Applying~\eqref{eq:coercivityN} to $\nu^N = \nu^{\otimes N}$ for a density $\nu$ with compact support, dividing by $N$ and letting $N\rightarrow \infty$ (using~\eqref{loc:dominated}) gives
\[\overline{\mathcal F}(\nu) \geqslant \co c -   2 \varepsilon \frac{\|\na^2 V\|_\infty}{\kappa} \cf \mathcal H \po \nu |\rho_*\pf\,. \]
Since $\varepsilon>0$ is arbitrary, this proves the coercivity inequality~\eqref{eq:coercivityF} with $c_{coer}^{\mathcal H}(\mathcal F) \geqslant c$ (for all $\nu$ with compact support and then all $\nu\in\mathcal P_2(\R^d)$ by density).

\end{proof}

\subsection{Degenerate minimizer: contraction}\label{sec:degenPL}

This section aims at adapting the discussion in Section~\ref{subsec:non-degen-contract} but in the situation where $\na^2 V_\kappa(m_*)$ is singular.  As it is quite long, it is split in three parts: Section~\ref{subsubsec:PLdegenMF} focuses on Łojasiewicz inequalities for the mean-field free energy. In Section~\ref{subsubsec:StandardLSIdegenerate}, we will see how, in these situations, the log-Sobolev constant of the Gibbs measure grows polynomially with $N$. In Section~\ref{subsubsec:degenerateLSIdegenerate}, we will see however that the Gibbs measure satisfies a somewhat degenerate LSI (which is simply a Łojasiewicz inequality when interpreting the Fokker-Planck equation associated with the particles as the gradient flow of the relative entropy with respect to the Gibbs measure) with the correct scaling in $N$ (in the sense that letting $N \rightarrow \infty$ in this inequality recovers the mean-field Łojasiewicz inequality for $\mathcal F$).

\subsubsection{General Łojasiewicz inequalities for the mean-field free energy}\label{subsubsec:PLdegenMF}

For a nondecreasing function $\Theta:\R_+\rightarrow\R_+$, we say that $\mathcal F$ satisfies a $\Theta$-ŁI (for \emph{Łojasiewicz Inequality}) if
\begin{equation}
\label{eq:degenPL}
\forall \mu \in \mathcal P_2(\R^d),\qquad \overline{\mathcal F}(\mu) \leqslant \Theta\po \mathcal I(\mu|\Gamma(\mu))\pf\,.
\end{equation}
Notice that, necessarily, $\Theta(r) > 0$ for $r>0$, since a global minimizer $\mu$ of $\mathcal F$ is also a critical point, i.e. $\mathcal I(\mu|\Gamma(\mu))=0$. We say that the inequality is tight if $\Theta$ is continuous at $0$ with $\Theta(0)=0$, defective otherwise.

When $\Theta$ is concave,~\eqref{eq:degenPL} is implied by the modified log-Sobolev inequality 
\begin{equation}
\label{eq:qLSI}
\forall \mu \in \mathcal P_2(\R^d),\qquad \overline{\mathcal F}(\mu) \leqslant \int_{\R^d} \Theta\po \left|\na \ln \frac{\mu}{\Gamma(\mu)}\right|^2\pf \dd \mu\,,
\end{equation}
introduced in \cite{bobkov2000brunn} and studied further in \cite{bobkov2005entropy,barthe2008modified,fathi2015modified,bou2024note}, particularly with $\Theta(r) = r^q$ for $q<1$ (in the linear case where $\mathcal F(\mu) = \mathcal H(\mu|\rho_*)$ for some $\rho_*$).

For instance,  in \cite[Proposition 15]{MonmarcheReygner}, the inequality~\eqref{eq:degenPL} is  proven to hold at the critical temperature with $\Theta(r) = C(r+r^{1/3})$ for some $C>0$ for~\eqref{eq:F_CurieWeiss} with quadratic attractive interaction in a double well potential. As in the non-degenerate case, in \cite{MonmarcheReygner}, the inequality~\eqref{eq:degenPL} is in fact proven by establishing a (degenerate) $\Theta$-free energy condition
\begin{equation}
\label{eq:degenFreeEnergCOnd}
\forall \mu \in \mathcal P_2(\R^d),\qquad \overline{\mathcal F}(\mu) \leqslant \Theta\po 2\kappa  \mathcal H(\mu|\Gamma(\mu))\pf\,,
\end{equation}
and then simply using that $\Gamma(\mu)$ satisfies a LSI with constant $\kappa^{-1}$. Combined with~\eqref{eq:dissipation},  \eqref{eq:degenPL} implies that, along the flow~\eqref{eq:EDP},
\begin{equation}
\label{loc:Fdecrdegen}
\partial_t \overline{\mathcal F} (\rho_t) \leqslant -  \Theta^{-1}\po \overline{\mathcal F}(\rho_t)\pf  \,, 
\end{equation}
with $\Theta^{-1}$ the generalized inverse of $\Theta$. If $\Theta(r)$ is of order $r^\alpha$ for some $\alpha\in(0,1)$ for small $r$, this gives a long-time convergence rate of order $t^{-\alpha/(1-\alpha)} $. If the inequality is defective then $\Theta^{-1}(r)=0$ for $r$ small enough and it doesn't imply that $\overline{\mathcal F}(\rho_t)$ goes to zero.

The next result provides many example where~\eqref{eq:degenPL} holds, since it relates this functional inequality with the finite-dimensional $\Theta$-ŁI inequality
\begin{equation}
\label{eq:PLVkappa-weak}
\overline{V}_\kappa(m):= V_\kappa(m) -  \inf V_\kappa \leqslant \Theta \po    |\na V_\kappa(m)|^2 \pf \,.
\end{equation}
For example, if $V_\kappa(m) =|m-m_*|^4$ for some $m_*\in\R^d$, then $|\na V_\kappa(m)| = 4|m-m_*|^3$ and then \eqref{eq:PLVkappa-weak} holds with $\Theta(r) =r^{2/3}/16 $, and this is optimal (in particular, according to Proposition~\ref{prop:equivalenceInegal}, the corresponding free energy $\mathcal F$ does not satisfy a standard PL inequality). More generally,~\eqref{eq:PLVkappa-weak} always holds with $\Theta = \Theta_1$ where
\begin{equation}
\label{loc:Theta1}
\Theta_1(r) = \sup\{V_\kappa(m) - \inf V_\kappa\ :\  |\na V_\kappa(m)|^2 \leqslant r\}\,,
\end{equation}
which is nondecreasing. Moreover  $\Theta_1(0)=0$ if all critical points of $V_\kappa$ are global minimizers, and then $\Theta_1$ is continuous at $0$ if  $\liminf_{|m|\mapsto \infty} |\na V_\kappa(m)|>0$ (this  condition being implied by Assumption~\ref{assu:general}, where we also get that $\limsup_{r\rightarrow \infty}\Theta_1(r)/r<\infty$). Moreover, in this situation, if $V_\kappa$ is analytical, then $\Theta_1(r)$ scales like $r^{\alpha}$ for some $\alpha>0$ from \cite[Proposition 1]{law1965ensembles}. 

\begin{proposition}\label{prop:PLdegenere}
Under Assumption~\ref{assu:general}, for a nondecreasing function $\Theta:\R_+\rightarrow\R_+$, the following holds:
\begin{enumerate}
\item If $V_\kappa$ satisfies the $\Theta$-ŁI inequality~\eqref{eq:PLVkappa-weak} and $r \mapsto (\Theta(r) - \frac{r}{2\kappa})_+$ is non-decreasing, then $\mathcal F$ satisfies the  $\widetilde{\Theta}$-free energy condition~\eqref{eq:degenFreeEnergCOnd} with $\widetilde{\Theta}(r) = \max \po \Theta(r),\frac{r}{2\kappa}\pf$.
\item The $\Theta$-free energy condition~\eqref{eq:degenFreeEnergCOnd} implies the $\Theta$-ŁI inequality~\eqref{eq:degenPL} and that, necessarily, $\Theta (r) \geqslant r/(2\kappa)$ for all $r\geqslant 0$.
\item If $\mathcal F$ satisfies the $\Theta$-ŁI inequality~\eqref{eq:degenPL} then $V_\kappa$ satisfies the $\Theta$-ŁI inequality~\eqref{eq:PLVkappa-weak}. Moreover, necessarily, 
\begin{equation}
\label{loc:rdkappa}
\forall r \geqslant d\kappa\,,\qquad \Theta(r) \geqslant \frac{r}{2\kappa} - \frac{d}{2}\po 1 + \ln\po \frac{r}{d\kappa}\pf\pf\,,
\end{equation}
and in particular $\liminf_{r\rightarrow\infty} \Theta(r)/r \geqslant (2\kappa)^{-1}$. 
\end{enumerate}
In particular, if $r \mapsto \Theta(r) - \frac{r}{2\kappa}$ is non-decreasing, then the three inequalities are equivalent. As a consequence, the three inequalities holds with $\Theta$ given by
\[\Theta(r) = \Theta_1(0) + \int_0^r \max \po \Theta_1'(s), \frac{1}{2\kappa}\pf \dd s\,,\]
where $\Theta_1$ is given in~\eqref{loc:Theta1} (and is almost everywhere differentiable as a nondecreasing function).
\end{proposition}
\begin{proof}
The first point is proven similarly to $(ii)\Rightarrow(iv)$ in Proposition~\ref{prop:equivalenceInegal}, in particular using~\eqref{loc3}.

The second point follows from the LSI satisfied by $\Gamma(\mu)$. The fact that $\Theta(r) \geqslant r/(2\kappa)$ is obtained by applying~\eqref{eq:degenFreeEnergCOnd}  with $\mu  = \mathcal N(m_*,\sigma^2I)$ with any $\sigma^2>0$, so that $\overline{\mathcal F}(\mu) = \mathcal H(\mu|\rho_*)$ and $\mathcal H(\mu|\Gamma(\mu)) = \mathcal H(\mu|\rho_*)$ (and, moreover, $\mathcal H(\mu|\rho_*)$ can take any positive value when $\sigma^2$ varies in $(0,\infty)$).

For the third point, as in the proof of $(ii)\Rightarrow (i)$ in  Proposition~\ref{prop:equivalenceInegal}, if we assume~\eqref{eq:degenPL}, applying it with $\mu=\gamma_m$ for $m\in\R^d$ gives~\eqref{eq:PLVkappa-weak} while applying it with $\mu = \mathcal N(m_*,\sigma^2I)$ for any $\sigma^2>\kappa^{-1}$ and using~\eqref{eq:locgaussHI} shows that 
\[\frac{d}{2}  \co \sigma^2 \kappa - 1   - \ln (\kappa \sigma^2)\cf   \leqslant \Theta\po    d \sigma^2 \po \kappa -\sigma^{-2}\pf^2\pf \leqslant \Theta\po    d \sigma^2  \kappa^2\pf \,. \]
Setting $r=d \sigma^2  \kappa^2$ gives~\eqref{loc:rdkappa}.
\end{proof}

Now, the question of the $N$ particle system in the degenerate case leads to a less standard situation than the uniform-in-$N$ LSI from Proposition~\ref{prop:equivalenceInegal}. Indeed, reasoning as in the proof of the latter, it is clear that \eqref{eq:degenPL} could be deduced from an inequality of the form
\begin{equation} \label{eq:weakLSIchelou}
\forall \nu^N \in \mathcal P_2(\R^{dN}),\qquad \frac1N \mathcal H\po \nu^N |\mu_\infty^N\pf  \leqslant   \Theta\po \frac1N  \mathcal I(\nu^N|\mu_\infty^N) + \delta_N\pf + \delta_N\,,
\end{equation}
where $\delta_N \rightarrow 0$ as $N\rightarrow 0$ (possibly $\delta_N=0$ in case of a tight inequality), with a function $\Theta$ independent from $N$. This is different from weak LSI such as considered for instance in \cite{cattiaux2007weak,toscani2000trend,erdogdu2021convergence} and references within, which involve other quantities depending on $\nu^N$ (such as moments or $L^p$ norms). These inequalities have been introduced and used to study cases where the Gibbs measure has fat tails (compared to Gaussian distributions) and does not satisfies a LSI. Our situation here is different: for each $N$ the Gibbs measure does satisfy a standard LSI, however the constant is not uniform in $N$ (see Corollary~\ref{cor:GibbsLSIdegenerate} below) and thus dividing the LSI by $N$ and letting $N\rightarrow \infty$ does not give anything at the limit.  Similarly, modified LSI of the form~\eqref{eq:qLSI} have been used to strengthen LSI for lighter-than-Gaussian tails, but this question of the scaling in $N$ doesn't seem to have been considered. 

One approach to establish~\eqref{eq:weakLSIchelou} could be to deduce it from~\eqref{eq:degenPL} by  adapting to this degenerate case the conditional approximation method of \cite{SongboToAppear}. This might be an interesting direction to address this question, but we leave this study to future work.  Instead, in the following, we will focus on the objective to establish the tight version of~\eqref{eq:weakLSIchelou} in the typical degenerate situation where $V_\kappa$ is strictly convex everywhere except at its minimizer $m_*$ in the vicinity of which the smallest eigenvalue of  $\na^2 V_\kappa(m)$ behaves like some power of $|m-m_*|$. First, in Section~\ref{subsubsec:StandardLSIdegenerate}, we will discuss what can be said concerning the standard LSI in this situation, relating the growth of the log-Sobolev constant to the level of degeneracy of $\na^2 V_\kappa$ at $m_*$ (see Corollary~\ref{cor:GibbsLSIdegenerate}). Second, in Section~\ref{subsubsec:degenerateLSIdegenerate}, we will show that, however, a degenerate LSI of the form~\eqref{eq:weakLSIchelou} with -- crucially -- the correct scaling in $N$ can be established for a non-linear $\Theta$ (see Corollary~\ref{cor:degenLSImuGibbs}).

\subsubsection{Standard LSI for the Gibbs measure in some degenerate cases}\label{subsubsec:StandardLSIdegenerate}

For a degenerate situation as discussed above, relying on Lemma~\ref{lem:LSI}, we can get (sharp) polynomial bounds (in $N$) for the LSI constant of $\mu_\infty^N$, applying for instance the following:

\begin{proposition}\label{prop:LSIdegen}
 Let $u \in \mathcal C^2(\R^d,\R)$ be such that it admits a unique global minimizer $x_*\in\R^d$ and there exists $c_1 \geqslant c_2>0$ and $\beta\geqslant 2$ such that  $\na^2 u(x)  \geqslant \min(c_1,c_2|x-x_*|^{\beta-2})$ for all $x\in \R^d$.  Then, for all $N\geqslant1/c_2$,  the probability measure  $\nu_N \propto e^{-Nu}$ satisfies a LSI with constant $2e  (Nc_2)^{-2/\beta}$.
 \end{proposition}
 
\begin{remark}
In the case $\beta=2$, we retrieve the Bakry-Emery criterion up to a factor $2e$. In the general case, the scaling in $N$ is sharp. Indeed, $u(x)=|x|^\alpha$ with $\alpha\geqslant 2$ satisfies the assumptions of the proposition with $\beta= \alpha$. In that case, $\nu_N$ is the image of $\nu_1$ by a homothety of ratio $N^{-1/\alpha}$, so that the optimal LSI constant of $\nu_N$ is $N^{-2/\alpha} C_{LS}(\nu_1)$. See also Proposition~\ref{prop:lowerboundCLSIdegenerate} below for a more general result which shows the sharpness of the rate.
\end{remark} 

\begin{example}\label{ex:degen}
In the running example~\eqref{eq:VmuEncoder}, as computed in Example~\ref{ex:hessienne}, when $\kappa=\lambda_n$, then $0$ is the global minimizer but $\na^2 V_\kappa(0)$ is singular (see Figure~\ref{fig:Vkappa}\emph{(c)}). In this case, if $\lambda_n$ is a simple eigenvalue, the assumptions of Proposition~\ref{prop:LSIdegen} holds with $\beta=4$. Moreover, more generally, a similar situation arises locally (in the sense discussed in Section~\ref{subsec:local_ineq}) whenever $\kappa$ equals a simple eigenvalue of $M$.
\end{example}
 
 \begin{proof}
 Without loss of generality we assume that $x_*=0$. For any $r\geqslant0$, $\na^2 u(x) \geqslant  \lambda_r \1_{|x|\geqslant r} $ with $\lambda_r=\min(c_1,c_2r^{\beta-2})$. As computed in \cite[Section 3.3]{MonmarcheBruit}, for any choice of $r\geqslant0$, we can decompose
 \[u(x) = v(x) + g(x)\]
where $v$ is $ \lambda_r/2$-strongly convex  and $\max g - \min g \leqslant r^2\lambda_r$. By the Bakry-Emery and Holley-Stroock criteria, $\nu_N$  satisfies a LSI with constant
 \[ C_N = \frac2{N \lambda_r}   \exp\po N r^2 \lambda_r\pf  \,. \] 
The result follows by applying this with $r= (Nc_2)^{-1/\beta}$ (since for $N\geqslant 1/c_2$ we get that $r^{\beta-2} \leqslant 1 \leqslant c_1/c_2$).
  
 \end{proof}

The next statement shows that as soon as the potential $u$ admits a polynomially degenerate local minimizer, then the  log-Sobolev constant of $\nu_N\propto e^{-Nu}$ cannot go to zero at  rate  $1/N$.
 
 \begin{proposition}\label{prop:lowerboundCLSIdegenerate}
 Let $u\in\mathcal C^2(\R^d,\R)$. Assume that $u$ admits an isolated local minimizer $x_*\in\R^d$ and that there exist $a,\delta,C,\alpha>0$, $\gamma>\beta>2$ and  $v\in\R^d $ with $|v|=1$  such that, for all $x\in\mathcal B(x_*,\delta)$,
 \begin{equation}
 \label{loc:ua1}
 |u(x) - a |v\cdot (x-x_*)|^\beta - u(x-(v\cdot (x-x_*))v)| \leqslant C|v\cdot x|^\gamma\quad\text{and}\quad u(x)-u(x_*) \leqslant C |x|^\alpha\,.
 \end{equation}
 Then, considering the probability density $\nu_N\propto e^{-Nu}$, there exists $c>0$ such that for all $N\geqslant 1$, $C_{LS}(\nu_N) \geqslant cN^{-\frac2{\beta}}$.
 
 \end{proposition}
\begin{proof}
Without loss of generality, up to translation and rotation we assume that $x_*=0$, $u(0)=0$ and $v$ is the first vector of the canonical basis, in which case, writing $\tilde u(x) = a_1 |x_1|^\beta + u(0,x_2,\dots,x_d)$, the first part of~\eqref{loc:ua1} reads
 \begin{equation}
 \label{loc:ua1_decale}
 |u(x) - \tilde u(x)| \leqslant C|x_1|^\gamma \qquad \forall x\in\mathcal B(0,\delta)\,,
 \end{equation}
 and the second part implies that $ u(0,x_{\neq 1}) \leqslant |x_{\neq 1}|^\alpha$ when $|x_{\neq 1}|\leqslant \delta$ (where we write $x_{\neq 1} =(x_2,\dots,x_d)$). 
 Since the LSI implies a Poincaré inequality with the same constant (see \cite[Proposition 5.1.3]{BakryGentilLedoux}), for any $f\in H^1(\nu_N)$, 
 \[C_{LS}(\nu_N) \geqslant \frac{\int_{\R^d} f^2 \nu_N - \po \int_{\R^d} f \nu_N  \pf^2  }{\int_{\R^d}|\na f|^2 \nu_N }\,.\]
 It only remains to apply this with a suitably designed $f$ and estimate the right hand side with Laplace method. We  take $\delta'\in(0,\delta)$ such that $\inf\{u(x)\ :\ \varepsilon \leqslant |x| \leqslant \delta'\} >0$ for any $\varepsilon\in(0, \delta']$ (this exists since $0$ is an isolated minimizer) and set
 \[f(x) = x_1  \prod_{i=1}^d \chi( x_i)\]
 where $\chi$ is a smooth function with $\chi(s)=1$ when $|s|\leqslant \delta'/(2\sqrt{d})$  and $\chi(s)=0$ when $|s|\geqslant \delta'/\sqrt{d}$. In particular, $f(x)=0$ when $|x|\geqslant \delta'$. Let $\eta \in(\gamma^{-1},\beta^{-1})$ and, for $N\geqslant 1$, write
 \[D_N = \left\{x\in \R^d\ :\ |x_1|\leqslant N^{-\eta},\ |x_j| \leqslant \frac{\delta'}{2\sqrt{d}}\ \forall j\in\cco 2,d\ccf\right\}\,, \]
 which has the following properties:  for $N$ large enough, $D_N \subset \mathcal B(0,\delta')$ and $f(x) = x_1$ for all $x\in D_N$, and 
 \begin{equation}
 \label{loc:infsup}
 \inf_{\mathcal B(0,\delta')\setminus D_N} u \geqslant \frac{a}{2} N^{-\eta\beta}\,,\qquad \sup_{D_N}|u-\tilde u| = \underset{N\rightarrow \infty}{\mathcal O}(N^{-\gamma \eta})\,.  
 \end{equation}
 Then, for any function $g\in\mathcal C(\R^d,\R$) with support in $\mathcal B(0,\delta')$,
 \begin{equation}\label{loc:neglig}
  \int_{\R^d} g   e^{-Nu}     =   \int_{\mathcal B(0,\delta')} g e^{-Nu} =   \int_{D_N} g e^{-Nu} +  \underset{N\rightarrow \infty}{\mathcal O}\po e^{-(1-\eta\beta)N}\pf \,,
 \end{equation}
 with $1-\eta\beta>0$ by choice of $\eta$.  For either $g=1$,  $g=f^2$ or $g=|\na f|^2$ (which are nonnegative and, on $D_N$, depends only on $x_1$), using the second part of~\eqref{loc:infsup} with $\gamma\eta>1$,
 \[\int_{D_N} g e^{-Nu} = \int_{D_N} g e^{-N\tilde u}\po 1+ \underset{N\rightarrow\infty}o(1)\pf = c_N \int_{-N^{-\eta}}^{N^{-\eta}} g(x_1) e^{-N|x_1|^\beta} \dd x_1 \po 1+ \underset{N\rightarrow\infty}o(1)\pf \,,\]
 where, 
 \[c_N = \int_{[\pm \delta'/(2\sqrt{d})]^{d-1}} e^{-Nu(0,x_{\neq 1})} \dd x_{\neq 1} \geqslant \int_{[\pm \delta'/(2\sqrt{d})]^{d-1}} e^{-NC|x_{\neq 1}|^\alpha} \dd x_{\neq 1} \underset{N\rightarrow\infty} \sim N^{-\frac{d}\alpha} \int_{\R^{d-1}} e^{-C|y|^\alpha} \dd y\,.  \]
Besides,
\begin{equation}\label{loc:laplace}
\int_{-N^{-\eta}}^{N^{-\eta}} e^{-N|x_1|^\beta} \dd x_1 \underset{N\rightarrow\infty}\sim N^{-\frac1{\beta}} b_0\,,  \qquad \int_{-N^{-\eta}}^{N^{-\eta}} x_1^2  e^{-N|x_1|^\beta} \dd x_1 \underset{N\rightarrow\infty}\sim N^{-\frac3{\beta}} b_2\,, 
\end{equation}
with $b_k = \int_{\R} x_1^k e^{-|x|^\beta}\dd x_1 >0$ for $k=0,2$. This shows that, for $g=1$, $g=f^2$ or $g=|\na f|^2$, the second term in the right hand side of~\eqref{loc:neglig} is negligible with respect to the first one. At this point, we have obtained that
\[\int_{\R^d} f^2 \nu_N \underset{N\rightarrow\infty}\sim N^{-\frac2{\beta}} \frac{b_2}{b_0}\,,\qquad \int_{\R^d} |\na f|^2 \nu_N \underset{N\rightarrow\infty}\longrightarrow 1\,.\]
To control $\int_{\R^d} f \nu_N$, we use~\eqref{loc:infsup} to bound
\[\left|\int_{D_N} f e^{-Nu }\right| = \left|\int_{D_N} f \po e^{-Nu} - e^{-N\tilde u} \pf\right| \leqslant \int_{\R^d} |f| e^{-Nu} \times \underset{N\rightarrow \infty}{\mathcal O}(N^{-\gamma \eta})\,,  \]
hence
\[ \left| \int_{\R^d} f \nu_N \right| \leqslant \int_{\R^d} |f| \nu_N  \times \underset{N\rightarrow \infty}{\mathcal O}(N^{-\gamma \eta})  \leqslant \sqrt{\int_{\R^d} |f|^2 \nu_N}  \times \underset{N\rightarrow \infty}{\mathcal O}(N^{-\gamma \eta})\,. \]
 As a conclusion,
  \[C_{LS}(\nu_N) \geqslant \frac{\int_{\R^d} f^2 \nu_N - \po \int_{\R^d} f \nu_N  \pf^2  }{\int_{\R^d}|\na f|^2 \nu_N } \underset{N\rightarrow\infty}\sim N^{-\frac2{\beta}} \frac{b_2}{b_0}\,,\]
  which concludes the proof of the proposition.
 
\end{proof}

Combining Propositions~\ref{prop:LSIdegen} and \ref{prop:lowerboundCLSIdegenerate} with Lemma~\ref{lem:LSI} yields the following:

\begin{corollary}\label{cor:GibbsLSIdegenerate}
Under Assumption~\ref{assu:general}:
\begin{itemize}
\item If $u=V_\kappa$ satisfies the conditions in Proposition~\ref{prop:LSIdegen}, then there exists $C>0$ such that $C_{LS}(\mu_\infty^N) \leqslant C N^{1-\frac{2}{\beta}} $  for all $N\geqslant 1$.
\item If $u=V_\kappa$ satisfies the conditions in Proposition~\ref{prop:lowerboundCLSIdegenerate}, then there exists $c>0$ such that $C_{LS}(\mu_\infty^N) \geqslant c N^{1-\frac{2}{\beta}}  $  for all $N\geqslant 1$.
\end{itemize}
\end{corollary}

In the case $\beta=4$, we recover the scaling of order $\sqrt{N}$ as in the spectral gap analysis in \cite{becker2020spectral} of some mean-field models at critical temperature.

\subsubsection{Degenerate LSI}\label{subsubsec:degenerateLSIdegenerate}

The goal of this section is to establish inequalities of the form~\eqref{eq:weakLSIchelou} (with $\delta_N=0$), with the correct scaling in $N$, in situations such as considered in Corollary~\ref{cor:GibbsLSIdegenerate}. Let us first adapt Lemma~\ref{lem:LSI} to this situation.

\begin{lemma}
\label{lem:LSIdegenerate}
Under Assumption~\ref{assu:general}, let $N\geqslant 1$. Assume that the probability density $\nu_\infty^N  \propto \exp(- NV_\kappa)$ satisfies 
\begin{equation}
\label{loc:nuinfintydegenerate}
\forall \mu \ll \nu_\infty^N\,,\qquad  \frac1N \mathcal H\po \mu|\nu_\infty^N\pf \leqslant \Theta \po \frac1{N^2} \mathcal I \po \mu|\nu_\infty^N\pf\pf \,,
\end{equation}
for some non-decreasing $\Theta:\R_+\rightarrow\R_+$. Then the  Gibbs measure $\mu_\infty^N$ defined in~\eqref{eq:GIbbsdef} satisfies
\[\forall \mu^N \ll \mu_\infty^N\,,\qquad  \frac1N \mathcal H\po \mu^N|\mu_\infty^N\pf \leqslant \widetilde{\Theta} \po \frac1{N} \mathcal I \po \mu^N|\mu_\infty^N\pf\pf \,,\]
with $\widetilde \Theta(r) = \sup_{s\in[0,r]} \Theta(s) + \frac{r-s}{2\kappa}$.

\end{lemma}

\begin{proof}
As in the proof of Lemma~\ref{lem:LSI}, this is a direct consequence of the fact that  $\mu_\infty^N$ is obtained by an orthogonal transformation from the tensor product of a Gaussian distribution with variance $\kappa I_{d(N-1)}$ and of the image by the dilatation by a factor $\sqrt{N}$ of $\nu_\infty^N$, together with the respective behavior of the relative entropy (unchanged) and the Fisher information (multiplied by $N$) along this dilatation, and the tensorization property of inequalities the form~\eqref{loc:nuinfintydegenerate}, that we state below in Lemma~\ref{lem:tensor}.
\end{proof}

\begin{lemma}
\label{lem:tensor}
For $i=1,2$, let $d_i \geqslant 1$, $\nu_i$ be a probability measure over $\R^{d_i}$ and $\Theta_i:\R_+\rightarrow\R_+$ be such that
\[\forall \mu \ll \nu_i\,,\qquad \mathcal H(\mu|\nu_i) \leqslant \Theta_i\po \mathcal I(\mu|\nu_i)\pf \,.\]
Assume that $\Theta_1$ is nondecreasing and that $\Theta_2$ is concave. Then,
\[\forall \mu \ll \nu_1\otimes \nu_2\,,\qquad \mathcal H(\mu|\nu_1\otimes\nu_2) \leqslant \Theta\po \mathcal I(\mu|\nu_1\otimes\nu_2)\pf\,, \]
with $\Theta(r) = \sup\{\Theta_1(s) + \Theta_2(r-s),\ s\in[0,r]\}$.
\end{lemma}

\begin{remark}
Since we do not assume that $\Theta_i(0)=0$, this covers the case of defective inequalities as~\eqref{eq:weakLSIchelou} with $\delta_N\neq 0$.
\end{remark}

\begin{proof}
The argument is the same as the in the classical case (e.g. \cite[Proposition 5.2.7]{BakryGentilLedoux}). Let $\mu \ll \nu_1\otimes \nu_2$. Write $\mu_1$ its first $d_1$-dimensional marginal, and $\mu_{2|1}$ the conditional density of the second variable given the first one. Then
\begin{align*}
\mathcal H(\mu|\nu_1\otimes\nu_2) &= \mathcal H(\mu_1|\nu_1) + \int_{\R^{d_1}} \mathcal H(\mu_{2|1}|\nu_2) \dd \nu_1 \\
& \leqslant \Theta_1\po \mathcal I(\mu_1|\nu_1) \pf +  \int_{\R^{d_1}} \Theta_2\po \mathcal I(\mu_{2|1}|\nu_2)\pf  \dd \nu_1 \\
&\leqslant \Theta_1\po \mathcal I_1(\mu|\nu_1 \otimes \nu_2) \pf +  \Theta_2\po \mathcal I_2(\mu|\nu_1 \otimes \nu_2) \pf\,,
\end{align*}
where we introduced the notations
\[\mathcal I_i(\mu|\nu_1 \otimes \nu_2)  = \int_{\R^{d_1+d_2}} \left|\na_{x_i} \ln \frac{\mu}{\nu_1\otimes \nu_2}\right|^2 \dd \mu\]
and used that $\Theta_2$ is concave,that $\int \mathcal I(\mu_{2|1}|\nu_2) \nu_1 = \mathcal I_2(\mu|\nu_1\otimes\nu_2)$,  that $\mathcal I(\mu_1|\nu_1)  \leqslant \mathcal I_1(\mu|\nu_1 \otimes \nu_2) $ by Cauchy-Schwarz and that $\Theta_1$ is nondecreasing. Conclusion then follows from the fact that $\mathcal I(\mu|\nu_1 \otimes \nu_2)  = \mathcal I_1(\mu|\nu_1 \otimes \nu_2)  + \mathcal I_2(\mu|\nu_1 \otimes \nu_2) $.
\end{proof}

To establish an inequality of the form~\eqref{loc:nuinfintydegenerate} in degenerate convex situations, we will rely on the following, which is a particular case of \cite[Corollary
12]{MonmarcheWang} (applied with $L=0$ and $\sigma^2=1$).

\begin{proposition}
\label{prop:critereDefectLSI}
Let $u\in\mathcal C^2(\R^d,\R)$ be concave and such that there exist $x_*\in\R^d$, $R\geqslant 0$ and $\rho>0$ such that 
\begin{equation}
\label{loc:cond1}
\forall x,y\in\R^d,\qquad |x-x_*|\geqslant R\qquad \Rightarrow \qquad \po \na u(x) - \na u(y)\pf \cdot (x-y) \geqslant \rho |x-y|^2\,.
\end{equation}
Then the probability density $\nu \propto e^{-u}$ satisfies the defective LSI
\begin{equation}
\label{loc:defLSI}
\forall \mu \ll \nu\,,\qquad  \mathcal H(\mu|\nu) \leqslant A \mathcal I(\mu|\nu) + B
\end{equation}
with
\[A = \frac{12}{\rho} \,,\qquad B= 6 \ln\po 1+4d + 2\rho R^2\pf + \frac{3}{4} \max \po 1+4d , 2\rho R^2\pf\,.  \]
\end{proposition}

For consistency between our different results, we deduce~\eqref{loc:cond1} as a consequence of the assumption of Proposition~\ref{prop:LSIdegen}:

\begin{lemma}
\label{lem:conditions}
Let $u$ be as in Proposition~\ref{prop:LSIdegen}. Then, for all $R>0$,~\eqref{loc:cond1} holds with 
\begin{equation}
\label{loc:rhoR}
\rho = \rho_R := \frac13 \min(c_1,c_2(R/2)^{\beta-2})\,. 
\end{equation}
\end{lemma}
\begin{proof}
Let $R>0$ and $x,y\in\R^d$ with $|x-x*|\geqslant R$. Let $\rho_R' = \min(c_1,c_2(R/2)^{\beta-2})$. Since $\mathcal B(x_*,R/2)$ (over which $\na^2 u$ is bounded below by $\rho_R'$) has a diameter $R$ and is at distance at least $R/2$ from $x$, the proportion of $t\in[0,1]$ such that $z_t:=x+t(x-y) \notin \mathcal B(x_*,R/2)$ is at least one third (for $y_t\in\mathcal B(x_*,R/2$, we can simply use that $\na^2 u(z_t) \geqslant 0$), from which we bound
\[\po \na u(x) - \na u(y)\pf \cdot (x-y) =  \int_0^1 (x-y)\cdot \na^2 u(z_t) (x-y) \dd t  \geqslant \frac{\rho_R'}3 |x-y|^2\,, \]
as announced.
\end{proof}

\begin{proposition}
\label{prop:LSIdegenNnu}
Let $u$ be as in Proposition~\ref{prop:LSIdegen}, with $\beta>2$. Then, for all $N\geqslant 3(1+4d)/(8c_2)$, the probability density $\nu_N \propto e^{-Nu}$ satisfies:
\[\forall \mu \ll \nu_N\,,\qquad \frac1N \mathcal H(\mu|\nu_N) \leqslant \Theta \po \frac1{N^2} \mathcal I(\mu|\nu_N)\pf\,,\]
with $\Theta(r) = C \max \po r,r^{\frac{\beta}{2\beta-2}}\pf$ where, writing $c_3=3(1+4d)/(2^{3-\beta}c_2)$,
\[C = \max\left\{ \frac{2^{\beta+4}}{c_2},2^{4-\beta}c_2, e c_2^{-\frac2\beta}   c_3^{\frac{\beta-2}\beta}, \frac{36}{c_2 } + \frac{2^{8-2\beta} }3c_2 \right\}\,.   \]
\end{proposition}
\begin{proof}
Thanks to Lemma~\ref{lem:conditions}, for all $R>0$ and $N\geqslant 1$, $Nu$ satisfies \eqref{loc:cond1} with $\rho = N\rho_R$, $\rho_R$  being given in~\eqref{loc:rhoR}. Proposition~\ref{prop:critereDefectLSI} shows that for all $\mu \ll \nu_N$ and $R\leqslant 2$,
\begin{align*}
\mathcal H(\mu|\nu_N) &\leqslant 36\frac{2^{\beta-2}}{c_2 NR^{\beta-2}} \mathcal I(\mu|\nu_N) +  6 \ln\po 1+4d + \frac{2^{3-\beta}}3c_2 N R^\beta\pf + \frac{3}{4} \max \po 1+4d ,  \frac{2^{3-\beta}}3c_2 N R^\beta\pf \\
& \leqslant C_0\po \frac{1}{ NR^{\beta-2}} \mathcal I(\mu|\nu_N) +   N R^\beta \pf\,,
\end{align*}
when $N R^\beta \geqslant 3(1+4d)/(2^{3-\beta}c_2) =:c_3$, where we used that $6\ln(2x) + \frac34 x \leqslant 4x$ for $x\geqslant 5$ and introduced the constant $C_0= \max(36\times 2^{\beta-2}/c_2,2^{5-\beta}c_2/3) $. 

For a given $\mu \ll \nu_N$, we distinguish three cases. First, if 
\[ 2^{2\beta-2} N^2 \geqslant  \mathcal I(\mu|\nu_N)    \geqslant N^{\frac2\beta} c_3^{\frac{2\beta-2}\beta}     \,,  \] 
we can apply the  inequality with $R=(\frac{1}{ N^2} \mathcal I(\mu|\nu_N))^{\frac1{2\beta-2}}$, which satisfies both that $N R^{\beta}\geqslant c_3$ and $R\leqslant 2$,   leading to
\[
\frac1N \mathcal H(\mu|\nu_N)  \leqslant 2C_0 \po \frac{1}{ N^2} \mathcal I(\mu|\nu_N)\pf ^{\frac\beta{2\beta-2}}  \,.
\]
Second, if $\mathcal I(\mu|\nu_N) \geqslant 2^{2\beta-2} N^2 $, we apply Proposition~\ref{prop:critereDefectLSI} with $R=2$, obtaining
\begin{align*}
\mathcal H(\mu|\nu_N) &\leqslant \frac{36}{c_2 N} \mathcal I(\mu|\nu_N) +   \frac{32}3c_2 N  \leqslant  \frac1N \po  \frac{36}{c_2 } + \frac{2^{8-2\beta} }3c_2 \pf \mathcal I(\mu|\nu_N)  \,,
\end{align*}
when $N  \geqslant 3(1+4d)/(8c_2)$. Third, if $ \mathcal I(\mu|\nu_N)    \leqslant N^{\frac2\beta} c_3^{\frac{2\beta-2}\beta}$, applying the LSI for $\nu_N$ established in Proposition~\ref{prop:LSIdegen} gives, for any $\alpha\in(0,1)$,
\[\frac1N \mathcal H(\mu|\nu_N) \leqslant e c_2^{-\frac2\beta}  N^{-1-\frac2\beta} \mathcal I(\mu|\nu_N) \leqslant e c_2^{-\frac2\beta} N^{-1-\frac2\beta} \po N^{\frac2\beta} c_3^{\frac{2\beta-2}\beta}\pf^{1-\alpha}   \mathcal I^\alpha (\mu|\nu_N) \,. \]
We choose $\alpha$ so that
\[N^{\frac2\beta(1-\alpha)-1-\frac2\beta} = N^{-2\alpha}\,, \]
leading to $\alpha =\beta/(2\beta-2)$. As a conclusion, for such $\mu$, 
\[\frac1N \mathcal H(\mu|\nu_N) \leqslant e c_2^{-\frac2\beta}   c_3^{\frac{\beta-2}\beta}  \po \frac1{N^2}  \mathcal I (\mu|\nu_N) \pf^{\frac{\beta}{2\beta-2}} \,, \] 
which concludes the proof.
\end{proof}

\begin{remark}
We see that, in the proof, to tackle the lack of strong convexity,  high values of $\mathcal I(\mu|\nu_N)$ (of the order at least $N^{\frac{2}{\beta}}$) were treated thanks to the defective inequality~\eqref{loc:defLSI}, while  smaller values of $\mathcal I(\mu|\nu_N)$ were treated with the tight LSI established in Proposition~\ref{prop:LSIdegen} thanks to the Holley-Stroock perturbation approach. 
\end{remark}

\begin{remark}\label{rem:lighttail}
This is not the topic of the present work but let us notice that the argument used here also provides improved LSI for lighter-than-Gaussian distributions (even for $N=1$). Indeed, assuming that, outside a compact set, $\na^2 u(x) \geqslant c |x|^{\beta-2}$ with $\beta>2$ and $c>0$, Proposition~\ref{prop:critereDefectLSI} shows that $\nu \propto e^{-u}$ satisfies
\[\forall \mu \ll \nu\,,\qquad \mathcal H(\mu|\nu) \leqslant C \po  R^{2-\beta} \mathcal I(\mu|\nu) + R^{\beta}\pf \,,\]
for some $C>0$ for all $R$ large enough. Applying this with $R = \po\mathcal I(\mu|\nu)\pf^{\frac{1}{2\beta-2}}$ for large values of the Fisher information (and using the standard LSI for smaller values) leads to 
\[\forall \mu \ll \nu\,,\qquad \mathcal H(\mu|\nu) \leqslant \Theta\po \mathcal I(\mu|\nu) \pf \,,\]
with $\Theta(r)$ is linear for small values but behaves as $r^{\frac{\beta}{2\beta-2}} = o(r)$ as $r\rightarrow \infty$. Differentiating the relative entropy along the Fokker-Planck equation associated to $\nu$ (i.e. the Wasserstein flow of $\mathcal H(\cdot|\nu)$) gives a bound
\[\forall t>0\,, \rho_0 \in\mathcal P_2(\R ^d),\qquad  \mathcal H(\rho_t |\nu) \leqslant \frac{C}{\max\po 1,t^{\frac{\beta}{\beta-2}}\pf } e^{-ct}\]
where $C,c>0$ are constants independent from $\rho_0$ (in particular, $\mathcal H(\rho_t |\nu)$ becomes finite instantaneously when $t>0$, with a bound uniform over all probability measures $\rho_0$). These sort of bounds have already been studied under modified LSI \cite{bobkov2000brunn,bobkov2005entropy,barthe2008modified,fathi2015modified,bou2024note}, stronger than  the Łojasiewicz  inequalities considered here.
\end{remark}

As a conclusion, combining Propositions~\ref{lem:LSIdegenerate} and \ref{prop:LSIdegenNnu}, we have obtained the following:

\begin{corollary}
\label{cor:degenLSImuGibbs}
Under Assumption~\ref{assu:general}, if $u=V_\kappa$ satisfies the conditions in Proposition~\ref{prop:LSIdegen}, then the  Gibbs measure $\mu_\infty^N$ defined in~\eqref{eq:GIbbsdef} satisfies
\[\forall \mu^N \ll \mu_\infty^N\,,\qquad  \frac1N \mathcal H\po \mu^N|\mu_\infty^N\pf \leqslant \Theta \po \frac1{N} \mathcal I \po \mu^N|\mu_\infty^N\pf\pf \,,\]
with $\Theta(r) = (C+1/(2\kappa)) \max\po r,r^{\frac{\beta}{2\beta-2}}\pf $ with $C$ defined in Proposition~\ref{prop:LSIdegenNnu}.
\end{corollary}

\subsection{Degenerate minimizer: coercivity}\label{sec:degenCoer}

Similarly to Section~\ref{sec:degenPL} which was the adaptation of Section~\ref{subsec:non-degen-contract} to the case of degenerate minimizers, the present section is the adaptation to this situation of the coercivity analysis from Section~\ref{subsec:coerci-non-degen}. In this section, $\rho_* = \gamma_{m_*}$ is a global minimizer of $\mathcal F$.

For a nondecreasing function $\Phi :\R_+ \rightarrow \R_+ $ with $\Phi(0)=0$ and $\Phi(r)>0$ for $r>0$, we say that $\mathcal F$ satisfies a $\Phi$-entropic coercivity inequality if
\begin{equation}
\label{eq:degenerateCoercivityF}
\forall \mu \in\mathcal P_2(\R^d),\qquad \overline{\mathcal F}(\mu) \geqslant \Phi\po \mathcal H(\mu|\rho_*)\pf\,,
\end{equation}
and  a $\Phi$-$\mathcal W_2$ coercivity inequality if
\begin{equation}
\label{eq:degenerateCoercivityFW2}
\forall \mu \in\mathcal P_2(\R^d),\qquad \overline{\mathcal F}(\mu) \geqslant \Phi\po \frac{\kappa}{2} \mathcal W_2^2(\mu,\rho_*)\pf\,,
\end{equation}
(which can also be interpreted as a generalized Talagrand inequality). Similarly, we say that $V_\kappa$ satisfies  a $\Phi$-coercivity inequality if
\begin{equation}
\label{eq:degenerateCoercivityV}
\forall m\in\R^d,\qquad V_\kappa(m) - V_\kappa(m_*) \geqslant \Phi\po \frac{\kappa}{2}|m-m_*|^2 \pf\,.
\end{equation}
This last inequality is always satisfied with 
\begin{equation}
\label{loc:Phi1}
\Phi_1(r) = \inf\{V_\kappa(m)-V_\kappa(m_*)\ :\ \kappa |m-m_*|^2 \geqslant 2r\}\,,
\end{equation}
which is nondecreasing, and is positive on $(0,\infty)$ provided $m_*$ is the unique global minimizer of $V_\kappa$ and $\liminf_{|m|\rightarrow\infty} V_\kappa(m) > V_\kappa(m_*)$ (this last condition being satisfied under Assumption~\ref{assu:general}, where we also get that $\liminf_{r\rightarrow\infty}\Phi_1(r)/r >0$).

\begin{proposition}\label{prop:coercivity_degenerate}
Under Assumption~\ref{assu:general}, let $m_*\in\R^d$ be a global minimizer of $V_\kappa$ and $\rho_*=\gamma_{m_*}$. For a nondecreasing function $\Phi :\R_+ \rightarrow \R_+ $ with $\Phi(0)=0$ and $\Phi(r)>0$ for $r>0$, the following holds:
\begin{enumerate}
\item If~\eqref{eq:degenerateCoercivityF} holds, then so does~\eqref{eq:degenerateCoercivityFW2}, and necessarily $\Phi(r) \leqslant r$ for all $r\geqslant 0$. 
\item If \eqref{eq:degenerateCoercivityFW2} holds then so does~\eqref{eq:degenerateCoercivityV}, and necessarily
\begin{equation}
\label{loc:bjnj}
\Phi(r)  \leqslant  (\sqrt{d/2}+\sqrt{r})^2\,,
\end{equation}
and in particular $\limsup_{r\rightarrow\infty} \Phi(r)/r\leqslant 1$.
\item If~\eqref{eq:degenerateCoercivityV} holds then $\mathcal F$ satisfies a $\hat \Phi$-entropic coercivity inequality with $\hat \Phi(r) = \inf\{r -s + \Phi(s),\ s\in[0, r]\}$.
\end{enumerate}
In particular, if $r\mapsto \Phi(r) -r$ is nonincreasing, then the three inequalities \eqref{eq:degenerateCoercivityF}, \eqref{eq:degenerateCoercivityFW2} and \eqref{eq:degenerateCoercivityV} are equivalent. As a consequence, they all hold with $\Phi$ given by
\[\Phi(r) = \int_0^r \min(1,\Phi_1'(s))\dd s \,, \]
where $\Phi_1$ is given in~\eqref{loc:Phi1} (and is almost everywhere differentiable as a nondecreasing function).
\end{proposition} 

This result generalizes Proposition~\ref{prop:coercivity}, which corresponds to $\Phi(r)=cr$ for $c\in(0,1]$.

\begin{proof}
First point: from~\eqref{eq:degenerateCoercivityF}, \eqref{eq:degenerateCoercivityFW2}  follows from the T2 inequality satisfied by $\rho_*$. Considering $\mu \neq \rho_*$ with $m_\mu= m_*$ as in the proof of Proposition~\ref{prop:coercivity} shows that, in \eqref{eq:degenerateCoercivityF}, necessarily, $\Phi(r) \leqslant r$ for all $r\geqslant 0$. 

Second point: applying~\eqref{eq:degenerateCoercivityFW2} with $\mu=\gamma_m$ shows that
\[V_\kappa(m) - V_\kappa(m_*) = \overline{\mathcal F}(\mu) \geqslant \Phi\po \frac{\kappa}{2} \mathcal W_2^2(\mu,\rho_*)\pf= \Phi\po \frac{\kappa}{2}|m-m_*|^2\pf \,,\]
hence~\eqref{eq:degenerateCoercivityV}. Moreover, the bound~\eqref{loc:bjnj} is deduced by applying~\eqref{eq:degenerateCoercivityFW2} with $\mu = \mathcal N(m_*,\sigma^2)$ with $\sigma^2 = (1+\sqrt{2r/d})^2/\kappa$, so that, similarly to~\eqref{loc:d2},
\[\Phi(r) = \Phi\po \frac{\kappa}{2}\mathcal W_2^2(\mu,\rho_*)\pf \leqslant \overline{\mathcal F}(\mu)  \leqslant \frac{d}{2}\co \kappa \sigma^2 - 1 - \ln(\kappa \sigma^2)\cf  \leqslant  \frac{d\kappa \sigma^2}{2}\,. \]

Third point: assuming~\eqref{eq:degenerateCoercivityV}, recalling~\eqref{locgh}, \eqref{loc1bis} (applied with $m'=m_*$), we get 
\begin{align*}
\mathcal F(\mu) - \mathcal F(\rho_*) 
 &=  V_\kappa(m) - V_\kappa(m_*) +  \mathcal H\po \mu|\rho_* \pf - \frac{\kappa}{2}|m_\mu-m_*|^2 \\
 & \geqslant  \Phi\po \frac{\kappa}{2}|m_{\mu}-m_*|^2 \pf - \frac{\kappa}{2}|m_\mu-m_*|^2  +  \mathcal H\po \mu|\rho_* \pf \ \geqslant \   \hat \Phi\po \mathcal H(\mu|\rho_*)\pf \,,
\end{align*}
thanks to the definition of $\hat \Phi$ and the fact that  $\frac{\kappa}2|m_\mu-m_*|^2 \leqslant \mathcal H(\mu|\rho_*)$ due to \eqref{loc1bis}.

The last point of the proposition is that, if $r\mapsto \Phi(r)-r$ is non-increasing, then $\hat\Phi(r)=\Phi(r)$.

\end{proof}

The case of the $N$ particle system is similar to the non-degenerate situation described in Proposition~\ref{prop:coercivityN}. Recall the definition~\eqref{eq:FNnuN} of  the modulated $N$-particle free energy.

\begin{proposition}
\label{prop:coercivityN_degenerate}
Under Assumption~\ref{assu:general}, let $m_*\in\R^d$ be a global minimizer of $V_\kappa$ and $\rho_*=\gamma_{m_*}$. Let $\Phi$ be as in Proposition~\ref{prop:coercivity_degenerate}. Assume furthermore that $\na^2 V$ is bounded, and that $r\mapsto r - \Phi(r)$ is nondecreasing and concave. Then, \eqref{eq:degenerateCoercivityF} is equivalent to the fact that for all  $\varepsilon>0$, $N\geqslant 1$ and all density $\nu^N \in \mathcal P_2(\R^{dN})$, 
\begin{multline}
\frac1N \mathcal F^N(\nu^N|\rho_*) \geqslant \Phi \po \frac1N  \mathcal H \po \nu^N |\rho_*^{\otimes N}\pf\pf  -  2 \frac{\|\na^2 V\|_\infty}{\kappa}  \po  \varepsilon + \frac{1+\varepsilon^{-1}}{N}\pf \frac1N  \mathcal H \po \nu^N |\rho_*^{\otimes N}\pf\\
    -   \frac{\|\na^2 V\|_\infty}{\kappa N}d (1+\varepsilon   + \varepsilon^{-1})   \,.\label{eq:coercivityN_degenere}
\end{multline}
\end{proposition}

\begin{remark}
It is always possible to ensure that~\eqref{eq:degenerateCoercivityF} is satisfied with $\Phi$ such that $r\mapsto r - \Phi(r)$ is nondecreasing and concave, by taking 
\begin{equation}
\label{locPhir}
\Phi(r) = \int_0^r \min(1,\inf_{u\geqslant s} \Phi_1'(u))\dd s \,. 
\end{equation}
\end{remark}

\begin{proof}
We follow the proof of Proposition~\ref{prop:coercivityN} (and use its definitions and notations, for instance $\bar\nu$). The fact that~\eqref{eq:coercivityN_degenere} implies~\eqref{eq:degenerateCoercivityF}  works exactly  as the proof of the implication $(iv)\Rightarrow (i)$ in Proposition~\ref{prop:coercivityN}, to which we refer. For the converse implication, as in the proof of Proposition~\ref{prop:coercivityN}, we can write~\eqref{eq:degenerateCoercivityF} equivalently as the fact that
\begin{equation}
\label{loc:reinterpret_bis}
(\mathrm{Id}-\Phi)\po  \mathcal H(\mu|\rho_*) \pf + \int \ln \rho_* \mu + \mathcal E(\mu) - \mathcal F(\rho_*) \geqslant 0\,.
\end{equation}
Using~\eqref{eq:chainruleH} and that $\mathrm{Id}-\Phi$ is nondecreasing and concave gives
\[\po \mathrm{Id}-\Phi\pf \po \frac1N \mathcal H \po \nu^N |\rho_*^{\otimes N}\pf \pf \geqslant \po \mathrm{Id}-\Phi\pf \po \mathbb E \po \mathcal H \po \bar\nu |\rho_*\pf\pf\pf \geqslant \mathbb E \co \po \mathrm{Id}-\Phi\pf \po  \mathcal H \po \bar\nu |\rho_*\pf\pf\cf \,.\]
Dividing~\eqref{loc:FN} by $N$ and reasoning as in \eqref{loc:654} thanks to~\eqref{loc:reinterpret_bis} we get
\[
\frac1N \mathcal F^N(\nu^N|\rho_*) \geqslant  \Phi \po \frac1N \mathcal H \po \nu^N |\rho_*^{\otimes N}\pf\pf  +   \int_{\R^{dN}} V(\bar x) \nu^N (\dd \bx) -  \mathbb E \po V(m_{\bar\nu})\pf \,.\]
The last two terms are then bounded exactly as in the proof of Proposition~\ref{prop:coercivityN}, 
\end{proof}

\begin{remark}
Under Assumption~\ref{assu:general} and assuming that $m_*$ is the unique global minimizer of $V_\kappa$, $\Phi$ given in~\eqref{locPhir} is such that for all $\delta>0$ there exists $c_\delta>0$ such that $\Phi(r) \geqslant c_\delta r$ for all $r\geqslant \delta$,  so that~\eqref{eq:coercivityN_degenere} differs from~\eqref{eq:coercivityN} only for small values of  $H_N:=\frac1N  \mathcal H \po \nu^N |\rho_*^{\otimes N}\pf $. Assuming that $\Phi(r) \sim r^{1+\alpha}$ for some $\alpha >0$ when $r\rightarrow 0$, we may apply~\eqref{eq:coercivityN_degenere} with $\varepsilon$ given by a fixed fraction of $H_N^\alpha$ to get that
\[
\frac1N \mathcal F^N(\nu^N|\rho_*) \geqslant c H_N^{1+\alpha} -  \frac{C}{N H_N^\alpha} \]
as soon as $H_N \leqslant \delta$ for some $c,C,\delta>0$ independent from $N$. As a consequence,  by distinguishing cases, for $H_N\leqslant \delta$,
\[H_N \leqslant \max \po  \po \frac2{cN} \mathcal F^N(\nu^N|\rho_*)\pf^{\frac1{1+\alpha}} \, ,\, \po \frac{2C}{cN}\pf^{\frac{1}{1+2\alpha}}  \pf\,. \]
This can then be used as in \cite[Section 3.2]{SongboToAppear} to get uniform in time propagation of chaos for initial conditions distributed in a Wasserstein ball around the global minimizer, with a deteriorated rate with respect to the case $\alpha=0$. Indeed, the proof of \cite[Proposition 21]{MonmarcheReygner} is easily adapted to degenerate cases to get
\[\frac1N \mathcal F^N(\rho_t^N|\rho_*) \leqslant s(N) + w(t)  \] 
with $\rho_t^N$ the law of the $N$ particles and $s$ and $w$ vanish as $N\rightarrow \infty$ and $t\rightarrow 0$ (specifically, $w$ is the long-time convergence rate of the non-linear flow to $\rho_*$, obtained from~\eqref{loc:Fdecrdegen}). Applying this with $t$ given by a small fraction of $\ln N$ and combining this (for smaller values of $t$) with classical finite-time propagation of chaos, the long-time convergence of the non-linear flow $(\rho_t)_{t\geqslant 0}$ and the Talagrand inequality for $\rho_*^{\otimes N}$, we get
\[\sup_{t\geqslant 0} \mathcal W_2^2(\rho_t^N,\rho_t^{\otimes N}) \underset{N\rightarrow\infty}\longrightarrow 0\,.\]
We do not detail this further by lack of a specific application of interest in mind (indeed, the situation considered in Section~\ref{sec:CurieWeiss} is simpler since there is no other critical points than the global minimizer, and we are able to get tight coercivity inequalities; an academic example with a degenerate minimizer and another non-global local minimizer is easily designed).
\end{remark}

Finally, let us notice that, as in the non-degenerate case (and the finite-dimensional situation \cite{law1965ensembles}), a coercivity inequality is implied by a Łojasiewicz inequality (as stated in \cite[Theorem 1]{BLANCHET20181650}  for local inequalities; for completeness and consistency, we state it with our notations and for global inequalities):

\begin{proposition}\label{prop:degenPL->coer}
Under Assumption~\ref{assu:general}, assume furthermore a tight $\Theta$-ŁI inequality~\eqref{eq:degenPL} for some non-decreasing $\Theta:\R_+ \rightarrow \R_+$ continuous at $0$ with $\Theta(0)=0$. Writing $\Theta^{-1}$ its generalized inverse, for $r\geqslant 0$, set
\begin{equation}
\label{loc:g}
g(r) = \int_0^r \po \Theta^{-1}(u)\pf^{-1/2}\dd u\,,
\end{equation}
assuming integrability at $0$. Then,  along the gradient flow~\eqref{eq:gradientEDP}, for all $t\geqslant 0$ and $\rho_0\in\mathcal P_2(\R^d)$,
\[\mathcal W_2(\rho_t,\rho_0) \leqslant  g\po  \overline{\mathcal F}(\rho_0)\pf - g\po  \overline{\mathcal F}(\rho_t)\pf \,.\]
In particular, if $\rho_t$ converges in $\mathcal W_2$ to some $\rho_*\in \mathcal P_2(\R^d)$ as $t\rightarrow\infty$ then
\begin{equation}
\label{loc:WW2}
\mathcal W_2(\rho_*,\rho_0) \leqslant  g\po  \overline{\mathcal F}(\rho_0)\pf \,.
\end{equation}
As a consequence, if  $\rho_*$ is an isolated  minimizer of $\mathcal F$ then the $\Phi$-$\mathcal W_2$ coercivity inequality~\eqref{eq:degenerateCoercivityFW2} holds with $\Phi$ given by $\Phi(\frac{\kappa}{2} r^2)= g^{-1}(r)$ for $r\geqslant 0$.
\end{proposition}

\begin{remark}\label{rem:general}
The proof is in fact completely general,  doesn't rely on the particular form of the free energy~\eqref{eq:F} and, besides, works for any gradient flow in a metric space, see  \cite{BLANCHET20181650}. However, for a free energy of the form~\eqref{eq:F}, combining Propositions~\ref{prop:coercivity_degenerate} and~\ref{prop:degenPL->coer} also gives a $\Phi$-entropic coercivity inequality~\eqref{eq:degenerateCoercivityF}, which is less clear in general.  
\end{remark}

\begin{remark}\label{rem:orderCoer}
If $V_\kappa(m)-V_\kappa(m_*)$ and $|\na V_\kappa(m)|$ respectively behave as $|m-m_*|^\beta$ and $|m-m_*|^{\beta-1}$ for some $\beta\geqslant 2$ as $m\rightarrow m_*$, then $\Theta(r)$ behaves like $r^{\frac{\beta}{2(\beta-1)}}$ when $r\rightarrow 0$, so that  $\po\Theta^{-1}(u)\pf^{-1/2}$ is of order $u^{1/\beta-1}$ at zero, hence integrable. In this situation, at zero, $g(r)$ is of order  $r^{1/\beta}$ and then $\Phi(r)$ is of order $r^{\beta/2}$.
\end{remark}

\begin{remark}
If $V_\kappa$ has a set $\mathcal M$ of non-isolated global minimizers, under a tight $\Theta$-ŁI we still get that $\rho_t$ converges to some global minimizer as $t\rightarrow \infty$ for all initial conditions, and thus letting $t\rightarrow \infty$ in~\eqref{loc:WW2} we get that for all $\rho_0\in\mathcal P_2(\R^d)$,
\[\mathcal W_2(\mathcal M,\rho_0) \leqslant  g\po  \overline{\mathcal F}(\rho_0)\pf \,,\]
as in the classical finite-dimensional case~\cite{law1965ensembles}.
\end{remark}

\begin{proof}
Along the gradient flow~\eqref{eq:gradientEDP}, writing $f_t = \overline{\mathcal F}(\rho_t)$,
\begin{eqnarray*}
\partial_t g\po f_t\pf & = & - g'(f_t) \mathcal I\po \rho_t|\Gamma(\rho_t)\pf  \\
& \leqslant & -  g'(f_t)  \sqrt{ \Theta^{-1}(f_t) \mathcal I\po \rho_t|\Gamma(\rho_t)\pf}\\
& = & - \sqrt{ \mathcal I\po \rho_t|\Gamma(\rho_t)\pf}\,.
\end{eqnarray*}
As in the proof of \cite[Lemma 4]{MonmarcheReygner}, the Benamou-Brenier formula~\cite{BenamouBrenier} gives
\[\mathcal W_2(\rho_t,\rho_0) \leqslant \int_0^t \sqrt{ \mathcal I\po \rho_s|\Gamma(\rho_s)\pf} \dd s \leqslant g(f_0) - g(f_t)\,,\]
which proves the first claim. The second is obtained by using that $g(f_t) \geqslant 0$ and letting $t\rightarrow\infty$ in the previous inequality. The last claim  follows from the second one since a tight Łojasiewicz inequality implies that, first, critical points of $\mathcal F$ are necessarily global minimizers (from which $\rho_*$ is the unique critical point and is the global minimizer of $\mathcal F$) and, second, $f_t \rightarrow 0$ for all initial condition (which, by the fact $\rho_*$ is isolated, implies that $\rho_t\rightarrow\rho_*$ for all $\rho_0$).
\end{proof}

The particular case of Proposition~\ref{prop:degenPL->coer} with $\mathcal F(\mu) = \mathcal H(\mu|\nu)$ for some fixed $\nu$ generalizes the classical fact that a (standard) LSI implies a (standard) Talagrand inequality. Combining it with Corollary~\ref{cor:degenLSImuGibbs} immediately gives the following:

\begin{corollary}
\label{cor:PL->degenT2}
Under the settings of Corollary~\ref{cor:degenLSImuGibbs}, $\mu_\infty^N$ satisfies:
\begin{equation*}
\forall \mu^N \ll \mu_\infty^N\,,\qquad  \frac1{N} \mathcal W_2^2\po \mu^N,\mu_\infty^N\pf \leqslant \xi \po \frac1{N} \mathcal H \po \mu^N|\mu_\infty^N\pf\pf \,,
\end{equation*}
with $\xi(r) = C \max(r,r^{\frac{2}{\beta}})$ for some $C>0$ depending only on $c_2,\beta$ and $\kappa$.
\end{corollary}

\begin{proof}
Thanks to  Corollary~\ref{cor:degenLSImuGibbs}, this is proven by applying Proposition~\ref{prop:degenPL->coer} (with Remark~\ref{rem:general}) to the free energy $\mathcal F(\mu^N) = \frac1N\mathcal H(\mu^N|\mu_\infty^N)$ (in dimension $dN$), whose gradient flow associated to the rescaled Wasserstein distance $N^{-1/2}\mathcal W_2$ is exactly the Fokker-Planck equation associated to the Langevin process~\eqref{eq:particulesEDS2}, with dissipation $\partial_t \mathcal F(\rho_t^N) = -\frac{1}{N}\mathcal I(\mu^N|\mu_\infty^N)$. Working with these suitably scaled quantities exactly gives
\[\forall \mu^N \ll \mu_\infty^N\,,\qquad  \frac1{\sqrt{N}} \mathcal W_2\po \mu^N,\mu_\infty^N\pf \leqslant g \po \frac1{N} \mathcal H \po \mu^N|\mu_\infty^N\pf\pf \,,
\]
 with the function $g$ defined in~\eqref{loc:g} from the function $\Theta$ introduced in Corollary~\ref{cor:degenLSImuGibbs}. Straightforward computations shows that $g(r) \leqslant C' \max(r^{1/2},r^{\frac{1}{\beta}})$ for some $C'$ depending only on $c_2,\kappa$ and $\beta$. We conclude by taking the square in the inequality.
\end{proof}

\subsection{Local inequalities}\label{subsec:local_ineq}

When $V_\kappa$ (hence $\mathcal F$) have several local minima, global inequalities as stated in the previous sections (and in particular Proposition~\ref{prop:equivalenceInegal}) cannot hold. However, it is still possible for local inequalities to hold, leading to local convergence rates. This has been first developed in~\cite{MonmarcheReygner} for the non-linear flow~\eqref{eq:EDP} and in~\cite{MonmarcheMetastable} for the particle system. More precisely, we say that $\mathcal F$ (resp. $V_\kappa$)  satisfies a PL inequality locally at $\rho_*\in\mathcal P_2(\R^d)$ (resp. at  $m_*\in\R^d$) if there exists $C,\delta>0$ such that
\begin{equation}
\label{eq:locPLF}
\forall \mu \in\mathcal B(\rho_*,\delta),\qquad \mathcal F(\mu) - \inf_{\mathcal B(\rho_*,\delta)} \mathcal F = \mathcal F(\mu)  - \mathcal F(\rho_*) \leqslant \frac{C}{2} \mathcal I \po \mu|\Gamma(\mu)\pf\,,
\end{equation} 
respectively
\begin{equation}
\label{eq:locPLV}
\forall m \in\mathcal B(m_*,\delta),\qquad V_{\kappa}(m) - \inf_{\mathcal B(m_*,\delta)} V_\kappa = V_\kappa(m)  - V_\kappa(m_*) \leqslant \frac{C}{2}|\na V_\kappa(m)|^2\,. 
\end{equation}
Similarly, the local version of the free energy and coercivity inequalities~\eqref{eq:ineqEntropy} and~\eqref{eq:coercivityF} is obtained by replacing in each case $\overline{\mathcal F}$ by $\mathcal F(\mu) - \inf_{\mathcal B(\rho_*,\delta)} \mathcal F$ and requiring the inequality to hold only for $\mu \in \mathcal B(\rho_*,\delta)$.

The same adaptation gives the local versions of the degenerate PL and coercivity inequalities discussed in Sections~\ref{sec:degenPL} and \ref{sec:degenCoer}, but we won't discuss in details these situations.

For the ``local" version of the uniform LSI~\eqref{eq:LSI} for the Gibbs measure $\mu_\infty^N$, however, we will follow~\cite{MonmarcheMetastable}, and thus it does not consist in requiring that the inequality only holds for measures $\mu$ in some strict subset of $\mathcal P_2(\R^{dN})$. Rather, the goal is to find a modified energy $\widetilde{\mathcal E} :\mathcal P_2(\R^d) \rightarrow \R$ such that $\mathcal E(\mu) = \widetilde{\mathcal E}(\mu)$ for all $\mu \in \mathcal B(\rho_*,\delta)$ and such that the associated Gibbs measure $\tilde \mu_\infty^N$ satisfies a uniform LSI. The point is then that the particle systems $\bX$ and $\tilde{\bX}$ solving \eqref{eq:particulesEDS1} respectively with $\mathcal E$ and $\widetilde{\mathcal E}$ exactly coincide as long as their empirical measures $\pi_{\bX}$ and $\pi_{\tilde{\bX}}$ remains in $\mathcal B(\rho_*,\delta)$, which holds for very long times (of order $e^{cN}$ for some $c>0$). As a consequence, useful large-time estimates can be obtained for the original process from the uniform LSI satisfied by the modified Gibbs measure.

Now, this definition of ``local" inequalities involving a modified energy $\tE$ allows to recover the localized version~\eqref{eq:locPLF}. Indeed, the associated free energy $\tF$ and local equilibrium $\widetilde{\Gamma}$ are such that $\tF(\mu) = \mathcal F(\mu)$ and $\widetilde{\Gamma}(\mu) = \Gamma(\mu)$ for all $\mu\in\mathcal B(\rho_*,\delta)$.  As a consequence, if $\rho_*$ is a global minimizer of $\tF$ and if $\tF$ satisfies a global PL inequality, then it implies the local PL for $\mathcal F$. 

The immediate ``local" corollary of Propositions~\ref{prop:equivalenceInegal} and~\ref{prop:coercivity} is the following.

\begin{proposition}\label{prop:equivLocale}
Let $\rho_* = \gamma_{m_*}$ be an isolated local minimizer of $\mathcal F$. The following are equivalent:
\begin{enumerate}
\item $\na^2 V_\kappa(m_*)$ is positive definite.
\item $\mathcal F$ satisfies a  PL inequality locally at $\rho_*$.
\item $\mathcal F$ satisfies a free energy inequality locally at $\rho_*$.
\item $\mathcal F$ satisfies a  coercivity inequality  locally at $\rho_*$.
\item $\mathrm{Hess}\mathcal F_{|\rho_*}$ given in \eqref{eq:HessFm} is a positive symmetric operator over $H^1(\rho_*)$ with a positive spectral gap.
\end{enumerate}
Moreover, in this situation, there exists a modified energy $\tE$ coinciding with $\mathcal E$ on $\mathcal B(\rho_*,\delta)$ for some $\delta>0$   such that the associated Gibbs measure $\tilde\mu_\infty^N$  satisfies a uniform LSI.
\end{proposition}

\begin{remark}\label{rem:tE}
The modified $\tE$ designed in the proof is still of the form $\tE(\mu) = \int_{\R^d} |x|^2 \mu(\dd x) + \tilde V(m_\mu)$ for some $\tilde V\in \mathcal C^2(\R^d)$. In particular, $\mathcal E(\mu) = \tE(\mu)$ as soon as $m_\mu \in \mathcal B(m_*,\delta)$. Moreover, in this proof, we take $\tilde V$ to be strongly convex.
\end{remark}

\begin{proof}
Since $V_\kappa$ is $\mathcal C^2$ and $m_*$ is a local minimizer, the fact that $\na^2 V_\kappa(m_*)$ is positive definite is equivalent to the local coercivity condition, and to the local PL inequality (the fact that the PL inequality implies that $\na^2 V_\kappa(m_*)$ is non-singular is \cite[Proposition 7]{ChewiStromme}, and the converse is simply that $V_\kappa$ is then locally strongly convex, which implies the PL inequality). Following the proofs of Proposition~\ref{prop:equivalenceInegal} and \ref{prop:coercivity}, the local PL and coercivity inequalities for $V_\kappa$ are equivalent to the same inequalities for $\mathcal F$ and to the free energy inequality (using indeed that $ |m_\mu - m_\nu| \leqslant \mathcal W_2(\mu,\nu) $ for all $\mu,\nu \in\mathcal P_2(\R^d)$ and that $\mathcal W_2(\gamma_m,\gamma_{m'}) = |m-m'|$ for all $m,m'\in\R^d$).

Concerning the Hessian of the free energy, using the notations from Section~\ref{subsec:hessian}, thanks to the orthogonal decomposition~\eqref{eq:Hessorthodecompos} and the coercivity inequality~ \eqref{eq:HessCOnvexEortho} over $E_\ell^\perp$, $\mathrm{Hess}\mathcal F_{|\rho_*}$ is positive with a positive spectral gap if and only if its restriction on the finite dimensional space $E_\ell$ is positive definite, which according to~\eqref{eq:HessEl} is exactly saying that $\na V_\kappa(\rho_*)$ is definite positive.

If $\na^2 V_\kappa$ is positive definite, we can find $\delta>0$, a compactly-supported $\mathcal C^2$ function $\chi$ with $\chi(m)=m$ for $m\in\mathcal B(m_*,\delta)$ and $L>0$ such that $\tilde V_\kappa$ given by $\tilde V_\kappa(m) = \chi(m)V_\kappa(m) + L (|m-m_*|^2 - \delta^2)_+^2 $ is strongly convex. Set $\tilde V(m) = \tilde V_\kappa(m) - \frac{\kappa}2|m|^2$ and consider the corresponding energy $\tE(\mu) = \tilde V(\mu) + \int_{\R^d} |x|^2 \mu(\dd x)$. Since $\tilde V_\kappa$ is strongly convex, it satisfies a (global) PL inequality. Proposition~\ref{prop:equivalenceInegal} then shows that the associated Gibbs measure $\tilde \mu_\infty^N$ satisfies a uniform LSI.
\end{proof}

\section{Exit and transition times}\label{sec:exittimes}

\subsection{Motivations and settings}

As discussed in the introduction,  the dynamics of the empirical distribution of an interacting particle system simulated in practice to approximate the gradient flow of a free energy is a small random perturbation of the deterministic flow. As such, when the free energy admits several critical points, the particle system is expected to exhibit a metastable behavior, as in the finite dimensional case~\eqref{gradientepsilon}.  One aspect of this behavior is described by the exit time from the vicinity of critical points, and transition times between local minimizers. These questions are well understood in finite dimension, as we will discuss below. Since we will focus on exit and transition times involving domains defined only in terms of the barycentre $\bar x = \frac1N\sum_{i=1}^N x_i$, our study will boil down to the finite dimensional case (as discussed in Section~\ref{sec:barycentre}). In other words, this section simply consists in stating classical results in finite dimension and  interpreting them in terms of the particle system associated to the free energy in the case of the toy model~\eqref{eq:F}.  The reason of the restriction to such domains is the following: in Section~\ref{subsec:local_ineq}, the domains of $\mathcal P_2(\R^d)$ on which  local inequalities hold are entirely determined by the expectation of the measures, and similarly whether $\mathcal E(\mu)=\tE(\mu)$ depends only on $m_\mu$. As we will see in Section~\ref{sec:practical_conclusion}, exit times from domains where local inequalities hold plays a key role in obtaining large-time local convergence of practical interest for the particle system.

An alternative question would be instead to focus on the fluctuation and large deviations of the empirical distribution of the centered Ornstein-Uhlenbeck~\eqref{eq:particles-centered}. Nevertheless, we won't discuss this.

The question of exit and transition times for mean field systems traces back at least to the finite-time Large Deviations results of Dawson and G{\"a}rtner \cite{dawson1986large}. However, this situation is still much less understood than the finite-dimensional case. In particular, the Arrhenius law conjectured in \cite[Theorem 4]{dawson1986large} remains to be proven. 

In the remaining of this section, we will successively consider two questions: first, in Section~\ref{subsec:transitiontimes}, we will study the transition time from one minimizer to another, obtaining the Arrhenius law (Theorem~\ref{thm:Arrhenius}), the Eyring-Kramers formula (Theorem~\ref{thm:EK}) and the convergence to the exponential distribution (Theorem~\ref{thm:expo}).  Second,  Section~\ref{subsec:exittimes} will be devoted to the exit event from a saddle point (Theorem~\ref{thm:exit}).

For all these questions, in the finite-dimensional case, there has been a number of works which allow for instance for degenerate critical points or complex geometrical situations.  Such references will be provided in each case. Since our analysis boils down to the finite dimensional case, obviously all the results from these references can be applied in our situation. However, for our purpose (which is to understand what kind of result may be expected in other mean-field models), it is sufficient to focus only on the simplest situation, with only three critical points, all non-degenerate: two minizers, and one index-1 saddle point. Hence, for the sake of simplicity and clarity, we will only state results under these following conditions:

\begin{assumption}
\label{Assu:exit-imes}
The potential $V_\kappa$ has exactly three critical points, $x_0,x_1,z\in\R^d$. For $i=0,1$, $\na^2 V_\kappa(x_i)$ is positive definite while $\na^2 V_\kappa(z)$ is non-singular with exactly one negative eigenvalue $-\bar{\lambda}_1<0$ associated to an eigenspace of dimension $1$, spanned by $v_1\in\R^d$. The infimum of $\sup_{[0,1]} V_\kappa\circ \gamma$ over all continuous path $\gamma:[0,1]\rightarrow \R^d$ such that $\gamma(0)=x_0$ and $\gamma(1)=x_1$ is $V_\kappa(z)$. For $y\in\{x_0,x_1,z\}$ we write $B_y = \{\bx \in \R^{dN},\ \bar x \in \mathcal B(y,\delta)\}$ where $\delta>0$ is small enough so that, first, the three balls $\mathcal B(y,\delta)$ for $y\in\{x_0,x_1,z\}$ do not intersect, second, a trajectory of the gradient flow $\dot y_t = -\na V_\kappa(y_t)$ initialised at $y_0 \in \mathcal B(x_i,\delta)$ with $i=0,1$ converges to $x_i$ and, third, the unstable manifold of the gradient flow at $z$ intersects the boundary of $\mathcal B(z,\delta)$ at exactly two points  (cf. Section~\ref{subsec:exittimes}). We consider a particle system $\bX$ solving~\eqref{eq:particles} with initial distribution $\nu_0^N$ and denote by $\bar\nu_0$ the initial distribution of the barycentre $\bar X_0$. 
\end{assumption}

\begin{example}
As a consequence of Example~\ref{ex:hessienne2}, in the running example, Assumption~\ref{Assu:exit-imes} is satisfied when $\lambda_n > \kappa > \lambda_{n-1}$, with $(\lambda_k)_{k\in\cco 1,n\ccf}$ the eigenvalues of $M$ sorted in increasing order (with multiplicities), if $\lambda_n$ is a simple eigenvalue. In that case, the two minimizers are $m^{(\lambda_n)}$ and $-m^{(\lambda_n)}$ and the saddle point is $0$ (see Figure~\ref{fig:Vkappa}\emph{(a)(b)}).
\end{example}

For a measurable domain $D\subset \R^{dN}$, we  write
\[\tau_D = \inf\{t\geqslant 0,\ \bX_t \in D\}\,.\]
Recall that, as discussed in Section~\ref{sec:barycentre}, $(\bar X_t)_{t\geqslant 0}$ is then an autonomous diffusion process solution the small-noise equation~\eqref{eq:EDSbarycentre}. For a domain $D=\{\bx\in\R^{dN},\ \bar x \in D'\}$ for some $D'\subset \R^d$,
\begin{equation}
\label{eq:tauD}
\tau_D = \bar{\tau}_{D'} := \inf\{t\geqslant 0,\ \bar{X}_t \in D'\}\,.
\end{equation}

\subsection{Transition time between minimizers}\label{subsec:transitiontimes}

The Arrhenius law states that, at low temperature $\varepsilon$, the time for the noisy gradient descent~\eqref{gradientepsilon} to leave a local energy well is exponentially large with $\varepsilon^{-1}$, with a rate given by the depth of the well, i.e. the minimal energy gap to overcome to leave the well. In finite dimensional settings, this is one of the results of the Freidlin-Wentzell theory~\cite{freidlin1998random}, which applies in much more general settings than Assumption~\ref{Assu:exit-imes}.

\begin{theorem}\label{thm:Arrhenius}
Under Assumption~\ref{Assu:exit-imes}, for all $\varepsilon>0$, uniformly over all initial distributions $\nu_0^N$ with $\nu_0^N(B_{x_0})=1 $,
\begin{equation}
\label{eq:Arrhenius}
\mathbb P_{\nu_0^N}\po e^{N \po \mathcal F(\gamma_z) - \mathcal F(\gamma_{x_0}) - \varepsilon\pf} \leqslant \tau_{B_{x_1}} \leqslant e^{N \po\mathcal F(\gamma_z) - \mathcal F(\gamma_{x_0}) + \varepsilon\pf} \pf  \underset{N\rightarrow\infty}\longrightarrow 1 \,
\end{equation}
and
\begin{equation}
\label{loc:EK}
\frac1N \ln \mathbb E_{ \nu_0^N} \po \tau_{B_{x_1}} \pf \underset{N\rightarrow\infty}\longrightarrow \mathcal F(\gamma_z) - \mathcal F(\gamma_{x_0})\,.
\end{equation}
\end{theorem}

\begin{proof}
Since $\gamma_z$ and $\gamma_{x_0}$ are critical points of the free energy, \eqref{locgh} gives
\begin{equation}
\label{loc:FFVV}
\mathcal F(\gamma_z) - \mathcal F(\gamma_{x_0}) =  V_\kappa(z) - V_\kappa(x_0)\,. 
\end{equation}
Thanks to~\eqref{eq:tauD}, the result then follows from \cite[Theorem 6.2]{freidlin1998random} (which is stated for diffusion processes on compact manifolds but the deduction when $V_\kappa \rightarrow \infty$ at infinity easily follows by applying it to a periodized process which coincide with $\bar X_t$ until $V_\kappa(\bar X_t)$ hits some value larger than $V_\kappa(z)$, as in the proof of \cite[Lemma 6]{monmarche2021simulated}).
\end{proof}

The estimate~\eqref{loc:EK} on the expected transition times can be refined by identifying the sub-exponential prefactor, leading to the so-called  Eyring-Kramers formula~\eqref{eq:EyringKramers}. We will rely on the finite dimensional result of~\cite{bovier2004metastability}, which allows for more complex situations than Assumption~\ref{Assu:exit-imes}, and have been extended later on in \cite{lelievre2022eyring} (to allow for saddle point at the boundary of the domain of interest) or \cite{avelin2023geometric} (to allow for degenerate critical points). See also the monograph \cite{bovier2016metastability}.
 
 \begin{theorem}\label{thm:EK} 
 Under Assumption~\ref{Assu:exit-imes}, for $m\in\{x_0,z\}$, the operator $\mathrm{Hess}\mathcal F_{|\gamma_m}$ given in \eqref{eq:HessFm} has a discrete spectrum $\lambda_{1,m} \leqslant \lambda_{2,m} \leqslant \dots$ in $L^2(\gamma_m)$, with $\lambda_{1,z} = - \bar{\lambda}_1 < \lambda_{2,z}$ and $\lambda_{1,x_0}>0$. Moreover, the product
 $\prod_{k\in \N} \frac{|\lambda_{k,z}|}{\lambda_{k,x_0}}$
 converges and, uniformly over all initial distributions  $\nu_0^N$ with $\nu_0^N(B_{x_0})=1 $,
\begin{equation}
\label{eq:EyringKramers}
\mathbb E_{\nu_0^N} \po \tau_{B_{x_1}} \pf = \frac{2\pi  }{\bar{\lambda}_1} \sqrt{\prod_{k\in \N} \frac{|\lambda_{k,z}|}{\lambda_{k,x_0}}} e^{N \co \mathcal F(\gamma_z) - \mathcal F(\gamma_{x_0}) \cf } \co 1 + \mathcal O \po \frac{\ln N}{\sqrt{N}} \pf  \cf \,.
\end{equation}
 \end{theorem}

 \begin{remark}
 The form of~\eqref{eq:EyringKramers}, and in particular the presence of an infinite product of ratios of eigenvalues, is consistent with other Eyring-Kramers formula established in infinite dimension for small-white-noise perturbations of some PDE, see \cite{barret2015sharp,BerglundGentz,Berglund2} and references within, and with the notion of functional determinants \cite{gel1960integration,forman1987functional,mckane1995regularization}.
 \end{remark}
 
 \begin{proof}
 Thanks to~\eqref{eq:tauD}, applying \cite[Theorem 3.2]{bovier2004metastability} gives
 \begin{equation}
 \label{loc:EK2}
\mathbb E_{\nu_0^N} \po \tau_{B_{x_1}} \pf = \frac{2\pi \sqrt{|\mathrm{det}\na^2 V_\kappa(z)|}}{\bar{\lambda}_1\sqrt{\mathrm{det}\na^2 V_\kappa(x_0)}} e^{N \co V_\kappa(z)-V_\kappa(x_0) \cf } \co 1 + \mathcal O \po \frac{\ln N}{\sqrt{N}} \pf  \cf \,.
\end{equation}
More precisely, this result is stated with an initial condition $\bar X_0 = x_0$, so we combine it with   \cite[Theorem 1]{brassesco1998couplings}  or rather with  \cite[Theorem 5.3]{Locherbach} (which is a reformulation of the latter with a more explicit dependency on the initial distribution) to see that 
\[ \mathbb E_{\nu_0^N} \po \tau_{B_{x_1}} \pf = \mathbb E_{\tilde \nu_0^N} \po \tau_{B_{x_1}} \pf \po 1+ \mathcal O \po N^{-1} \pf\pf \]
uniformly over pairs of initial distributions  $\nu_0^N$, $\tilde \nu_0^N$ under which $\bar X_0 \in \mathcal B(x_0,\delta)$ almost surely.

It remains to see that~\eqref{loc:EK2} is exactly the formula~\eqref{eq:EyringKramers}. The exponential term has already been treated in~\eqref{loc:FFVV}. The spectral analysis of $H_m:= \mathrm{Hess}\mathcal F_{|\gamma_m}$ has been performed in Section~\ref{subsec:hessian}. In particular, $H_z$ and $H_{x_0}$ both admits the orthogonal decomposition~\eqref{eq:Hessorthodecompos} on the space of linear functions $E_\ell$ and $E_\ell^\perp$. On $E_\ell^\perp$, $H_{m}$ is conjugated (through the translation $x \mapsto x +m$)  to $-L_0$ with $L_0$ the Ornstein-Uhlenbeck operator, which has an unbounded discrete negative  spectrum with a positive spectral gap (with an eigenbasis given by Hermite polynomials, see \cite[Section 2.7.1]{BakryGentilLedoux}). In particular, the product $\prod_{k=1}^n \frac{|\lambda_{k,z}|}{\lambda_{k,x_0}}$ is eventually stationary for $n$ large enough, equal to the ratio of eigenvalues of $L_{z}$ and $L_{x_0}$ over $E_\ell$. From~\eqref{eq:HessEl}, these eigenvalues are exactly those of $\na^2 V(z)$ and $\na^2 V(x_0)$, respectively, which concludes.
 \end{proof}
 
 Finally, rescaled by its expectation, the transition times converges to the exponential distribution, which  describes the unpredictability of these metastable transitions.

 \begin{theorem}\label{thm:expo}
 Under Assumption~\ref{Assu:exit-imes}, uniformly over all initial distributions  $\nu_0^N$ with $\nu_0^N(B_{x_0})=1 $,
\begin{equation}
\label{eq:exponentiality}
\sup_{t\geqslant 0} \left| \mathbb P_{\nu_0^N} \po \tau_{B_{x_1}} > t\mathbb E(\tau_{B_{x_1}})\pf - e^{-t}\right| \underset{N\rightarrow\infty}\longrightarrow 0 \,.
\end{equation}
 \end{theorem} 
 
 \begin{proof}
 Thanks to~\eqref{eq:tauD}, this follows from  \cite[Theorem 1]{brassesco1998couplings}  (again, with  \cite[Theorem 5.3]{Locherbach} to get the uniformity over the initial distributions).
 \end{proof}

\subsection{Exit time from saddle points}\label{subsec:exittimes}

The small-noise behavior of the exit time of an elliptic diffusion process from an unstable equilibrium has first been considered by Kifer in \cite{Kifer} and since then refined in subsequent works, in particular by Bakhtin et al. \cite{bakhtin2008,bakhtin2011noisy,monter2011normal,10.1214/18-AAP1387,bakhtin2021atypical}. Although these references allow for more general situations,  here we state a simple statement where the barycentre of the initial condition is exactly at the saddle point.

We need to introduce some notations from~\cite{bakhtin2008}, to which we refer for details. Under Assumption~\ref{Assu:exit-imes}, there exists a curve $\gamma\in\mathcal C(\R, \R^d)$, invariant by the flow $\dot x_t = -\na V_\kappa(x_t)$, with $\gamma(0)=z$ and $\gamma'(0)=v$. Using that $\gamma(t) \sim z + t v$ for small $t$, we can find $\delta$ small enough so that this curve intersects  the boundary of $\mathcal B(z,\delta)$ exactly at two points, as required in Assumption~\ref{Assu:exit-imes}. We call $\{z_{1},z_{-1}\}$ these two points. Up to reparametrizing $\gamma$, we assume that $z_s=\gamma(s)$ for $s\in\{-1,1\}$. The gradient ascent $\dot y_t = \na_\kappa V(y_t)$, initialized at $z_s$, stays in the curve and converges to $z$ as $t\rightarrow 0$. For $u\in(0,1)$, denote by $T_s(u)$ the time needed for the solution started at $z_s$ to reach $\gamma(su)$. Then, as stated in \cite[Lemma 1]{bakhtin2008}, 
\[T_s = \lim_{u\rightarrow 0} \po  T_s(u) + \frac{\ln u}{\bar{\lambda}_1}\pf\]
is well-defined and finite.

 We can now describe the exit time and location of the process starting from a saddle point.

\begin{theorem}\label{thm:exit}
Under Assumption~\ref{Assu:exit-imes}, assume furthermore that $\bar X_0 =z$ a.s. under $\nu_0^N$ and that $(\frac1N|\bX_0|^2)_{N \geqslant 1}$ is almost surely bounded. Write $\tau = \tau_{B_z^c}$. Then, there exists two independent variables $Q$ and $Z$, with $Q$ uniformly distributed over $\{-1,1\}$ and $Z\sim\mathcal N(0,1)$ such that, almost surely,
\begin{equation}
\label{loc:W2exit}
\mathcal W_2 \po \pi (\bX_{\tau}),\gamma_{z_Q}\pf \underset{N \rightarrow \infty}\longrightarrow 0 
\end{equation}
and 
\begin{equation}
\label{loc:tauexit}
\tau - \frac{\ln (N/2)}{2\bar{\lambda}_1}  \underset{N \rightarrow \infty}\longrightarrow T_{Q} + \frac{\ln(|Z|/\sqrt{2\bar{\lambda}_1}) }{\bar{\lambda}_1} \,.
\end{equation}
\end{theorem}

\begin{proof}
According to~\cite[Theorem 1]{bakhtin2008}, \eqref{loc:tauexit} holds and, almost surely,
\begin{equation}
\label{loc:barXtau}
 \bar{X}_{\tau}  \underset{N \rightarrow \infty}\longrightarrow z_Q\,. 
\end{equation}
Hence, it only remains to study the centered dynamis of $X_t^{i,c}:= X_t^i-\bar X_t$ given by~\eqref{eq:particles-centered}. Indeed,  given $\mu_1,\mu_2\in\mathcal P_2(\R^d)$ and writing $\mu_i^c$ for $i=1,2$ their centered version (i.e. the law of $Z - m_{\mu_i}$ when $Z\sim \mu_i$), we can consider an optimal $\mathcal W_2$ coupling $(Y_1,Y_2)$ of $\mu_1^c$ and $\mu_2^c$ and then bound
\[\mathcal W_2^2 (\mu_1,\mu_2) \leqslant \mathbb E \po|m_{\mu_1} + Y_1 - m_{\mu_2} - Y_2|^2\pf \leqslant 2 |m_{\mu_1}-m_{\mu_1}|^2 + 2 \mathcal W_2^2(\mu_1^c,\mu_2^c)\,.\] 
As a consequence, due to~\eqref{loc:barXtau} and~\eqref{loc:tauexit}, \eqref{loc:W2exit} will be established if we show that for any $C>1$,
\[\sup_{t\in [\ln N/C, C\ln N]} \mathcal W_2 \po \pi (\bX_t^c),\gamma_{0}\pf \underset{N \rightarrow \infty}\longrightarrow 0 \,.\]
 Consider an Ornstein-Uhlenbeck process $ \mathbf{Z}=(Z^i,\dots,Z^N)$ on $\R^{dN}$ with $\mathbf{Z}_0 \sim \gamma_0^{\otimes N}$ and 
 \begin{equation}
 \label{loc:Zti}
\dd Z_t^i = - \kappa Z_t^i \dd t + \sqrt{2}\dd B_t^i\,.
 \end{equation}
Then $Z_t^{i,c}:= Z_t^i-\bar Z_t$ solves the same equation as $X_t^{i,c}$. Substracting the two equations gives
\[|Z_t^{i,c} - X_t^{i,c}| = e^{-\kappa t }|Z_0^{i,c} - X_0^{i,c}| \,, \]
and we bound
\begin{align*}
 \mathcal W_2 \po \pi (\bX_t^c),\gamma_{0}\pf  &\leqslant \mathcal W_2 \po \pi (\bX_t^c), \pi (\mathbf{Z}_t^c)\pf  + \mathcal W_2 \po \pi (\mathbf{Z}_t^c),\gamma_{0}\pf\\
 &\leqslant  \frac{e^{-\kappa t}}{\sqrt N} |\bX_0^c - \mathbf{Z}_0^c| + \mathcal W_2 \po \pi (\mathbf{Z}_t^c),\gamma_{0}\pf\\
  &\leqslant  \frac{e^{-\kappa t}}{\sqrt N} |\bX_0 - \mathbf{Z}_0| + \mathcal W_2 \po \pi (\mathbf{Z}_t^c),\gamma_{0}\pf\,,
\end{align*}
where we used that $\bx^c$ is the orthogonal projection of $\bx$ on the orthogonal of $(1,\dots,1)$. By assumption and by the law of large numbers, $M:= \sup_{N} |\bX_0 - \mathbf{Z}_0| /\sqrt{N}$ is almost surely finite, and thus we bound
\[\sup_{t\in [\ln N/C, C\ln N]} \mathcal W_2 \po \pi (\bX_t^c),\gamma_{0}\pf \leqslant M e^{-\frac{\kappa}C \ln N} + \sup_{t\in [\ln N/C, C\ln N]} \mathcal W_2 \po \pi (\mathbf{Z}_t^c),\gamma_{0}\pf \,.\]
It only remains to prove that the second term almost surely vanishes.

For $\bx\in\R^{dN}$,  by considering an optimal coupling between $\pi_{\bx}$ and $\gamma_0$, we can find a random variable $Y\sim \gamma_0$ such that
\[\mathcal W_2^2(\pi_{\bx},\gamma_0) = \mathbb E\po \frac1N \sum_{i=1}^N |x_i-Y|^2\pf\,.\]
Then, writing $x^{c}_i=x_i - \bar x$ and $\bx^c=(x^c_i)_{i\in\cco 1,N\ccf}$,
\[\mathcal W_2^2(\pi_{\mathbf{x}^c},\gamma_0) \leqslant  \mathbb E\po \frac1N \sum_{i=1}^N |x^{c}_i-Y|^2\pf \leqslant 2|\bar x|^2 + 2 \mathbb E\po \frac1N \sum_{i=1}^N |x_i-Y|^2\pf \leqslant 4 \mathcal W_2^2(\pi_{\bx},\gamma_0)\,. \]
Applying this with $\bx= \mathbf Z_t$, we see that we will have concluded the proof if we show that, for any $C>1$,
\begin{equation}
\label{loc:as}
\sup_{t\leqslant C \ln N} \mathcal W_2(\pi_{\mathbf{Z_t}},\gamma_0)\underset{N \rightarrow \infty}\longrightarrow 0 \quad \text{almost surely}\,. 
\end{equation}
This is the content of Lemma~\ref{lem:OU} below.

\end{proof}

\begin{lemma}
\label{lem:OU}
Let $\mathbf{Z}=(Z^1,\dots,Z^N)$ be i.i.d. stationary Ornstein-Uhlenbeck processes, solving~\eqref{loc:Zti} with $\mathbf{Z} \sim \gamma_0^{\otimes N}$. Then, for all $C>1$, \eqref{loc:as} holds.
\end{lemma}

\begin{proof}
Consider $\varepsilon_N,\delta_N\in(0,1]$, to be chosen later on, and write $t_k = k \delta_N$ for $k\in\N$. By triangular inequality,
\[\sup_{t\leqslant C \ln N} \mathcal W_2^2(\pi_{\mathbf{Z_t}},\gamma_0) \leqslant \sup_{k\leqslant C \ln N/\delta_N} \co  \mathcal W_2(\pi_{\mathbf{Z_{t_k}}},\gamma_0) + \sup_{s\in[0,\delta_n]} \mathcal W_2(\pi_{\mathbf{Z}_{t_k}},\pi_{\mathbf{Z}_{t_k+s}}) \cf \,.\]
Consider the events
\[
A_{k,N} = \left\{  \mathcal W_2(\pi_{\mathbf{Z_{t_k}}},\gamma_0)  \geqslant \varepsilon_N \right\} \,,\qquad
B_{k,N} = \left\{  \sup_{s\in[0,\delta_n]} \mathcal W_2(\pi_{\mathbf{Z}_{t_k}},\pi_{\mathbf{Z}_{t_k+s}})  \geqslant \varepsilon_N \right\} \,.
\]
Then 
\begin{align}
\mathbb P \po \sup_{t\leqslant C \ln N} \mathcal W_2^2(\pi_{\mathbf{Z_t}},\gamma_0)  \geqslant 2\varepsilon_N \pf & \leqslant \sum_{k\leqslant C \ln N /\delta_N} \co \mathbb P(A_{k,N}) + \mathbb P\po A_{k,N}^c \cap B_{k,N} \pf  \cf  \nonumber \\
& \leqslant \frac{C\ln N}{\delta_N} \co \mathbb P(A_{0,N}) + \mathbb P(A_{0,N}^c \cap B_{0,N}) \cf\,,\label{loc:clnN} 
\end{align}
using that the process is stationary. First, since $\gamma_0$ has Gaussian moments, \cite[Theorem 2]{FournierGuillin} gives
\begin{equation}
\label{locA0N}
\mathbb P(A_{0,N}) \leqslant  C_1 e^{-c_1 N^\alpha \varepsilon_N^\beta }
\end{equation}
for come constants $C_1,c_1,\alpha,\beta>0$ (depending on $d$ and $\kappa$). Second, for $s\geqslant 0$, we bound
\begin{align*}
\mathcal W_2 (\pi_{\mathbf{Z}_{0}},\pi_{\mathbf{Z}_{s}}) & \leqslant \frac1{\sqrt{N}} |\mathbf{Z}_0 - \mathbf{Z}_{s}| \\
& \leqslant \frac{\kappa}{\sqrt{N}}\po s |\mathbf{Z}_0| + \int_0^{s} |\mathbf{Z}_0 - \mathbf{Z}_{u}|\dd u \pf + \frac{1}{\sqrt{N}}|\mathbf{B}_{s}|  \,.
\end{align*}
Under the event $A_{0,N}^c$,
\[\frac1N |\mathbf{Z}_0|^2  = \mathcal W_2^2 (\delta_0,\pi_{\mathbf{Z}_{s}}) \leqslant 2 \mathcal W_2^2 (\delta_0,\gamma_0) + 2 \varepsilon_N \leqslant \frac{2d}{\kappa}+2 \,. \]
Besides, writing $R_i = \sup_{s\leqslant \delta_N} |B_s^i|^2 / \delta_N$, we have for all $s\leqslant \delta_N$
\[\frac{1}{N}|\mathbf{B}_{s}|^2 \leqslant \frac{\delta_N}N \sum_{i=1}^N R_i\,.\]
By Grönwall lemma, we obtain that, under $A_{0,N}^c$, 
\[\sup_{s\leqslant\delta_N} \mathcal W_2 (\pi_{\mathbf{Z}_{0}},\pi_{\mathbf{Z}_{s}}) \leqslant  C' \po \delta_N + \sqrt{\frac{\delta_N}N \sum_{i=1}^N R_i}\pf \,,\]
for some constant $C'>0$ independent from $N$. Assuming that $C' \delta_N \leqslant \varepsilon_N/2$, we get that 
\[\mathbb P(A_{0,N}^c \cap B_{0,N})  \leqslant \mathbb P \po \sqrt{\frac{\delta_N}N \sum_{i=1}^N R_i}\ \geqslant \frac{\varepsilon_N}{2C'}\pf \,. \]
Now, thanks to the scaling properties of the Brownian motions, the variables $R_i$ are i.i.d. copies of $R=\sup_{s\leqslant 1} |B_s|$, whose law is independent from $N$ and  has exponential moments of all orders. Chernoff inequalities gives
\[\mathbb P(A_{0,N}^c \cap B_{0,N})   \leqslant \exp\po \frac{- N\varepsilon_N^2}{4\delta_N(C')^2})\pf \po \mathbb E  \po e^R\pf\pf ^N  \,. \]
Assuming that  $\delta_N = o(\varepsilon_N^2)$ as $N\rightarrow \infty$, this gives
\[\mathbb P(A_{0,N}^c \cap B_{0,N})   \leqslant C_2 e^{-c_2 N\varepsilon_N^2/\delta_N} \]
for some $C_2,c_2>0$ independent from $N$. Plugging this and~\eqref{locA0N} in~\eqref{loc:clnN} gives
\[\mathbb P \po \sup_{t\leqslant C \ln N} \mathcal W_2^2(\pi_{\mathbf{Z_t}},\gamma_0)  \geqslant 2\varepsilon_N \pf  \leqslant \frac{C\ln N}{\delta_N} \co  C_1 e^{-c_1 N^\alpha \varepsilon_N^\beta } + C_2 e^{-c_2 N\varepsilon_N^2/\delta_N} \cf \,. \]
Conclusion follows from the Borel-Cantelli lemma,  choosing $\varepsilon_N = N^{-\alpha/(2\beta)}$ and $\delta_N = \varepsilon_N^3$  for $N$ large enough.

\end{proof}

\section{Practical local estimate for particle systems}\label{sec:practical_conclusion}
 
 This short section aims at explaining how local functional inequalities as in Section~\ref{subsec:local_ineq} can be combined with controls on metastable exit times as in Section~\ref{sec:exittimes} to get a quantitative control of the fast local approximate convergence of particle systems to minimizers of the free energy.
 
 \begin{proposition}\label{prop:localBoundParticles}
 Let $\rho_*  = \gamma_{m_*}\in \mathcal P_2(\R^d)$ be an isolated local minimizer of $\mathcal F$. Assume that the equivalent conditions in Proposition~\ref{prop:equivLocale} are satisfied. Let $C_0>0$. Then, there exist $\lambda,\delta,a>0$ (independent from $C_0$) and $C>0$ such that for any  $N\geqslant 1$, any interchangeable initial distribution $\rho_0^N \in \mathcal P_2(\R^{dN})$ with finite fourth moment and such that $\mathbb P_{\rho_0^N}(\bar{X}_0 \notin \mathcal{B}(m_*,\delta)) \leqslant C_0 N^{-2}$, denoting by $\rho_t^N$ the law of the particle system $\bX_t$ solving~\eqref{eq:particles} with initial distribution $\bX_0 \sim \rho_0^{ N}$, and  any $t\in[1,e^{aN}]$,
 \begin{equation}
 \label{loc:eqH}
\mathcal H \po \rho_t^N | \rho_*^{\otimes N}\pf \leqslant C \po  e^{-\lambda t} \mathcal W_2^2 \po  \rho_{0}^N ,\rho_*^{\otimes N}\pf + 1  + \sqrt{\mathbb E \po |X_0^1|^4 \pf} \pf\,.
 \end{equation}
 \end{proposition}
 
 The practical interest of a bound of the form~\eqref{loc:eqH} in terms of sampling $\rho_*$ is discussed in the remarks after \cite[Theorem 3]{MonmarcheMetastable}, to which we refer the interested reader \rev{for details. Assuming that $\rho_0^N = \rho_0^{\otimes N}$ for some $\rho_0$, the right hand side of~\eqref{loc:eqH} is $\mathcal O(e^{-\lambda t} N + 1)$, corresponding to the sum of a long-time convergence error (vanishing at rate independent from $N$) and a propagation of chaos error (independent from $t$). Even in the nicest situation ($\mathcal F$ convex with unique global minimiser $\rho_*$, uniform-in-time propagation of chaos) the bound of order $N e^{-\lambda t} + 1$ is optimal, since at equilibrium $\mathcal H(\mu_\infty^N|\rho_*^{\otimes N})$ is of order one in this case}. We emphasize that such a bound cannot hold uniformly in time when $\mathcal F$ has several isolated critical points,  and that a key point of Proposition~\ref{prop:localBoundParticles} is that it is still possible to get it over an exponentially long timescale, way beyond  the simulation time in most applications. \rev{Indeed, if there is a global minimiser of $\mathcal F$ different from $\rho_*$, then  the one-particle marginal distribution $\mu_\infty^{1,N}$ of $\mu_\infty^N$, which is the long-time limit of $\rho_t^N$, does not converge to $\rho_*$ as $N\to\infty$, and thus $\liminf_{t\to\infty} \mathcal H(\rho_t^N|\rho_*^{\otimes N}) \geqslant \mathcal H(\mu_\infty^N|\rho_*^{\otimes N}) \geqslant N \mathcal H(\mu_\infty^{1,N}|\rho_*) \geqslant cN$ for some $c>0$. The best uniform-in-time upper bound on $\mathcal H(\rho_t^N|\rho_*^{\otimes N})$ is thus $\mathcal O(N)$, which doesn't describe any convergence: this is the same order as $\mathcal H(\rho_0^{\otimes N}|\rho_*^{\otimes N})$.}
 
 \begin{proof}
 We follow the proof of \cite[Theorem 2]{MonmarcheMetastable}, on which we rely for intermediary results and to which we refer for details. In particular, Assumption~\ref{assu:general} implies~\cite[Assumption 2]{MonmarcheMetastable}. Let $\widetilde{\mathcal E}$ be as in Proposition~\ref{prop:equivLocale} (and Remark~\ref{rem:tE}), associated to a strongly convex function $\tilde V\in \mathcal C^2(\R^d)$ coinciding with $V$ over $\mathcal B(m_*,\delta_0)$ for some $\delta_0>0$ and such that the associated Gibbs measure $\tilde\mu_\infty^N$ (given by~\eqref{eq:GIbbsdef} with potential $\tilde U_N(\bx) = N\widetilde{\mathcal E}(\pi_{\bx})$) satisfies a uniform LSI. We can assume that $\delta_0$ is small enough so that $m_*$ is the unique critical point of $V_\kappa$ over $\mathcal B(m_*,\delta_0)$. Since $\tilde V$ and $V$ coincide on $\mathcal B(m_*,\delta_0)$, $m_*$ is the global minimizer of $\tilde V_\kappa:x\mapsto \tilde V(x) + \frac{\kappa}{2}|x|^2$.
 
  Let $\bY$ be the particle system solving~\eqref{eq:particulesEDS2} with the modified potential $\tilde U_N$, driven by the same Brownian motion $\mathbf{B}$ as $\bX$, initialized with $\bY_0=\bX_0$. By design, $\bX_t=\bY_t$ for all $t \leqslant \tau:= \inf\{s\geqslant 0,\ \bar Y_t \notin \mathcal B(m_*,\delta_0)\}$. Denote by   $\tilde \rho_t^N$ the law of $\bY_t$.
  
  In the following we write $a(N,t)  \lesssim b(N,t)$ if there exists $C>0$, uniform over the initial distributions $\rho_0^N$ considered in the proposition, such that for all $N\geqslant 1$ and $t\geqslant 0$, $a(N,t) \leqslant C b(N,t)$.
 
  From \cite[Proposition 23]{MonmarcheReygner}
 \begin{align}
\mathcal H \po \rho_{t+1}^N |\rho_*^{\otimes N} \pf &\lesssim  \mathcal W_2^2 \po \rho_{t}^N ,\rho_*^{\otimes N} \pf +1 \nonumber\\
&  \lesssim   \mathcal W_2^2 \po \rho_{t}^N ,\tilde \rho_{t}^{ N} \pf + \mathcal W_2^2 \po \tilde \rho_{t}^N ,\rho_*^{\otimes N} \pf  +1\,.\label{loc:H12}
\end{align}
Using the Talagrand inequality satisfied by $\rho_*^{\otimes N}$ and \cite[Theorem 10]{MonmarcheMetastable}, there exists  $\lambda>0$ such that 
\begin{equation}
\label{loc:H13}
\mathcal W_2^2 \po \tilde \rho_{t}^N ,\rho_*^{\otimes N}\pf   \lesssim  e^{-\lambda t} \mathcal W_2^2 \po  \rho_{0}^N ,\rho_*^{\otimes N}\pf +1 \,.
\end{equation}
Besides, by interchangeability of the particles,
\begin{multline}
\label{loc:multi}
\mathcal W_2^2 \po \rho_{t}^N ,\tilde \rho_{t}^{ N} \pf \leqslant \mathbb E \po |\bX_t - \bY_t|^2 \pf = N \mathbb E \po |X_t^1-Y_t^1|^2 \1_{\bX_t\neq \bY_t} \pf  \\
\leqslant N \sqrt{\mathbb P \po \bX_t\neq \bY_t \pf \mathbb E \po |X_t^1|^4 + |Y_t^1|^4 \pf   } \,. 
\end{multline}
By standard Lyapunov arguments (as in \cite[Appendix B]{MonmarcheMetastable}), it is classical under Assumption~\ref{assu:general} to get that
\[ \mathbb E \po |X_t^1|^4 + |Y_t^1|^4 \pf \lesssim  1+ \mathbb E \po |X_0^1|^4 \pf\,.\]
Let $t_N = \mathbb E_{m_*}(\tau)$ (the $m_*$ in subscript meaning that we consider $\bar Y_0=m_*$). Since $\tilde V_\kappa$ is strongly convex with global minimizer at $m_*$, by classical results on exit times (as in Section~\ref{sec:exittimes}, namely \cite[Theorem 6.2]{freidlin1998random}, \cite[Theorem 3.2]{bovier2004metastability},  \cite[Theorem 1]{brassesco1998couplings}  and   \cite[Theorem 5.3]{Locherbach}), there exists $a>0$ such that $e^{aN} \lesssim t_N$ and, for $t \leqslant t_N / \ln N^2$,
\[\sup_{y\in \mathcal B(m_*,\delta_0/2)} \mathbb P_y (\tau \leqslant t) \leqslant 1 - e^{-t/t_N} + \sup_{y\in \mathcal B(m_*,\delta_0/2)} \sup_{s\geqslant 0} |\mathbb  P_y (\tau/t_N > s) - e^{-s}| \lesssim N^{-2}\,, \]
uniformly over initial conditions $\bar Y_0 \in \mathcal B(m_*,\delta_0/2)$. Hence, 
\[\mathbb P \po \bX_t\neq \bY_t \pf  \leqslant \mathbb P_{\rho_0^N} \po \bar X_0 \notin \mathcal B(m_*,\delta_0/2)\pf  + \sup_{y\in \mathcal B(m_*,\delta_0/2)} \mathbb P_y (\tau \leqslant t) \lesssim N^{-2}\,.  \]  
Using this in~\eqref{loc:multi} and combining the latter with~\eqref{loc:H13} in~\eqref{loc:H12} concludes the proof.
 
 \end{proof}
 
 \section{Application to a spin model at critical temperature}\label{sec:CurieWeiss}
 
By contrast to most of the rest of this work, this section is not devoted to the toy model~\eqref{eq:F}, but to a more classical mean-field system, namely the double-well Curie-Weiss model on $\R$. This model is well known to undergo a phase transition at some critical temperature. Our goal is to show that the ideas introduced in Sections~\ref{sec:degenPL} and \ref{sec:degenCoer} are useful to study this model at the critical temperature.

 \rev{The difference  is that the invariant measure is not anymore, up to an orthogonal transformation, the tensor product of the marginal law of the barycentre and of the centered system, as in~\eqref{eq:decomposeorthonu}, which was heavily used to study the toy models, in particular in Sections~\ref{sec:degenPL} and \ref{sec:degenCoer}. A key argument to solve this  issue is the two-scale decomposition~\eqref{loc:decompose} of the invariant measure, which can be interpreted as follows: conditionally to some macroscopic variable which is a proxy of the barycentre, the particles are i.i.d., corresponding somehow to the $d(N-1)$-dimensional Gaussian part in~\eqref{eq:decomposeorthonu}. Thus, the analysis boils down to the study of the marginal distribution of this macroscopic variable, as $\nu_\infty^N$ in~\eqref{eq:decomposeorthonu}, to which we can directly apply the degenerate low-temperature results from Section~\ref{sec:degenPL}. Then, it remains to see that the proof of Lemma~\ref{lem:LSIdegenerate} can be adapted by generalising the tensorisation argument of Lemma~\ref{lem:tensor} to a non-independent marginal/conditional decomposition.   }

\rev{We focus on the symmetric double-well Curie-Weiss model for clarity, but the two-scale approach has been shown to work in more general settings in~\cite{Dagallier} and the arguments in the present section readily extend to these settings\footnote{\rev{Actually, this is now written in the follow-up work~\cite{MonmarchePLLSI}}}.  }

\medskip

 Let us recall the free energy~\eqref{eq:F_CurieWeiss} for the Curie-Weiss model:
 \begin{align*}
\mathcal F_0(\mu)& = \frac{\kappa_0}4 \int_{\R^{2d}} |x-y|^2 \mu(x)\mu(y)\dd x\dd y  +   \int_{\R^d} V_0(x) \mu(x)\dd x + \sigma^2 \int_{\R^d} \mu(x)\ln(\mu(x))\dd x\\
&= -\frac{\kappa_0}2 |m_\mu|^2 +   \int_{\R^d} \co  V_0(x)+ \frac{\kappa_0}{2}|x|^2 \cf  \mu(x)\dd x + \sigma^2 \int_{\R^d} \mu(x)\ln(\mu(x))\dd x\,.
 \end{align*}
In this section we specify $d=1$ and $V_0(x)=\frac{x^4}{4}-\frac{x^2}{2}$ (symmetric double-well potential). Moreover, for consistency with the rest of this work, we rather work with $\mathcal F=\sigma^{-2} \mathcal F_0$ and set $\kappa=\sigma^{-2}\kappa_0$ and $V(x)=\sigma^{-2}\co  V_0(x) + \frac{\kappa_0}{2}|x|^2\cf$. From  \cite{Tugautdoublewell}, there exists a critical temperature $\sigma_c^2>0$ (with our parameters, for $\kappa_0=1$, $\sigma_c^2 \simeq 0.46$) such that, if $\sigma^2 \geqslant \sigma^2_c$, the mean-field PDE~\eqref{eq:EDP} has a unique stationary solution (i.e. $\mathcal F$ has a unique critical point), which is symmetric and is the global minimizer of $\mathcal F$, while if $\sigma^2 <\sigma_c^2$ there are exactly three stationary solutions, two with non-zero mean, one being the symmetric of the other, which are the global minimizers of $\mathcal F$, and a symmetric one, which is not. We use mutatis mutandis the same notations as in the toy model~\eqref{eq:F}, namely
\[\mathcal E(\mu) = -\frac{\kappa}{2} |m_\mu|^2 +    \int_{\R^d} V(x) \mu(x)\dd x\,,\qquad U_N(\bx) = N \mathcal E(\pi_{\bx})\,,\qquad \mu_\infty^N \propto \exp(-U_N),\]
\[\Gamma(\mu) \propto \exp\po - V(x) + \kappa m_\mu x  \pf\,,\qquad \gamma_m \propto \exp\po - V(x) + \kappa m x  \pf\,. \]
 The critical points of the free energy are characterized by their means. Specifically, they are exactly the measures $\gamma_m$ with $m$ the fixed  points of the function $f$ given by 
 \begin{equation}\label{eq:fCW}
 f(m) = \int_{\R} x \gamma_m(x)\,,
 \end{equation}
 see Figure~\ref{fig:f}.  When $\sigma^2=\sigma_c^2$, $0$ is the unique fixed point of $f$.  We write $\rho_*$ the unique stationary solution at temperature $\sigma_c^2$, which is simply $\rho_* \propto \exp\po - V\pf$.

\begin{figure}
\begin{center}
\includegraphics[scale=0.5]{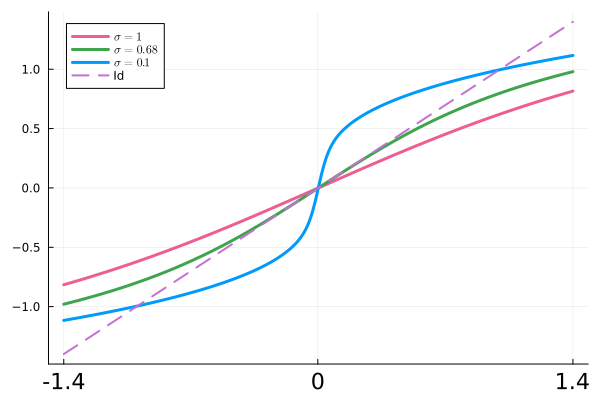}
\caption{Graph of $f$ given in~\eqref{eq:fCW} at three temperatures $\sigma=0.1$ (blue) $\sigma=0.68 \simeq \sigma_c$ (green) and $\sigma=1$ (magenta). At $\sigma=\sigma_c$ (resp. $\sigma<\sigma_c$, resp. $\sigma>\sigma_c$), $f'(0)=1$ (resp. $f'(0)>1$, resp. $f'(0)<1$).  }\label{fig:f}
\end{center}
\end{figure}

The main results of this section are gathered here.

\begin{theorem}\label{thm:CurieWeiss}
For the symmetric double-well Curie–Weiss model described here, at critical temperature $\sigma^2=\sigma_c^2$, there exists $C>0$ such that, denoting $\Theta(r) = C \max(r,r^{2/3})$ and $\Phi(r)= \min(r,r^2)/C$, the following holds.
\begin{enumerate}
\item The free energy $\mathcal F$ satisfies the Łojasiewicz inequality
\begin{equation}
\label{loc:CW1}
\forall \mu \in \mathcal P_2(\R^d)\,,\qquad \mathcal F(\mu) - \mathcal F(\rho_*) \leqslant \Theta\po \mathcal I\po \mu|\Gamma(\mu)\pf\pf \,,
\end{equation}
 and the non-linear Talagrand inequality
\begin{equation}
\label{loc:CW2}
\forall \mu \in \mathcal P_2(\R^d)\,,\qquad \Phi\po \mathcal W_2^2(\mu,\rho*)\pf  \leqslant \mathcal F(\mu) - \mathcal F(\rho_*)  \,.
\end{equation}
\item For all $N\geqslant 1$, writing $C_{LS}(\mu_\infty^N)$ the optimal log-Sobolev constant $\mu_\infty^N$,
\begin{equation}
\label{loc:CW3}
 \frac{\sqrt{N}}{C} \leqslant C_{LS}(\mu_\infty^N) \leqslant C \sqrt{N}\,.
\end{equation}
\item  For all $N\geqslant 1$,
\begin{equation}
\label{loc:CW4}
\forall \mu^N \ll \mu_\infty^N\,,\qquad \frac1N \mathcal H\po \mu^N |\mu_\infty^N \pf \leqslant \Theta\po \frac1N \mathcal I\po \mu^N |\mu_\infty^N \pf\pf\,.
\end{equation}
and 
\begin{equation}
\label{loc:CW5}
\forall \mu^N \ll \mu_\infty^N\,,\qquad \Phi\po \frac1{ N} \mathcal W_2^2\po \mu^N ,\mu_\infty^N \pf\pf  \leqslant \frac1N \mathcal H\po \mu^N |\mu_\infty^N \pf\,.
\end{equation}

\end{enumerate}
\end{theorem}

\begin{remark}
The transport inequality~\eqref{loc:CW2}, where $\Phi(r)$ behaves like $r^2$ for small values, gives some insight on why the strategy in \cite[Proposition 15]{MonmarcheReygner} only gave a sub-optimal Łojasiewicz inequality~\eqref{loc:CW1} with $\Theta(r)$ scaling like $r^{1/3}$ instead of $r^{2/3}$. Indeed, the proof in \cite{MonmarcheReygner} relies on a first bound $\mathcal F(\mu)-\mathcal F(\rho_*) \leqslant \mathcal F\po \mu|\Gamma(\mu)\pf + C \mathcal W_2^2(\mu,\rho_*)$. However,~\eqref{loc:CW2} suggests that  $\mathcal W_2^2(\mu,\rho_*)$ may be of order $\sqrt{\mathcal F(\mu)-\mathcal F(\rho_*)}$ for small values, hence the loss of exponent.
\end{remark}
 
 \begin{proof}
  To establish this result, we first decompose $\mu_\infty^N$ as in \cite{bauerschmidt2019very}, namely
  \begin{equation}
  \label{loc:decompose}
\mu_\infty^N(\bx)  = \int_{\R^d} \mu_\xi^{\otimes N}(\bx) \nu_N(\dd \xi)
  \end{equation}
where we used the Laplace transform of the Gaussian distribution and introduced
\[\mu_{\xi}(x) = \frac{1}{Z(\xi)}\exp \po -   V(x) +x\xi    \pf\,,\qquad Z(\xi) = \int_{\R^{d}} \exp \po -   V(x) +   x  \xi  \pf\dd x\]
and 
\begin{equation}
\label{loc:omega}
\nu_N(\xi) \propto  \exp\po - N \omega(\xi)  \pf\,,\qquad \omega(\xi) = \frac{|\xi|^2}{2\kappa} - \ln Z(\xi)\,.  
\end{equation}
 Notice that
 \begin{equation}
 \label{loc:omega'}
 \omega'(\xi) = \frac{\xi}{\kappa} - f\po \frac{\xi}{\kappa}\pf\,,
 \end{equation}
 with $f$ given in~\eqref{eq:fCW} and plotted in Figure~\ref{fig:f} (green curve for $\sigma = \sigma_c$). As seen from \cite{Tugautdoublewell} and \cite[Lemma A.2]{Minibatch}, at critical temperature, $f$ is strictly increasing with $f'(m)\in(0,1)$ for all $m\neq 0$, $f'(0)=1$, $f(m)/|m|\rightarrow 0$ as $|m|\rightarrow\infty$, $f''(0)=0$ and $f^{(3)}(0)<0$. As a consequence, $\omega''$ is positive everywhere except at $0$ where $\omega^{(3)}(0)=0$ and $\omega^{(4)}(0)>0$. This shows that $\omega$ satisfies the conditions of Propositions~\ref{prop:LSIdegen} and \ref{prop:LSIdegenNnu} with $\beta=4$. Applying these results shows that there exists $C>0$ such that for all $N\geqslant 1$, $\nu_N$ satisfies 
\begin{equation}
\label{loc:LSInuNCW}
\forall \mu \ll \nu_\infty^N\,,\qquad  \frac1N \mathcal H\po \mu|\nu_N\pf \leqslant \Theta_N \po \frac1{N^2} \mathcal I \po \mu|\nu_N\pf\pf \,,
\end{equation}
with both $\Theta_N= C\max(r, r^{2/3})$ and  $\Theta_N(r)= \sqrt{N} C r$.

For any $\mu \ll \mu_\infty^N$, writing $\varphi^2 = \frac{\dd \mu}{\dd \mu_\infty^N}$,
 \[\mathcal H \po \mu | \mu_\infty^N\pf  = \int_{\R^d} \mathcal H\po \mu| \mu_\xi^{\otimes N}\pf  \nu_N (\dd \xi) + \mathcal H\po \Phi^2 \nu_N |\nu_N\pf  \]
 with $\Phi^2(\xi) = \int_{\R^{dN}} \varphi^2 \mu_\xi^{\otimes N}$. According to \cite[Lemma 5]{MonmarcheMetastable}, $\mu_\xi$ satisfies a LSI with a constant $c_0$ independent from $\xi$,  which, as in the proof of \cite[Theorem 4]{MonmarcheMetastable}, thanks to \cite[Proposition 2.2]{Ledoux}, implies that 
      \[|\na \Phi(\xi)|^2 \leqslant C_0N \int_{\R^{dN}} |\na \varphi|^2 \dd \mu_\xi^{\otimes N}\,,    \] 
for some $C_0>0$ independent from $\xi$ and $N$, from which $ \mathcal I \po \Phi^2 \nu_N  | \nu_N\pf \leqslant C_0N  \mathcal I \po \mu  | \mu_\infty^N\pf$.  Using the  LSI satisfied by $\mu_\xi^{\otimes N}$ (uniformly in $N$ and $\theta$) and~\eqref{loc:LSInuNCW} gives
 \begin{align*}
  \frac 1N \mathcal H \po \mu | \mu_\infty^N\pf  & \leqslant \frac{c_0}N \mathcal I \po \mu | \mu_\infty^N\pf  +  \Theta_N\po \frac1{N^2} \mathcal I \po \Phi^2 \nu_N | \nu_N\pf\pf  \,.
\\
 & \leqslant \frac{c_0}N \mathcal I \po \mu | \mu_\infty^N\pf    +  \Theta_N\po \frac{C_0}{N} \mathcal I \po \mu | \mu_\infty^N\pf  \pf\,,
 \end{align*}
 which concludes the proof of the upper-bound on $C_{LS}(\mu_\infty^N)$ in~\eqref{loc:CW3} and of the degenerate LSI~\eqref{loc:CW4}. Letting $N\rightarrow \infty$ yields the $\Theta$-ŁI~\eqref{loc:CW1} for $\mathcal F$ (as in the proof of $(i),(iii)\Rightarrow (ii)$ of Proposition~\ref{prop:equivalenceInegal}). The corresponding coercivity inequalities~\eqref{loc:CW2} and~\eqref{loc:CW5} are obtained thanks to Proposition~\ref{prop:degenPL->coer} (see Remarks~\ref{rem:general} and \ref{rem:orderCoer} and the proof of Corollary~\ref{cor:PL->degenT2}).
 
 To get the lower bound on $C_{LS}(\mu_\infty^N)$ in~\eqref{loc:CW3}, reasoning as in Proposition~\ref{prop:lowerboundCLSIdegenerate}, we only have to find a good test function in the Poincaré inequality. We take $f(\bx) = \bar x$, so that $\int_{\R^{dN}} |\na f|^2 \mu_\infty^N = N^{-1}$ and, the Gibbs measure being even, $\int_{\R^{dN}} f \mu_\infty^N = 0$. Then,
 \begin{align*}
 \int_{\R^{dN}} f^2 (\bx) \dd \mu_\infty^N &= \int_{\R^d} \int_{\R^{dN}} \bar x^2\mu_\xi^{\otimes N}(\dd \bx) \nu_N(\dd \xi) \\
 & \geqslant  \int_{\R^d}  \po \int_{\R^d} x \mu_\xi(x)\pf^2   \nu_N(\dd \xi) \ \geqslant  \  \int_{-1}^1  m_{\mu_\xi}^2   \nu_N(\dd \xi)\,.
 \end{align*}
 It is clear that $\sup_{|\xi|\leqslant 1} Z(\xi)<\infty $ and that, by symmetry, $m_{-\xi} = -m_{\xi}$. For $\xi \in[0,1]$ we bound
 \[
 Z(\xi) m_{\xi} =   2\int_0^\infty x e^{-V(x)} \sinh(\xi x) \dd x \geqslant 2 \xi \int_0^\infty x^2 e^{-V(x)} \dd x\,.
 \]
 As a consequence, there exists $c>0$ such that for all $\xi\in[-1,1]$, $|m_\xi|\geqslant c|\xi|$. We end up with
 \begin{equation}
\label{loc:varianceCW}
 \int_{\R^{dN}} f^2 (\bx) \dd \mu_\infty^N  \ \geqslant  \  c^2 \int_{-1}^1  |\xi|^2   \nu_N(\dd \xi) \geqslant c' N^{-1/2}\,,
 \end{equation}
 for some $c'>0$ following Laplace method as in~\eqref{loc:laplace} (using that $\omega(x)-\omega(0)$ scales like $|x|^4$ at zero). Conclusion follows by the fact $C_{LS}(\mu_\infty^N) \geqslant \int_{\R^{dN}} f^2 (\bx) \dd \mu_\infty^N / \int_{\R^{dN}} |\na f|^2 (\bx) \dd \mu_\infty^N$.
 \end{proof}
 
In the following we present a sample of consequences from Theorem~\ref{thm:CurieWeiss}. Several variations can also be obtained by using Pinsker's inequality, Wasserstein-to-relative entropy Harnack inequalities from~\cite{RocknerWang}, etc. Moreover, estimates on the marginal distributions of particles can be deduced from global estimates on the joint distributions since, for any interchangeable probability measures $\mu^N,\nu^N \in \mathcal P_2(\R^{dN})$ and any $\rho\in\mathcal P_2(\R^d)$, denoting by $\mu^{(k,N)}$ the marginal law of $(X_1,\dots,X_k)$ when $\bX=(X_1,\dots,X_N)\sim \mu^N$ (and similarly for $\nu^N$) then for all $k\in\cco 1,N\ccf$, classically,
\begin{equation}
\label{loc:marginal}
\mathcal W_2^2 \po \nu^{(k,N)},\mu^{(k,N)}\pf \leqslant \frac{k}{N} \mathcal W_2^2\po \nu^N,\mu^N\pf\qquad  \mathcal H \po \nu^{(k,N)}| \rho^{\otimes k}\pf \leqslant \frac{1}{\lfloor N/k\rfloor} \mathcal H  \po \nu^N|\rho^{\otimes N}\pf\,.
\end{equation}
Moreover, the convergence of the empirical distribution can also be quantified from global estimates since, for $\bX\sim \mu^N$,
\begin{equation}
\label{loc:empirical}
\mathbb E \po \mathcal W_2^2 \po \pi(\bX) , \rho\pf \pf \leqslant \frac{2}{N} \mathcal W_2^2 \po \mu^N,\rho^{\otimes N}\pf + 2 \mathbb E \po \mathcal W_2^2 \po \pi(\bY) , \rho\pf \pf \,,
\end{equation}
with $\bY \sim \rho^{\otimes N}$. Optimal convergence rates for the last term as $N\rightarrow \infty$ are given in \cite[Theorem 1]{FournierGuillin}, as soon as $\rho$ admits moments of order higher than $2$.

 \begin{corollary}\label{Cor:CW}
 Under the settings of Theorem~\ref{thm:CurieWeiss}, there exists $\lambda, C>0$ such that the following holds.
 \begin{enumerate}
 \item \emph{(Propagation of chaos at stationarity)} For all $N\geqslant 1$,
 \begin{equation}
\label{eq:propchaosStation2}
\frac{1}{C^2\sqrt{N}}  \leqslant  \frac1{CN} \mathcal W_2^2\po \rho_*^{\otimes N} ,\mu_\infty^N \pf    \leqslant  \frac1N \mathcal H\po \mu_\infty^N |\rho_*^{\otimes N} \pf \leqslant \frac{C}{\sqrt{N}}\,.
\end{equation}
and \rev{we have the equivalence}
 \begin{equation}
\label{eq:propchaosStation} 
 \mathcal H\po \rho_*^{\otimes N} |\mu_\infty^N \pf \underset{N\rightarrow\infty}\sim \frac14 \ln N \,.
\end{equation}
Moreover, for any $\mu \in\mathcal P_2(\R^{dN})$,
\begin{equation}
\label{eq:propchaosStation3}
\frac1C  \mathcal W_2^2  \po \mu,\rho_*^{\otimes N} \pf  \leqslant  \mathcal H \po \mu|\rho_*^{\otimes N} \pf \leqslant  \mathcal H \po \mu|\mu_\infty^N \pf + \kappa \mathcal W_2^2 \po \mu,\mu_\infty^N \pf + C \sqrt{N}  \,.
\end{equation}
  \item \emph{(Polynomial convergence rates)} For any $N\geqslant 1$ and $t\geqslant 1$, for any $\rho_0^N \in \mathcal P_2(\R^{dN})$, denoting by $\rho_t^N$ the law of $\bX_t$ solving~\eqref{eq:particulesEDS2} with $\bX_0 \sim \rho_0^N$,
  \begin{equation}
  \label{eq:PolynomRate}
  \mathcal H(\rho_t^N |\mu_\infty^N) \leqslant e^{-\lambda t} \min \po C \mathcal W_2^2 (\rho_0^N,\mu_\infty^N) , \mathcal H(\rho_0^N |\mu_\infty^N)\pf  + \frac{CN}{t^2} \,.
  \end{equation}
   \item \emph{(Polynomial mean-field decay)} For any initial condition $\rho_0 \in \mathcal P_2(\R^d)$ and $t\geqslant 1$, along the gradient flow~\eqref{eq:EDP}, 
 \begin{equation}
 \label{eq:ratePolyMF}
\mathcal W_2^2(\rho_t,\rho_*) + \mathcal H\po \rho_t|\rho_*\pf  + \mathcal F(\rho_t) - \mathcal F(\rho_*) \leqslant C e^{-\lambda t}    \mathcal W_2^2 (\rho_0,\rho_*)  + \frac{C}{t^2}\,. 
 \end{equation}
 \item \emph{(Uniform in time propagation of chaos and mean-field limit)} For all initial distribution $\rho_0 \in \mathcal P_2(\R^d)$, there exists $C_0>0$ such that  for all $N \geqslant 1$, $k\in\cco 1,N\ccf $ and $t \geqslant 0$, considering particles initialized with $\rho_0^N = \rho_0^{\otimes N}$ and $(\rho_t)_{t\geqslant 0}$ solving~\eqref{eq:EDP},
 \[\mathbb E \po \mathcal W_2^2 \po \pi(\bX_t) , \rho_t\pf \pf  + \frac1N \mathcal W_2^2\po \rho_t^{\otimes N}, \rho_t^{N}\pf \leqslant \frac{C_0 }{\ln N}\,.\]
 \item \emph{(Creation of chaos)} For all $N\geqslant 1$, $k\in\cco 1,N\ccf$ and any initial distribution $\rho_0^N \in \mathcal P_2(\R^{dN})$, considering particles initialized with $\rho_0^N$,
 \begin{equation}
 \label{loc:creation}
  \frac1N \mathcal H \po \rho_t^{N}|\rho_*^{\otimes N} \pf \leqslant \  \frac {C }{\sqrt{N}}
 \end{equation}
for all $t\geqslant C \max \{\sqrt{N},\ln \mathcal W_2^2(\rho_0^N,\mu_\infty^N)\}$. In particular, for these $t$,
\[\mathcal W_2^2\po \rho_t^{(2,N)}, \po \rho_t^{(1,N)}\pf^{\otimes 2} \pf \leqslant \frac{C}{\sqrt{N}}\,. \]

 \end{enumerate}
 \end{corollary}
 
 \begin{proof} In the proof we denote by $C$ various constants which are independent from $N,t$ and the initial distributions and may change from line to line.

 \emph{(Proof of item 1.)} Since $\rho_*\propto \exp(-V)$, using the decomposition~\ref{loc:decompose},
\begin{align*}
\mathcal I(\mu_\infty^N|\rho_*^{\otimes N})
& = \kappa^2 N  \int_{\R^{dN}}  \bar x^2 \mu_\infty^N(\dd \bx) \\
& = \kappa^2 N \int_{\R^{dN}}  \int_{\R^{dN}}  \bar x^2 \mu_{\xi}^{\otimes N}(\dd \bx) \nu_N(\dd \xi)\\
& \leqslant C + \kappa^2 N \int_{\R^{dN}}  \po \int_{\R^{d}}  x_ 1 \mu_{\xi}(\dd x_1) \pf^2 \nu_N(\dd \xi)\,,
\end{align*}
where we used the uniform-in-$\xi$ Poincaré inequality satisfied by $\mu_\xi$ and that $\bx \mapsto \bar x$ is $N^{-1/2}$-Lipschitz. Notice that $\int_{\R^{d}}  x_ 1 \mu_{\xi}(\dd x_1) = f(\xi)$ and recall that $f(0)=0$ and $|f'(m)|\leqslant 1$ for all $m\in\R$. As a consequence,
\[\mathcal I(\mu_\infty^N|\rho_*^{\otimes N}) \leqslant C + \kappa^2 N \int_{\R^{dN}} \xi^2  \nu_N(\dd \xi) = \underset{N\rightarrow \infty}{\mathcal O}(\sqrt{N})\,,\]
thanks to the Laplace method (see \eqref{loc:laplace} with $\beta=4$). Using the uniform-in-$N$ LSI and Talagrand inequality satisfied by $\rho_*^{\otimes N}$ gives 
\[\frac1C \mathcal W_2^2 \po \mu_\infty^N,\rho_*^{\otimes N} \pf \leqslant  \mathcal H\po \mu_\infty^N |\rho_*^{\otimes N} \pf \leqslant C  \mathcal I\po \mu_\infty^N |\rho_*^{\otimes N} \pf = \underset{N\rightarrow \infty}{\mathcal O}(\sqrt{N})\,.\]
To conclude the proof of~\eqref{eq:propchaosStation2}, it remains to lower bound the $\mathcal W_2$ distance as follows. Given $\pi$ a coupling between $\mu_\infty^N$ and $\rho_*^{\otimes N}$, 
\begin{align*}
\int_{\R^{dN}}\bar x^2 \mu_\infty^N(\dd \bx) &\leqslant 2 \int_{\R^{dN}}\bar x^2 \rho_*^{\otimes N} + 2 \int_{\R^{2dN}}|\bar x-\bar y|^2 \pi(\dd \bx  ,\dd \mathbf{y})  \\
& \leqslant \frac{2\mathrm{Var}(\rho_*)}{N} +  \frac{2}N \int_{\R^{2dN}}|\bx-\mathbf{y}|^2 \pi(\dd \bx  ,\dd \mathbf{y})  \,.
\end{align*}
Taking the infimum over all couplings and rearranging the terms leads to
\[\frac2N \mathcal W_2^2 \po \mu_\infty^N,\rho_*^{\otimes N}\pf \geqslant \int_{\R^{dN}}\bar x^2 \mu_\infty^N(\dd \bx) -\frac{2\mathrm{Var}(\rho_*)}{N}  \geqslant  \frac{c'}{\sqrt{N}} - \frac{2\mathrm{Var}(\rho_*)}{N}\,,\]
thanks to~\eqref{loc:varianceCW}, from which the first inequality in~\eqref{eq:propchaosStation2} holds for $N$ large enough.


We turn to the proof of~\eqref{eq:propchaosStation}, computing explicitly
\begin{align*}
\mathcal H \po \rho_*^{\otimes N } |\mu_\infty^N \pf & = \frac{\kappa N}2  \int_{\R^{dN}}  \bar x^2 \rho_*^{\otimes N }(\dd \bx) + \ln \int_{\R^{dN}} e^{\frac{\kappa}{2}N\bar x^2 } \rho_*^{\otimes N }(\dd \bx) \\
&= \frac{\kappa}{2}\mathrm{Var}(\rho_*) + \ln\co \frac{\sqrt{\kappa N}}{\sqrt{2\pi}} \int_{\R^d} e^{-N \po \omega(\xi)-\omega(0)\pf } \dd \xi\cf 
\end{align*}
where we reasoned as in the decomposition~\eqref{loc:decompose} and $\omega$ is given in~\eqref{loc:omega}. Using that, as already used in the proof of Theorem~\ref{thm:CurieWeiss}, $\omega(\xi)-\omega(0)$ is strongly convex far from $0$ and behaves like $\xi^4$ at zero, by Laplace method (as in~\eqref{loc:laplace}), we get that 
\[ \int_{\R^d} e^{-N \po \omega(\xi)-\omega(0)\pf } \dd \xi \underset{N\rightarrow \infty}\sim N^{-1/4} b_0 \]
for some $b_0>0$. Plugging this in the previous equality gives~\eqref{eq:propchaosStation}.

Finally, the first inequality in~\eqref{eq:propchaosStation3} is simply given by the Talagrand inequality. Then, 
 \begin{align*}
 \mathcal H \po \mu|\rho_*^{\otimes N} \pf &= \mathcal H \po \mu|\mu_\infty^N \pf + \int_{\R^{dN}} \mu \ln \frac{\mu_\infty^N}{\rho_*^{\otimes N}}\\
 & =  \mathcal H \po \mu|\mu_\infty^N \pf + \mathcal H \po \mu_\infty^N |\rho_*^{\otimes N}\pf  + \int_{\R^{dN}} (\mu - \mu_\infty^N)(\bx)  \ln \frac{e^{-N U_N(\bx)}}{e^{-\sum_{i=1}^N V(x_i)}} \dd \bx \\
 & =  \mathcal H \po \mu|\mu_\infty^N \pf + \mathcal H \po \mu_\infty^N |\rho_*^{\otimes N}\pf  + \frac\kappa 2 N \int_{\R^{dN}}  \bar x^2  (\mu - \mu_\infty^N)(\bx)   \dd \bx \,.
 \end{align*}
 Considering $\pi$ an optimal $\mathcal W_2$-coupling of $\mu$ and $\mu_\infty$,
 \begin{align*}
  \int_{\R^{dN}}  \bar x^2  (\mu - \mu_\infty^N)(\bx)   \dd \bx  &= \int_{\R^{dN}}  (\bar x^2-\bar y^2) \pi(\dd \bx,\dd \mathbf{y})\\
   & \leqslant \int_{\R^{dN}}  \co 2(\bar x-\bar y)^2 + \bar y^2\cf \pi(\dd \bx,\dd \mathbf{y})  \\
  & =  \frac{2}{N} \mathcal W_2^2(\mu,\mu_\infty^N)   + \int_{\R^{dN}} \bar x^2 \mu_\infty^N\,.
 \end{align*}
 The last term is of order $\sqrt{N}$, which concludes the proof of~\eqref{eq:propchaosStation3}.

\medskip

\emph{(Proof of item 2.)}  Denoting $H_N(t) = \frac1N\mathcal H(\rho_t^N|\mu_\infty^N)$ and $I_N(t) = \frac1N\mathcal I(\rho_t^N|\mu_\infty^N)$, we simply combine the  Łojasiewicz inequality~\eqref{loc:CW4} with the fact that $H_N'(t) = - I_N(t)$,  distinguishing whether $t\leqslant t_*:= \inf\{s\geqslant 0,\ H_N(s) \leqslant 1\}$ or not. For $t\leqslant t_*$, $H_N'(t) \leqslant - C' H_N(t)$ for some $C'>0$ independent from $N$. For $t\geqslant t_*$,   $H_N'(t) \leqslant -C' H_N^{3/2}(t)$ with $H_N(t_*)\leqslant 1$, which gives $H_N(t) \leqslant \min(1, C/(t-t_*)^2) \leqslant C'(t-t_*+1)^2$. Now, for $t\geqslant 2 t_*$, we bound $H_N(t) \leqslant H_N(t/2+t_*) \leqslant C'(t/2+1)^2$, and for $t\leqslant 2t_*$, we bound $H_N(t) \leqslant H_N(t/2) \leqslant e^{-t/(2C)}H_N(0)$. This shows that
\[  \mathcal H(\rho_t^N |\mu_\infty^N) \leqslant e^{-\frac{t}{C}}\mathcal H(\rho_t^N |\mu_\infty^N)  + \frac{CN}{(1+t)^2} \,.\]
To conclude, the proof of~\eqref{eq:PolynomRate}, it remains to use \cite[Corollary 1.2]{RocknerWang} to bound
\[\mathcal H(\rho_{1}^N |\mu_\infty^N) \leqslant C \mathcal W_2^2\po \rho_0^N|\mu_\infty^N\pf\]
for all $t\geqslant 0$. Combining this with the previous inequality gives, for $t\geqslant 1$,
\[  \mathcal H(\rho_t^N |\mu_\infty^N) \leqslant C e^{-\frac{t-1}{C}}\mathcal H(\rho_1^N |\mu_\infty^N)  + \frac{CN}{t^2}\leqslant  e^{-\frac{t-1}{C}}\mathcal W_2^2\po \rho_0^N|\mu_\infty^N\pf + \frac{CN}{t^2} \,,\]
establishing~\eqref{eq:PolynomRate}.

\medskip

 \emph{(Proof of item 3.)} This can be obtained by dividing~\eqref{eq:PolynomRate} by $N$ and letting $N\rightarrow \infty$. The fact that we can also add $\mathcal H(\rho_t|\rho_*)$ in the left-hand side of~\eqref{eq:ratePolyMF}  is obtained as in the proof of the second par of \cite[Proposition 1]{MonmarcheReygner}, using \cite[Lemma 3]{MonmarcheReygner}.
 
 \medskip

\emph{(Proof of item 4.)} By standard finite-time propagation of chaos results, for all $t\geqslant 0$,
\[\mathcal W_2^2\po \rho_t^{\otimes N}, \rho_t^{N}\pf \leqslant  e^{Ct} C_0\,,\]
for some $C_0$ depending on $\rho_0$ but not $N$. On the other hand, thanks to~\eqref{eq:propchaosStation3} and then~\eqref{eq:ratePolyMF} and~\eqref{eq:PolynomRate} with~\eqref{loc:CW5},
\begin{align*}
\mathcal W_2^2\po \rho_t^{\otimes N}, \rho_t^{N}\pf &\leqslant 2 \mathcal W_2^2\po \rho_t^{\otimes N}, \rho_*^{\otimes N}\pf + 2 \mathcal W_2^2\po \rho_*^{\otimes N}, \rho_t^{N}\pf \\
&\leqslant  2 N \mathcal W_2^2\po \rho_t, \rho_*\pf  + C  \co  \mathcal H \po \rho_t^N |\mu_\infty^N \pf +  \mathcal W_2^2 \po \rho_t^N,\mu_\infty^N \pf +  \sqrt{N} \cf  \\
&\leqslant    C_0' N  \co  \frac{1}{t} +  \frac1{\sqrt{N}} \cf  \,,
\end{align*}
for some $C_0'$. Combining these two bounds depending whether $t \gtrless \ln N/(2C)$ gives the result for $\frac1N \mathcal W_2^2(\rho_t^N,\rho_t^{\otimes N})$. Conclusion follows by using~\eqref{loc:empirical} with \cite[Theorem 1]{FournierGuillin}.

\medskip

\emph{(Proof of item 5.)} This is similar except that we directly use~\eqref{eq:propchaosStation3} (and then~\eqref{loc:CW5}) to bound
\begin{align*}
\frac1N \mathcal H \po \rho_t^N|\rho_*^{\otimes N} \pf &\leqslant  \frac1N \mathcal H \po \rho_t^N|\mu_\infty^N \pf + \frac\kappa N \mathcal W_2^2 \po \rho_t^N,\mu_\infty^N \pf + \frac C{\sqrt{N}}  \\
 & \leqslant C\co \frac1N \mathcal H \po \rho_t^N|\mu_\infty^N \pf + \sqrt{ \frac1N \mathcal H \po \rho_t^N|\mu_\infty^N \pf} + \frac1{\sqrt{N}}\cf \ \leqslant \  \frac C{\sqrt{N}}
\end{align*}
for all $t\geqslant C \max \{\sqrt{N},\ln \mathcal W_2^2(\rho_0^N,\mu_\infty^N)\}$ thanks to~\eqref{eq:PolynomRate}. This gives~\eqref{loc:creation} by the scaling property of the entropy. The last statement is then a direct consequence since, using the Talagrand inequality for $\rho_*^{\otimes 2}$, by triangular inequality,
\[\mathcal W_2^2\po \rho_t^{(2,N)}, \po \rho_t^{(1,N)}\pf^{\otimes 2} \pf \leqslant C  \mathcal H\po \rho_t^{(2,N)},\rho_*^{\otimes 2}\pf +   C  \mathcal H\po \rho_t^{(1,N)},\rho_*\pf \leqslant  \frac C{\sqrt{N}} \]
for $t$ as before.
 \end{proof}

\begin{remark}
For $\varepsilon>0$ and a given  initial condition $\rho_0^N \in \mathcal P_2(\R^{dN})$ (with $\mathcal H(\rho_t^N|\mu_\infty^N) \leqslant C_0 N$ for some constant $C_0>0$, which is for instance the case when $\rho_0^N = \rho_0^{\otimes N}$ for some smooth $\rho_0\in\mathcal P_2(\R^d)$), let $t_N(\varepsilon) = \inf\{s\geqslant 0,\ \mathcal H(\rho_t^N|\mu_\infty^N) \leqslant \varepsilon\}$. If we use the standard LSI  with~\eqref{loc:CW3}, from 
\[\mathcal H(\rho_t^N |\mu_\infty^N) \leqslant e^{-\frac{t}{C\sqrt{N}}}\mathcal H(\rho_t^N |\mu_\infty^N) \,,\]
we get that $t_N(\varepsilon) \leqslant C\sqrt{N} (\ln(C_0 N/\varepsilon))_+ $. Alternatively, if we use the degenerate Łojasiewicz inequality~\eqref{loc:CW4}, we obtain~\eqref{eq:PolynomRate} for some constant $C>0$, and then $t_N(\varepsilon) \leqslant \sqrt{2CN/\varepsilon}  +  C(\ln(2C_0 N/\varepsilon))_+  $. Hence, for a fixed $\varepsilon>0$, we do not gain much with the degenerate inequality, going from a bound of order $\sqrt{N}\ln N$ to $\sqrt{N}$ (and the dependency in $\varepsilon$ is much worse, going from $\ln\varepsilon^{-1}$ to $\varepsilon^{-1/2}$). 

However, if we are rather interested in the marginal distribution of a fixed number $k$ of particles or on the empirical distribution of the system, due to~\eqref{loc:marginal} and \eqref{loc:empirical}, we should rather consider $t_N( N\varepsilon)$ with  fixed $\varepsilon>0$.  Now, the bound on this obtained by the standard LSI is of order $\sqrt{N} \ln (\varepsilon^{-1})$, while the bound obtained with the degenerate inequality is of order $\varepsilon^{-1/2}$ and independent from $N$. We see that, although the degenerate inequality only asymptotically yields polynomial convergence rates instead of exponential ones,  in this regime it is much more informative.

\end{remark}

\rev{Finally, let us turn to the question of the sharpness of the exponents in Theorem~\ref{thm:CurieWeiss}, namely that, for small values of $r$, a function $\Theta$ satisfying either~\eqref{loc:CW1} or \eqref{loc:CW4} (for all $N\geqslant 1$) is at least of order $r^{2/3}$ and a function $\Phi$ satisfying either \eqref{loc:CW2} or \eqref{loc:CW5} (for all $N\geqslant 1$) is at most of order $r^2$ (concerning the standard log-Sobolev inequality for $\mu_\infty^N$, the two-sided bound~\eqref{loc:CW3} already shows that $\sqrt{N}$ is optimal). Actually, as explained in the proof of Theorem~\ref{thm:CurieWeiss}, the inequalities~\eqref{loc:CW4}  and~\eqref{loc:CW5} respectively imply~\eqref{loc:CW1} and~\eqref{loc:CW2} by letting $N\rightarrow \infty$ (applying them to $\mu^N = \mu^{\otimes N}$), so it is sufficient to prove that the exponents of $\Theta$ and $\Phi$ are sharp in the mean-field inequalities~\eqref{loc:CW1} and~\eqref{loc:CW2}. To get an lower bound (resp. upper bound) on the optimal $\Theta$ (resp. $\Phi$), it is sufficient to apply~\eqref{loc:CW1} (resp.~\eqref{loc:CW2}) to some specific $\mu$.  The computations gathered in the next lemma shows that the exponents $r^{2/3}$ and $r^2$ are saturated by taking $\mu = \gamma_m$ for $m$ small.  }

\rev{\begin{lemma}\label{lem:gammam}
For the symmetric double-well Curie–Weiss model  at critical temperature $\sigma^2=\sigma_c^2$, there exist $c_1,c_2>0$ such that, as $m\rightarrow 0$,
\begin{equation}
\label{eq:locmfm}
\liminf_{m\to 0}\frac1{m^2} \mathcal W_2^2(\gamma_m,\rho_*)\geqslant 1 \,,\qquad   \mathcal F(\gamma_m) - \mathcal F(\rho_*) \sim c_1 m^4 \,,\qquad \mathcal I\po \gamma_m|\Gamma(\gamma_m)\pf \sim c_2 m^6 \,.
\end{equation}
\end{lemma}}

\rev{Notice that~\eqref{loc:CW2} (where $\Phi(r)$ is of order $r^2$ as $r\to 0$) with~\eqref{eq:locmfm} shows that $ \mathcal W_2^2(\gamma_m,\rho_*) = \mathcal O(m^2)$ as $m\to0$, i.e. the order of $\mathcal W_2^2(\gamma_m,\rho_*)$ is exactly $m^2$.  }

\begin{proof}
\rev{The key here is the behavior at $0$ of $f$ given in~\eqref{eq:fCW}. Indeed, as stated for instance in \cite[Lemmas 3.2 and A.2]{Minibatch}, it is known that, at the critical temperature,
\begin{equation}
\label{loc:fm3}
f(m) = m - a m^3 + \underset{m\to 0} o(m^3)
\end{equation}
for some $a>0$, see Figure~\ref{fig:f}. As a side comment, due to~\eqref{loc:omega'}, this corresponds to the so-called renormalised potential $\omega$ behaving like $m^4$ at the origin, i.e. convex but degenerate.  Now, let us express or bound all the quantities in~\eqref{eq:locmfm} in terms of $f$. First,
\[\mathcal W_2(\gamma_m,\rho_*) \geqslant \left|\int_{\R} x \gamma_m(x)\dd x -\int_{\R} x \rho_*(x)\dd x \right| = |f(m)| \,. \]
The first part of~\eqref{eq:locmfm} thus follows from~\eqref{loc:fm3}. Second, by direct computations,
\[
\mathcal F(\gamma_m) - \mathcal F(\rho_*) = -\frac{\kappa}{2} f^2(m) + \kappa m f(m) - \ln \int_{\R} e^{\kappa mx} \rho_*(x)\dd x\,.
\]
The derivative in $m$ of the last term is $\kappa f(m)$. In other words,
\[
\mathcal F(\gamma_m) - \mathcal F(\rho_*) = -\frac{\kappa}{2} f^2(m) + \kappa m f(m) - \kappa \int_0^m f(r) \dd r \,.
\]
Plugging~\eqref{loc:fm3} here gives
\[
\mathcal F(\gamma_m) - \mathcal F(\rho_*) = \frac{\kappa a}{4}m^4 + \underset{m\to0}o(m^4)\,.
\]
Similarly, since $\Gamma(\gamma_m)= \gamma_{f(m)}$,
\[\mathcal I\po \gamma_m|\Gamma(\gamma_m)\pf = \kappa^2 |m-f(m)|^2 \underset{m\to0}\sim \kappa^2 a^2 m^6\,. \]
}
\end{proof}

\rev{Applying~\eqref{loc:CW1} and~\eqref{loc:CW2}  to $\mu = \gamma_m$ and using Lemma~\ref{lem:gammam} immediately gives the following.}

\rev{
\begin{corollary}\label{cor:sharp}
For the symmetric double-well Curie–Weiss at critical temperature $\sigma^2=\sigma_c^2$, there exists $c>0$ such that, for any functions $\Theta : \R_+ \rightarrow \R_+$ and $\Phi:\R_+\to\R_+$ such that~\eqref{loc:CW1} and~\eqref{loc:CW2} and all $r\in(0,1]$,
\[\Theta(r) \geqslant  c r^{2/3}\,,\qquad \Phi(r) \leqslant r^2/c\,.\]
\end{corollary}
}


\subsection*{Acknowledgments}

The research of PM is supported by the project CONVIVIALITY (ANR-23-CE40-0003) of the French National Research Agency. The author would like to thank Songbo Wang for fruitful discussions on his work~\cite{SongboToAppear} and Julien Reygner and Max Fathi for pointing out respectively the references~\cite{barret2015sharp,BerglundGentz,Berglund2} on Eyring-Kramers formula in infinite dimension and~\cite{bobkov2000brunn,bobkov2005entropy,barthe2008modified,fathi2015modified,bou2024note} on $p$-modified log-Sobolev inequalities. Some computations on the running example were done by Virgile Revon during a research internship. The total  amount of generative artificial intelligence tools involved in this work  is exactly zero.

\bibliographystyle{plain}
\bibliography{biblio}

\end{document}